\documentclass[11pt]{article}
\usepackage[utf8]{inputenc}
\usepackage{amsmath,amsfonts,latexsym, xspace,amsthm,graphicx, amssymb}
\usepackage{enumerate}
\usepackage[letterpaper,left=1in,right=1in,top=1in,bottom=1in]{geometry}
\usepackage{amsmath, tikz,amssymb,color,comment}
\usepackage{needspace}
\usepackage{authblk}

\newcommand{\Gnp}{G_{n,p}}
\newcommand{\Gnm}{G_{n,m}}
\newcommand{\Gminus}{G^{-}}
\newcommand{\Gplus}{G^+}

\usepackage{amsthm,mathrsfs,fancyhdr,subfigure}
\usepackage{graphicx} 
\usepackage[noadjust,sort]{cite}
\newcommand{\E}{{\mathbb E}}

\newcommand{\vol}{{\rm vol}}
\newcommand{\Bin}{{\rm Bin}}

\newcommand{\eps}{\varepsilon}

\renewcommand{\AA}{ \mathcal{A} }
\newcommand{\cB}{\mathcal{B}}
\newcommand{\cA}{\mathcal{A}}

\newcommand{\cT}{\mathcal{T}}

\newcommand{\cG}{\mathcal{G}}

\newcommand{\q}{q^*}

\newcommand{\sqA}{sq_{\cA}}
\newcommand{\sq}{sq^*}

\newcommand{\cq}{cq^*}
\newcommand{\cqA}{cq_{\cA}}

\newcommand{\tq}{\tilde{q}^*}
\newcommand{\lfqA}{l\!fq_{\cA}}
\newcommand{\lfq}{l\!fq^*}

\newcommand{\tp}{\tilde{p}}

\newcommand{\cyclo}{\mathrm{cyclo}}
\newcommand{\isol}{\mathrm{isol}}

\newtheorem{thm}{Theorem}
\newtheorem{defn}{Definition}

\newtheorem{clm}[thm]{Claim}

\newtheorem{rem}[thm]{Remark}
\newtheorem{lemma}[thm]{Lemma}
\newtheorem{prop}[thm]{Proposition}
\newtheorem{lem}[thm]{Lemma}
\newtheorem{question}[thm]{Question}
\newtheorem{example}[thm]{Example}

\numberwithin{equation}{section}

\newcommand{\pr}{{\mathbb P}}
\renewcommand{\E}{{\mathbb E}}
\newcommand{\var}{{\rm var}}

\renewcommand{\Bin}{{\rm Bin}}

\newcommand{\cH}{{\mathcal H}}

\renewcommand{\eps}{\varepsilon}

\newcommand{\kfourisolb}{
  \begin{tikzpicture}[baseline=-0.6ex,scale=0.25]
  \tikzstyle{vertex}=[circle,fill=black, minimum size=2pt,inner sep=1pt]
  \node[vertex] (v1) at (-0.5, 0.5){};
  \node[vertex] (v2) at (-0.5,-0.5){};
  \node[vertex] (v3) at (0.5, 0.5){};
  \node[vertex] (v4) at (0.5,-0.5){};
  \node[vertex] (v5) at (1.5,-0.0){};
  \draw (v1)--(v2)--(v3)--(v4)--(v1)--(v3) (v2)--(v4);
  \end{tikzpicture}}

\newcommand{\threepathb}{
  \begin{tikzpicture}[baseline=-0.6ex,scale=0.25]
  \tikzstyle{vertex}=[circle,fill=black, minimum size=2pt,inner sep=1pt]
  \node[vertex] (v1) at (-0.5, 0.5){};
  \node[vertex] (v2) at (-0.5,-0.5){};
  \node[vertex] (v3) at (0.5, 0.5){};
  \node[vertex] (v4) at (0.5,-0.5){};
  \draw (v1)--(v3);
  \draw (v1)--(v2);
  \draw (v2)--(v4);
  \end{tikzpicture}}

\newcommand{\kfour}{
  \begin{tikzpicture}[baseline=-0.6ex,scale=0.25]
  \tikzstyle{vertex}=[circle,fill=black, minimum size=2pt,inner sep=1pt]
  \node[vertex] (v1) at (-0.5, 0.5){};
  \node[vertex] (v2) at (-0.5,-0.5){};
  \node[vertex] (v3) at (0.5, 0.5){};
  \node[vertex] (v4) at (0.5,-0.5){};
  \draw (v1)--(v2)--(v3)--(v4)--(v1)--(v3) (v2)--(v4);
  \end{tikzpicture}}

  \newcommand{\copawc}{
  \begin{tikzpicture}[baseline=-0.6ex,scale=0.25]
  \tikzstyle{vertex}=[circle,fill=black, minimum size=2pt,inner sep=1pt]
  \node[vertex] (v1) at (-0.5, 0.5){};
  \node[vertex] (v2) at (-0.5,-0.5){};
  \node[vertex] (v3) at (0.5, 0.5){};
  \node[vertex] (v4) at (0.5,-0.5){};
  \draw (v3)--(v2)--(v4);
  \end{tikzpicture}}

\newcommand{\copawd}{
  \begin{tikzpicture}[baseline=-0.6ex,scale=0.25]
  \tikzstyle{vertex}=[circle,fill=black, minimum size=2pt,inner sep=1pt]
  \node[vertex] (v1) at (-0.5, 0.5){};
  \node[vertex] (v2) at (-0.5,-0.5){};
  \node[vertex] (v3) at (0.5, 0.5){};
  \node[vertex] (v4) at (0.5,-0.5){};
  \draw (v3)--(v1)--(v4);
  \end{tikzpicture}}

\newcommand{\copawe}{
  \begin{tikzpicture}[baseline=-0.6ex,scale=0.25]
  \tikzstyle{vertex}=[circle,fill=black, minimum size=2pt,inner sep=1pt]
  \node[vertex] (v1) at (-0.5, 0.5){};
  \node[vertex] (v2) at (-0.5,-0.5){};
  \node[vertex] (v3) at (0.5, 0.5){};
  \node[vertex] (v4) at (0.5,-0.5){};
  \draw (v1)--(v2)--(v4);
  \end{tikzpicture}}

\newcommand{\copawf}{
  \begin{tikzpicture}[baseline=-0.6ex,scale=0.25]
  \tikzstyle{vertex}=[circle,fill=black, minimum size=2pt,inner sep=1pt]
  \node[vertex] (v1) at (-0.5, 0.5){};
  \node[vertex] (v2) at (-0.5,-0.5){};
  \node[vertex] (v3) at (0.5, 0.5){};
  \node[vertex] (v4) at (0.5,-0.5){};
  \draw (v1)--(v2)--(v3);
  \end{tikzpicture}}

\newcommand{\paw}{
  \begin{tikzpicture}[baseline=-0.6ex,scale=0.25]
  \tikzstyle{vertex}=[circle,fill=black, minimum size=2pt,inner sep=1pt]
  \node[vertex] (v1) at (-0.5, 0.5){};
  \node[vertex] (v2) at (-0.5,-0.5){};
  \node[vertex] (v3) at (0.5, 0.5){};
  \node[vertex] (v4) at (0.5,-0.5){};
  \draw (v1)--(v2)--(v3)--(v1)--(v4);
  \end{tikzpicture}}

\newcommand{\pawb}{
  \begin{tikzpicture}[baseline=-0.6ex,scale=0.25]
  \tikzstyle{vertex}=[circle,fill=black, minimum size=2pt,inner sep=1pt]
  \node[vertex] (v1) at (-0.5, 0.5){};
  \node[vertex] (v2) at (-0.5,-0.5){};
  \node[vertex] (v3) at (0.5, 0.5){};
  \node[vertex] (v4) at (0.5,-0.5){};
  \draw (v2)--(v1)--(v3)--(v4)--(v1);
  \end{tikzpicture}}

\newcommand{\pawc}{
  \begin{tikzpicture}[baseline=-0.6ex,scale=0.25]
  \tikzstyle{vertex}=[circle,fill=black, minimum size=2pt,inner sep=1pt]
  \node[vertex] (v1) at (-0.5, 0.5){};
  \node[vertex] (v2) at (-0.5,-0.5){};
  \node[vertex] (v3) at (0.5, 0.5){};
  \node[vertex] (v4) at (0.5,-0.5){};
  \draw (v1)--(v2)--(v3)--(v4)--(v2);
  \end{tikzpicture}}

\newcommand{\pawd}{
  \begin{tikzpicture}[baseline=-0.6ex,scale=0.25]
  \tikzstyle{vertex}=[circle,fill=black, minimum size=2pt,inner sep=1pt]
  \node[vertex] (v1) at (-0.5, 0.5){};
  \node[vertex] (v2) at (-0.5,-0.5){};
  \node[vertex] (v3) at (0.5, 0.5){};
  \node[vertex] (v4) at (0.5,-0.5){};
  \draw (v4)--(v3)--(v1)--(v2)--(v3);
  \end{tikzpicture}}

\newcommand{\pawe}{
  \begin{tikzpicture}[baseline=-0.6ex,scale=0.25]
  \tikzstyle{vertex}=[circle,fill=black, minimum size=2pt,inner sep=1pt]
  \node[vertex] (v1) at (-0.5, 0.5){};
  \node[vertex] (v2) at (-0.5,-0.5){};
  \node[vertex] (v3) at (0.5, 0.5){};
  \node[vertex] (v4) at (0.5,-0.5){};
  \draw (v3)--(v4)--(v2)--(v1)--(v4);
  \end{tikzpicture}}

\newcommand{\pawh}{
  \begin{tikzpicture}[baseline=-0.6ex,scale=0.25]
  \tikzstyle{vertex}=[circle,fill=black, minimum size=2pt,inner sep=1pt]
  \node[vertex] (v1) at (-0.5, 0.5){};
  \node[vertex] (v2) at (-0.5,-0.5){};
  \node[vertex] (v3) at (0.5, 0.5){};
  \node[vertex] (v4) at (0.5,-0.5){};
  \draw (v4)--(v1)--(v3)--(v4)--(v2);
  \end{tikzpicture}}

\newcommand{\pawbLab}{
  \begin{tikzpicture}[baseline=-0.6ex,scale=0.25]
  \tikzstyle{vertex}=[circle,fill=black, minimum size=2pt,inner sep=1pt]
\tikzstyle{Lvertex}=[circle, draw, minimum size=2pt,inner sep=1pt]
  \node[Lvertex] (v1) at (-0.5, 0.5){};
  \node[Lvertex] (v2) at (-0.5,-0.5){};
  \node[vertex] (v3) at (0.5, 0.5){};
  \node[vertex] (v4) at (0.5,-0.5){};
  \draw (v2)--(v1)--(v3)--(v4)--(v1);
  \end{tikzpicture}}

\newcommand{\pawbLbc}{
  \begin{tikzpicture}[baseline=-0.6ex,scale=0.25]
  \tikzstyle{vertex}=[circle,fill=black, minimum size=2pt,inner sep=1pt]
\tikzstyle{Lvertex}=[circle, draw, minimum size=2pt,inner sep=1pt]
  \node[Lvertex] (v1) at (-0.5, 0.5){};
  \node[vertex] (v2) at (-0.5,-0.5){};
  \node[Lvertex] (v3) at (0.5, 0.5){};
  \node[vertex] (v4) at (0.5,-0.5){};
  \draw (v2)--(v1)--(v3)--(v4)--(v1);
  \end{tikzpicture}}

  \newcommand{\pawbLbdgray}{
  \begin{tikzpicture}[baseline=-0.6ex,scale=0.25]
  \tikzstyle{vertex}=[circle,fill=gray, minimum size=2pt,inner sep=1pt]
\tikzstyle{Lvertex}=[circle, draw, minimum size=2pt,inner sep=1pt]
  \node[Lvertex] (v1) at (-0.5, 0.5){};
  \node[vertex] (v2) at (-0.5,-0.5){};
  \node[vertex] (v3) at (0.5, 0.5){};
  \node[Lvertex] (v4) at (0.5,-0.5){};
  \draw (v2)--(v1)--(v3)--(v4)--(v1);
  \end{tikzpicture}}

\newcommand{\pawbLac}{
  \begin{tikzpicture}[baseline=-0.6ex,scale=0.25]
  \tikzstyle{vertex}=[circle,fill=black, minimum size=2pt,inner sep=1pt]
\tikzstyle{Lvertex}=[circle, draw, minimum size=2pt,inner sep=1pt]
  \node[vertex] (v1) at (-0.5, 0.5){};
  \node[Lvertex] (v2) at (-0.5,-0.5){};
  \node[Lvertex] (v3) at (0.5, 0.5){};
  \node[vertex] (v4) at (0.5,-0.5){};
  \draw (v2)--(v1)--(v3)--(v4)--(v1);
  \end{tikzpicture}}

  \newcommand{\pawbLadgray}{
  \begin{tikzpicture}[baseline=-0.6ex,scale=0.25]
  \tikzstyle{vertex}=[circle,fill=gray, minimum size=2pt,inner sep=1pt]
\tikzstyle{Lvertex}=[circle, draw, minimum size=2pt,inner sep=1pt]
  \node[vertex] (v1) at (-0.5, 0.5){};
  \node[Lvertex] (v2) at (-0.5,-0.5){};
  \node[vertex] (v3) at (0.5, 0.5){};
  \node[Lvertex] (v4) at (0.5,-0.5){};
  \draw (v2)--(v1)--(v3)--(v4)--(v1);
  \end{tikzpicture}}

\newcommand{\pawbLcd}{
  \begin{tikzpicture}[baseline=-0.6ex,scale=0.25]
  \tikzstyle{vertex}=[circle,fill=black, minimum size=2pt,inner sep=1pt]
\tikzstyle{Lvertex}=[circle, draw, minimum size=2pt,inner sep=1pt]
  \node[vertex] (v1) at (-0.5, 0.5){};
  \node[vertex] (v2) at (-0.5,-0.5){};
  \node[Lvertex] (v3) at (0.5, 0.5){};
  \node[Lvertex] (v4) at (0.5,-0.5){};
  \draw (v2)--(v1)--(v3)--(v4)--(v1);
  \end{tikzpicture}}

\newcommand{\fourcycle}{
  \begin{tikzpicture}[baseline=-0.6ex,scale=0.25]
  \tikzstyle{vertex}=[circle,fill=black, minimum size=2pt,inner sep=1pt]
  \node[vertex] (v1) at (-0.5, 0.5){};
  \node[vertex] (v2) at (-0.5,-0.5){};
  \node[vertex] (v3) at (0.5, 0.5){};
  \node[vertex] (v4) at (0.5,-0.5){};
  \draw (v1)--(v2)--(v4)--(v3)--(v1);
  \end{tikzpicture}}

\newcommand{\fourcyclea}{
  \begin{tikzpicture}[baseline=-0.6ex,scale=0.25]
  \tikzstyle{vertex}=[circle,fill=black, minimum size=2pt,inner sep=1pt]
  \node[vertex] (v1) at (-0.5, 0.5){};
  \node[vertex] (v2) at (-0.5,-0.5){};
  \node[vertex] (v3) at (0.5, 0.5){};
  \node[vertex] (v4) at (0.5,-0.5){};
  \draw (v1)--(v3)--(v2)--(v4)--(v1);
  \end{tikzpicture}}

\newcommand{\fourcycleb}{
  \begin{tikzpicture}[baseline=-0.6ex,scale=0.25]
  \tikzstyle{vertex}=[circle,fill=black, minimum size=2pt,inner sep=1pt]
  \node[vertex] (v1) at (-0.5, 0.5){};
  \node[vertex] (v2) at (-0.5,-0.5){};
  \node[vertex] (v3) at (0.5, 0.5){};
  \node[vertex] (v4) at (0.5,-0.5){};
  \draw (v1)--(v4)--(v3)--(v2)--(v1);
  \end{tikzpicture}}

\newcommand{\cofourcycle}{
  \begin{tikzpicture}[baseline=-0.6ex,scale=0.25]
  \tikzstyle{vertex}=[circle,fill=black, minimum size=2pt,inner sep=1pt]
  \node[vertex] (v1) at (-0.5, 0.5){};
  \node[vertex] (v2) at (-0.5,-0.5){};
  \node[vertex] (v3) at (0.5, 0.5){};
  \node[vertex] (v4) at (0.5,-0.5){};
  \draw (v1)--(v2) (v3)--(v4);
  \end{tikzpicture}}

\newcommand{\cofourcycleb}{
  \begin{tikzpicture}[baseline=-0.6ex,scale=0.25]
  \tikzstyle{vertex}=[circle,fill=black, minimum size=2pt,inner sep=1pt]
  \node[vertex] (v1) at (-0.5, 0.5){};
  \node[vertex] (v2) at (-0.5,-0.5){};
  \node[vertex] (v3) at (0.5, 0.5){};
  \node[vertex] (v4) at (0.5,-0.5){};
  \draw (v1)--(v3) (v2)--(v4);
  \end{tikzpicture}}
  
\newcommand{\cofourcyclec}{
  \begin{tikzpicture}[baseline=-0.6ex,scale=0.25]
  \tikzstyle{vertex}=[circle,fill=black, minimum size=2pt,inner sep=1pt]
  \node[vertex] (v1) at (-0.5, 0.5){};
  \node[vertex] (v2) at (-0.5,-0.5){};
  \node[vertex] (v3) at (0.5, 0.5){};
  \node[vertex] (v4) at (0.5,-0.5){};
  \draw (v1)--(v4) (v3)--(v2);
  \end{tikzpicture}}

  \newcommand{\pfourb}{
  \begin{tikzpicture}[baseline=-0.6ex,scale=0.25]
  \tikzstyle{vertex}=[circle,fill=black, minimum size=2pt,inner sep=1pt]
  \node[vertex] (v1) at (-0.5, 0.5){};
  \node[vertex] (v2) at (-0.5,-0.5){};
  \node[vertex] (v3) at (0.5, 0.5){};
  \node[vertex] (v4) at (0.5,-0.5){};
  \draw (v2)--(v1)--(v3)--(v4);
  \end{tikzpicture}}

  \newcommand{\pfourc}{
  \begin{tikzpicture}[baseline=-0.6ex,scale=0.25]
  \tikzstyle{vertex}=[circle,fill=black, minimum size=2pt,inner sep=1pt]
  \node[vertex] (v1) at (-0.5, 0.5){};
  \node[vertex] (v2) at (-0.5,-0.5){};
  \node[vertex] (v3) at (0.5, 0.5){};
  \node[vertex] (v4) at (0.5,-0.5){};
  \draw (v4)--(v1)--(v3)--(v2);
  \end{tikzpicture}}

\newcommand{\pfourd}{
  \begin{tikzpicture}[baseline=-0.6ex,scale=0.25]
  \tikzstyle{vertex}=[circle,fill=black, minimum size=2pt,inner sep=1pt]
  \node[vertex] (v1) at (-0.5, 0.5){};
  \node[vertex] (v2) at (-0.5,-0.5){};
  \node[vertex] (v3) at (0.5, 0.5){};
  \node[vertex] (v4) at (0.5,-0.5){};
  \draw (v2)--(v1)--(v4)--(v3);
  \end{tikzpicture}}

\newcommand{\pfoure}{
  \begin{tikzpicture}[baseline=-0.6ex,scale=0.25]
  \tikzstyle{vertex}=[circle,fill=black, minimum size=2pt,inner sep=1pt]
  \node[vertex] (v1) at (-0.5, 0.5){};
  \node[vertex] (v2) at (-0.5,-0.5){};
  \node[vertex] (v3) at (0.5, 0.5){};
  \node[vertex] (v4) at (0.5,-0.5){};
  \draw (v2)--(v4)--(v1)--(v3);
  \end{tikzpicture}}
  
\newcommand{\pfourf}{
  \begin{tikzpicture}[baseline=-0.6ex,scale=0.25]
  \tikzstyle{vertex}=[circle,fill=black, minimum size=2pt,inner sep=1pt]
  \node[vertex] (v1) at (-0.5, 0.5){};
  \node[vertex] (v2) at (-0.5,-0.5){};
  \node[vertex] (v3) at (0.5, 0.5){};
  \node[vertex] (v4) at (0.5,-0.5){};
  \draw (v2)--(v3)--(v4)--(v1);
  \end{tikzpicture}}

  \newcommand{\pfourg}{
  \begin{tikzpicture}[baseline=-0.6ex,scale=0.25]
  \tikzstyle{vertex}=[circle,fill=black, minimum size=2pt,inner sep=1pt]
  \node[vertex] (v1) at (-0.5, 0.5){};
  \node[vertex] (v2) at (-0.5,-0.5){};
  \node[vertex] (v3) at (0.5, 0.5){};
  \node[vertex] (v4) at (0.5,-0.5){};
  \draw (v1)--(v3)--(v4)--(v2);
  \end{tikzpicture}}

\newcommand{\pfourLab}{
  \begin{tikzpicture}[baseline=-0.6ex,scale=0.25]
  \tikzstyle{vertex}=[circle,fill=black, minimum size=2pt,inner sep=1pt]
  \tikzstyle{Lvertex}=[circle, draw, minimum size=2pt,inner sep=1pt]
  \node[Lvertex] (v1) at (-0.5, 0.5){};
  \node[Lvertex] (v2) at (-0.5,-0.5){};
  \node[vertex] (v3) at (0.5, 0.5){};
  \node[vertex] (v4) at (0.5,-0.5){};
  \draw (v2)--(v1)--(v3)--(v4);
  \end{tikzpicture}}

\newcommand{\pfourLcdgray}{
  \begin{tikzpicture}[baseline=-0.6ex,scale=0.25]
  \tikzstyle{vertex}=[circle,fill=gray, minimum size=2pt,inner sep=1pt]
  \tikzstyle{Lvertex}=[circle, draw, minimum size=2pt,inner sep=1pt]
  \node[vertex] (v1) at (-0.5, 0.5){};
  \node[vertex] (v2) at (-0.5,-0.5){};
  \node[Lvertex] (v3) at (0.5, 0.5){};
  \node[Lvertex] (v4) at (0.5,-0.5){};
  \draw (v2)--(v1)--(v3)--(v4);
  \end{tikzpicture}}

\newcommand{\pfourLbc}{
  \begin{tikzpicture}[baseline=-0.6ex,scale=0.25]
  \tikzstyle{vertex}=[circle,fill=black, minimum size=2pt,inner sep=1pt]
  \tikzstyle{Lvertex}=[circle, draw, minimum size=2pt,inner sep=1pt]
  \node[Lvertex] (v1) at (-0.5, 0.5){};
  \node[vertex] (v2) at (-0.5,-0.5){};
  \node[Lvertex] (v3) at (0.5, 0.5){};
  \node[vertex] (v4) at (0.5,-0.5){};
  \draw (v2)--(v1)--(v3)--(v4);
  \end{tikzpicture}}

\newcommand{\pfourLac}{
  \begin{tikzpicture}[baseline=-0.6ex,scale=0.25]
  \tikzstyle{vertex}=[circle,fill=black, minimum size=2pt,inner sep=1pt]
  \tikzstyle{Lvertex}=[circle, draw, minimum size=2pt,inner sep=1pt]
  \node[vertex] (v1) at (-0.5, 0.5){};
  \node[Lvertex] (v2) at (-0.5,-0.5){};
  \node[Lvertex] (v3) at (0.5, 0.5){};
  \node[vertex] (v4) at (0.5,-0.5){};
  \draw (v2)--(v1)--(v3)--(v4);
  \end{tikzpicture}}

\newcommand{\pfourLbdgray}{
  \begin{tikzpicture}[baseline=-0.6ex,scale=0.25]
  \tikzstyle{vertex}=[circle,fill=gray, minimum size=2pt,inner sep=1pt]
  \tikzstyle{Lvertex}=[circle, draw, minimum size=2pt,inner sep=1pt]
  \node[Lvertex] (v1) at (-0.5, 0.5){};
  \node[vertex] (v2) at (-0.5,-0.5){};
  \node[vertex] (v3) at (0.5, 0.5){};
  \node[Lvertex] (v4) at (0.5,-0.5){};
  \draw (v2)--(v1)--(v3)--(v4);
  \end{tikzpicture}}

\newcommand{\pfourLad}{
  \begin{tikzpicture}[baseline=-0.6ex,scale=0.25]
  \tikzstyle{vertex}=[circle,fill=black, minimum size=2pt,inner sep=1pt]
  \tikzstyle{Lvertex}=[circle, draw, minimum size=2pt,inner sep=1pt]
  \node[vertex] (v1) at (-0.5, 0.5){};
  \node[Lvertex] (v2) at (-0.5,-0.5){};
  \node[vertex] (v3) at (0.5, 0.5){};
  \node[Lvertex] (v4) at (0.5,-0.5){};
  \draw (v2)--(v1)--(v3)--(v4);
  \end{tikzpicture}}

\newcommand{\fourstar}{
  \begin{tikzpicture}[baseline=-0.6ex,scale=0.25]
  \tikzstyle{vertex}=[circle,fill=black, minimum size=2pt,inner sep=1pt]
  \node[vertex] (v1) at (-0.5, 0.5){};
  \node[vertex] (v2) at (-0.5,-0.5){};
  \node[vertex] (v3) at (0.5, 0.5){};
  \node[vertex] (v4) at (0.5,-0.5){};
  \draw (v1)--(v2)--(v3) (v2)--(v4);
  \end{tikzpicture}}

\newcommand{\cofourstar}{
  \begin{tikzpicture}[baseline=-0.6ex,scale=0.25]
  \tikzstyle{vertex}=[circle,fill=black, minimum size=2pt,inner sep=1pt]
  \node[vertex] (v1) at (-0.5, 0.5){};
  \node[vertex] (v2) at (-0.5,-0.5){};
  \node[vertex] (v3) at (0.5, 0.5){};
  \node[vertex] (v4) at (0.5,-0.5){};
  \draw (v1)--(v3)--(v4)--(v1);
  \end{tikzpicture}}

\newcommand{\cofourstarb}{
  \begin{tikzpicture}[baseline=-0.6ex,scale=0.25]
  \tikzstyle{vertex}=[circle,fill=black, minimum size=2pt,inner sep=1pt]
  \node[vertex] (v1) at (-0.5, 0.5){};
  \node[vertex] (v2) at (-0.5,-0.5){};
  \node[vertex] (v3) at (0.5, 0.5){};
  \node[vertex] (v4) at (0.5,-0.5){};
  \draw (v3)--(v1)--(v2)--(v3);
  \end{tikzpicture}}

\newcommand{\cocherry}{
  \begin{tikzpicture}[baseline=-0.3ex,scale=0.25]
  \tikzstyle{vertex}=[circle,fill=black, minimum size=2pt,inner sep=1pt]
  \node[vertex] (v1) at (-0.5, 0){};
  \node[vertex] (v2) at (0.5,0){};
  \node[vertex] (v3) at (0,0.8){};
  \draw (v1)--(v2);
  \end{tikzpicture}}

\newcommand{\edge}{
  \begin{tikzpicture}[baseline=-0.6ex,scale=0.25]
  \tikzstyle{vertex}=[circle,fill=black, minimum size=2pt,inner sep=1pt]
  \node[vertex] (v1) at (0, -0.5){};
  \node[vertex] (v2) at (0,0.5){};
  \draw (v1)--(v2);
  \end{tikzpicture}}

\title{On graphs with modularity zero or near-zero}

\date{\today \vspace{-5mm}
}

\author[1]{Colin McDiarmid \thanks{Email: \textit{cmcd@ox.ac.uk}}}

\author[2]{Fiona Skerman\thanks{Email: \textit{fiona.skerman@math.uu.se}}}

\affil[1]{Department of Statistics, University of Oxford, United Kingdom}
\affil[2]{Department of Mathematics, Uppsala University, Sweden}

\begin{document}
\maketitle

\begin{abstract}
It is known that complete graphs and complete 
multipartite graphs have modularity zero. We show that the least number of edges we may delete from the complete graph $K_n$ to obtain a graph with non-zero modularity is $\lfloor n/2\rfloor +1$. Similarly we determine the least number of edges we may delete from or add to a complete bipartite graph to reach non-zero modularity.  We give some corresponding results for complete multipartite graphs, and a short proof that complete multipartite graphs have modularity zero.

We also analyse the modularity of very dense random graphs, and in particular we find that there is a transition to modularity zero when the average degree of the complementary graph drops below 1.

Finally we consider some natural variants of the definition of modularity; and investigate which graphs have corresponding modularity value 0, and the least number of edges we may delete from the complete graph $K_n$ to obtain a graph with non-zero modularity.
\end{abstract}

\section{Introduction} \label{sec.intro}
The modularity $\q(G)$ of a graph $G$ was introduced by Newman and Girvan in 2004~\cite{NewmanGirvan}, to give a measure of how well $G$ can be divided into `communities', and now modularity is at  the heart of the most popular algorithms used to cluster real data~\cite{popular}, including the  Louvain~\cite{louvain} and Leiden~\cite{traag2019louvain} algorithms. See Section~\ref{subsec.defns} below for precise definitions and further discussion.

We are interested here in the question of which graphs have modularity zero or near zero.
We saw in~\cite{ERmod} that for suitable random graphs~$R_n$, if the average degree of~$R_n$ tends to $\infty$ as $n \to \infty$ then $\q(R_n) \to 0$ in probability -- details are given below. But when is $\q(R_n)>0$ whp and when is it actually 0 whp? (For a sequence of events $A_n$ we say that $A_n$ holds \emph{with high probability (whp)} if $\pr(A_n) \to 1$ as $n \to \infty$.)
A key step in understanding this is deterministic.

It is known~\cite{nphard} and easy to see that complete graphs have modularity zero. 
Graphs which are close to a complete graph in that there are very few missing edges also have modularity zero.  We give the precise threshold in Theorem~\ref{thm.detmod0} below.

Each graph $G$ with at most 3 vertices has modularity $\q(G)=0$.
(By convention a graph $G$ with no edges has $\q(G)=0$.) 
The 4-vertex path has modularity $\frac16$, and the 4-vertex graph consisting of two disjoint edges has modularity $\frac12$; and indeed for each $n \geq 4$ there is an $n$-vertex graph $G$ with $\q(G)>0$. Given an integer~$n \geq 4$, let $f^{\rm q}(n)$ 
be the least number of edges missing from any $n$-vertex graph $G$ with $\q(G)>0$. 

\smallskip

\begin{thm} \label{thm.detmod0} 
For each $n \geq 4$ we have $f^q(n)=\lfloor n/2 \rfloor +1$.
\end{thm}
This theorem corrects Lemma~9 and the corresponding part of Theorem~2 in~\cite{trajanovski2012maximum}, which in the notation here says that $f^q(n) \geq n-2$ : see Section~\ref{subsec.prev_work} below for details. We also give results on how slight changes to the definition of modularity affect the result in Theorem~\ref{thm.detmod0} -- see the table in Figure~\ref{fig.mod_variants} on page~\pageref{fig.mod_variants} for a summary.

The following result gives a characterisation of the set of graphs which are extremal in the sense of Theorem~\ref{thm.detmod0}. 
We first consider the case $n=4$, before covering the general case $n \geq 5$ in the proposition. By Theorem~\ref{thm.detmod0} we have $f^q(4)=3$.  It is easy to see that a set $F$ of 3 edges in $K_4$ satisfies $\q(K_4 -F)>0$ if and only if $F$ forms a path. Note that if $F$ is a star
(which is a bipartite graph),
then $K_4-F = \kfour-\fourstar = \cofourstar$, i.e. $K_3$ plus an isolated vertex, which has modularity zero.
\begin{prop} \label{prop.detmod0b}
Let $n \geq 5$. For graphs $G$ on $[n]$ with $e(\overline{G})=f^q(n)$, 
$\q(G)>0$ if and only if $\overline{G}$ is bipartite. 
\end{prop}
A result related to this proposition, see Remark~\ref{rem.was_part_a}, will be useful to us in Example~\ref{ex.K2222}. 

Given a graph $H$, let $\delta^-(H)$ be
the least number of edges we can delete to obtain a subgraph $H'$ with $\q(H')>0$. If there is no subgraph $H'$ with $\q(H')>0$ we set $\delta^-(H)=\infty$.\label{defn.deltaminus}
(The only graph with $m \geq 4$ edges with this property is the $m$-edge star $K_{1,m}$. This can be seen by noting that $\delta^-(H)$ is finite as soon as $H$ contains two disjoint edges.)
In this notation, Theorem~\ref{thm.detmod0} says that $\delta^-(K_n) = \lfloor n/2 \rfloor +1$ for $n \geq 4$, so large complete graphs have modularity 0 `robustly'.

\bigskip

To recap, any complete graph has modularity zero, and this is robust in the sense that graphs `near' complete graphs also have modularity zero. In contrast, while any complete multipartite graph~$H$ has modularity zero~\cite{majstorovic2014note,bolla2015spectral} (see Section~\ref{sec.multi0} below for a short proof), here there is no robustness in general.
We focus on complete bipartite graphs, where we can tell the complete story  about robustness reasonably quickly: we find that, except for a few small examples, it is enough to add or remove a single edge to obtain a graph with positive modularity.  In Section~\ref{subsec.multi} we give some corresponding partial robustness results for complete multipartite graphs, and in Section~\ref{sec.multi0} we use a recent result on modularity and graph expansion~\cite{modexp} to give a short and simple proof that complete multipartite graphs have modularity~0.

To state the results, let us define natural counterparts to $\delta^-$: given a graph $H$, let $\delta^+(H)$ be the least number of edges we can \emph{add} (between already existing vertices) to obtain a graph $H'$ with $\q(H')>0$ (where $\delta^+(H)=\infty$ if there is no such graph $H'$). Likewise let $\delta(H)$ be the minimum number of edits, either adding or removing edges, to obtain a graph $H'$ with positive modularity (where~$\delta(H)=\infty$ if there is no such graph $H'$, that is if $v(H) \leq 3$).

\smallskip

\begin{thm} \label{thm.bip_delta_all}
Let $G=K_{s,t}$ be a complete bipartite graph with $s \leq t$.
\begin{description}
\item{(a)}
 If $s=1$ then $\delta^-(G)=\infty$, and if $s \geq 2$ then $\delta^-(G)=1$.
\item{(b)}
 If $s=1$ and $t \geq 4$, or if $s \geq 2$ and $t \geq 3$, then $\delta^+(G)=1$.  In the other cases (namely when $(s,t) = (1,1), (1,2), (1,3)$ or $(2,2)$), we have $\delta^+(G)=\infty$.
\item{(c)}
 If $s=1$ and $t \geq 4$, or if $s \geq 2$, then $\delta(G)=1$. 
 If $(s,t) = (1,1)$ or $ (1,2)$ then $\delta(G)=\infty$, and in the remaining case $(s,t)=(1,3)$ we have $\delta(G)=2$.
\end{description}
\end{thm}
Theorems~\ref{thm.detmod0} and~\ref{thm.bip_delta_all}
suggest that complete graphs $K_n$ may be the `furthest' from graphs with positive modularity.

\begin{question}\label{q.far}
Is it true that, for each graph $H$ with $n \geq 4$ vertices, we have $\delta(H) \leq \delta(K_n)$?
\end{question}

\medskip
 
Now consider random graphs. The modularity of random graphs $\Gnp$ and $G_{n,m}$ (see Section~\ref{subsec.defns} below for definitions) is quite fully analysed in~\cite{ERmod}, except for the very dense case.
By Theorem~1.3 of~\cite{ERmod}, when $p=p(n)$ satisfies $p \geq 1/n$ and $p$ is bounded away from 1, we have $\q(\Gnp)= \Theta((np)^{-1/2})$ whp; and by Theorem~4.1 of the same paper, when $1/n \leq p \leq 1- c_0/n$ for a suitable (large) constant $c_0$ we have $\q(\Gnp) = \Omega\left(\sqrt{\frac{1-p}{np}}\right)$ whp.
We focus here on the very dense case, when $p=1- \Theta(1/n)$.  Note that in this case 
$\sqrt{\frac{1-p}{np}}= \Theta(1/n)$. It follows easily from Theorem~\ref{thm.detmod0} that when $1-p$ is very small, say $1-p\leq (1-\eps)/n$ (for some $\eps>0$), then whp $\q(G_{n,p})=0$.  We shall see that when $1-p=1/n$, two different behaviours can occur each with probability about $\frac12$, namely $\q(G_{n,p})=0$ and $\q(G_{n,p})= \Theta(n^{-3/2})$; and if~$1-p$ is just a little larger, say $1-p \geq (1+\eps)/n$ (for some $\eps>0$), then whp ~$\q(G_{n,p}) = \Theta(1/n)$. 

\smallskip

\needspace{3\baselineskip}
\begin{thm}\label{thm.Gnp} 
\begin{description}
\item{(a)}
If $p=p(n)$ satisfies $p \geq 1- 1/n + \omega(n)/n^{3/2}$ (where $\omega(n) \to \infty$ as $n \to \infty$) then $\q(G_{n,p})=0$ whp.
\item{(b)}
  If  $p=1-1/n$, then (i) $\pr(\q(G_{n,p})=0) = \frac12 +o(1)$,
  and (ii) given $\eps>0$ there exist $0<\alpha<\beta$ such that $\pr(\alpha n^{-3/2} < \q(G_{n,p}) < \beta n^{-3/2}) \geq \frac12 - \eps$.
\item{(c)}
Given $1<c_1<c_2$, there exist $0< \alpha < \beta$ such that, if $p$ satisfies 
$1- c_2/n \leq p \leq 1-c_1/n$ then $\alpha/n \leq \q(G_{n,p}) \leq \beta/n$ whp.
\end{description}
\end{thm}
The lower bound in part\,(c) of Theorem~\ref{thm.Gnp} allows us to extend the range of $p$ in Theorem~4.1 of~\cite{ERmod} up to nearly $1-1/n$ (and part (a) shows that we cannot go much further). Using Theorem~4.1 of~\cite{ERmod} for the lower parts of the range we obtain the following result.

\smallskip

\begin{figure}[t]
    \centering
    \includegraphics[scale=0.78]{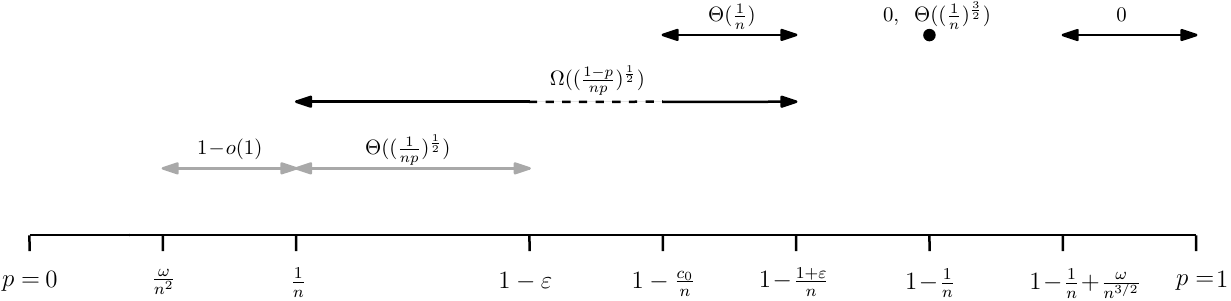}
    \caption{
An illustration of the bounds known to hold whp on $\q(G_{n,p})$ for different ranges of $p=p(n)$. 
On the top level, we depict the results of Theorem~\ref{thm.Gnp}; the left interval corresponds to part (a), the middle to part (b) and the right to part~(c). For $p=1-1/n$, the modularity of $G_{n,p}$ has probability about a half of being zero, and about a half of being of order $n^{-3/2}$, see part (b) of Theorem~\ref{thm.Gnp}.  
On the middle level, we depict the result of Theorem~\ref{thm.from-modER}, which extends earlier results from~\cite{ERmod} shown in grey on the bottom level.}\label{fig.Gnp}
\end{figure}

\begin{thm} \label{thm.from-modER}
Given $\eps>0$ there exists $\alpha>0$ such that, if $p=p(n)$ satisfies $1/n \leq p \leq 1-(1+\eps)/n$, then $\q(G_{n,p}) \geq \alpha \sqrt{\frac{1-p}{np}}$ whp.
\end{thm}

\subsection{Outline of the paper}

In the remainder of this section, we give definitions and preliminaries concerning modularity, and describe some relevant previous work. In Section 2, we give proofs for the new deterministic results above (Theorem~\ref{thm.detmod0}, Proposition~\ref{prop.detmod0b}, and Theorem~\ref{thm.bip_delta_all}), and for some related results on complete multipartite graphs; in Section 3 we give a short new proof (based on a recent expansion result) that complete multipartite graphs have modularity zero; in Section 4 we consider random graphs and prove Theorems~\ref{thm.Gnp} and~\ref{thm.from-modER} above; 
in Section 5 we discuss some natural variants of the definition of modularity (for example where we insist that the random graph with the given degrees is simple), and how these changes would impact our results; and in the final section we make a few concluding remarks.

\subsection{Definitions and preliminaries}\label{subsec.defns}
Given a non-empty graph $G$ (with at least one edge), we give a modularity score to each vertex-partition (or `clustering') : the modularity $\q(G)$ (sometimes called the `maximum modularity') of~$G$ is defined to be the maximum of these scores over all vertex-partitions. For a set $A$ of vertices, let the \emph{volume} $\vol(A)$ or $\vol_G(A)$ be the sum over the vertices $v$ in $A$ of the degree $d_v$, also let~$e(A)$ or~$e_G(A)$ denote the number of edges within $A$; and for $B$ disjoint from $A$, let $e(A, B)$ or $e_G(A,B)$ denote the number of edges between $A$ and $B$.
Given a vertex-partition $\cA$, let $\delta_{ij}(\cA)$ be
the indicator that vertices $i$ and $j$ are in the same part in the partition (so $\delta_{ii}(\cA)$ is always $1$). Also, denote by $e^{\rm int}_{\cA}(G)$ the number of edges in $G$ with both endpoints in the same part of $\cA$.

\begin{defn}[Newman \& Girvan~\cite{NewmanGirvan},  see also Newman~\cite{NewmanBook}]\label{def.mod}
Let $G=(V,E)$ be a graph with $m\geq 1$ edges. For a partition $\cA$ of $V$, the modularity score of $\cA$ on $G$ is 
\begin{equation*} q_\cA(G) = 
\frac{1}{2m}\sum_{i,j \in V}  \left( {\mathbf 1}_{ij\in E} - \frac{d_i d_j}{2m} \right) \delta_{ij}(\cA)
= \frac{1}{m}\sum_{A \in \cA} e(A) - \frac{1}{4m^2}\sum_{A\in \cA} \vol (A)^2; \end{equation*}
and the modularity of $G$ is $\q(G)=\max_\cA q_{\cA}(G)$, where the maximum is over all partitions~$\cA$ of~$V$.
\end{defn}

Isolated vertices are irrelevant. We need to give empty graphs (with no edges) some modularity value: conventionally we set $\q(G)=0$ for each such graph $G$, and we set $q_\cA(G)=0$ for every partition $\cA$ of $V(G)$. The second equation for $q_{\cA}(G)$ expresses modularity as the difference of two terms, the \emph{edge contribution} or \emph{coverage} $q^E_\cA(G)=\sum_A e(A)/m  =  e_\cA^{\rm int}(G)/m$, and the \emph{degree tax} $q^D_\cA(G)=\sum_A\vol(A)^2/(4m^2)$.  
Since, for any non-empty graph, we have $q^E_{\cA}(G) \leq 1$ and $q^D_{\cA}(G) >0$, we have $q_{\cA}(G)<1$ for any such graph $G$.  Also, the trivial partition $\cA_0$ with all vertices in one part has $q^E_{\cA_0}(G) = q^D_{\cA_0}(G) =1$, so $q_{\cA_0}(G)  = 0$.  Thus we have 
\begin{equation*} 0 \leq \q(G) < 1.\end{equation*}
Suppose that we use the configuration model (see for example~\cite{bollobas2001random, JLRbook, frieze2015book}) to pick a `nearly uniform' random pseudograph $R$ (where multiple edges and loops are allowed) with degree sequence $(d_1, \ldots,  d_n)$ where $\sum_i d_i = 2m$. Then the expected number of edges in $R$ between distinct vertices $i$ and $j$ is $d_i d_j/(2m-1)$; and it follows (see~\eqref{eqn.R}) that, if $m$ is large, the degree tax is approximately
the expected value of $e^{\rm int}_{\cA}(R)/m$, see~\eqref{eqn.R0} for exact calculations. 
This is the original rationale for the definition of the modularity score: whilst rewarding the partition for capturing a high proportion of edges within the parts, we should penalise by (approximately) the expected proportion of such edges. In Section~\ref{sec.mod_var} we amplify this discussion, and consider other possible distributions for the random (pseudo)graph $R$. Then we investigate when we get modularity~0 with some natural variant versions of this definition.
 
\smallskip

\needspace{4\baselineskip}\noindent
\emph{Random graphs}

Let~$n$ be a positive integer.  Given $0 \leq p \leq 1$, the (binomial) random graph $\Gnp$ has vertex set $[n]:=\{1,\ldots,n\}$ and the $\binom{n}{2}$ possible edges appear independently with probability $p$.  Given an integer $m$ with $0 \leq m \leq \binom{n}{2}$, the (Erd\H{o}s-R\'enyi) random graph $\Gnm$ is sampled uniformly from the $m$-edge graphs on vertex set $[n]$.  These two random graphs are closely related when $m \approx \binom{n}{2} p$, see for example~\cite{frieze2015book, luczak1990equivalence}. 

Recall, that for a sequence of events $A_n$ we say that $A_n$ holds \emph{with high probability (whp)} if $\pr(A_n) \to 1$ as $n \to \infty$. For a sequence of random variables $X_n$ and a real number $a$, we write $X_n \overset{p} \rightarrow a$ if $X_n$ converges in probability to $a$ as $n \to \infty$ (that is, if for each $\eps>0$ we have $|X_n-a| < \eps$ whp). For $x=x(n)$ and $y=y(n)$ we write $x \sim y$ to indicate that $\,x= (1+o(1))y\,$ as $n \to \infty$.

\smallskip

\noindent
\emph{More on modularity}

Let $G=(V,E)$ be a nonempty graph, and let $\nu = \vol(G)$.  For $U \subseteq V$, let $p_G(U) = 2e(U) \nu - \vol(U)^2$. Observe that $p_G(\emptyset)=p_G(V)=0$. If $\cA$ is a partition of $V(G)$, then $q_{\cA}(G) = \nu^{-2}  \sum_{A \in \cA} p_G(A)$.

There are two useful alternative expressions for $p_G(U)$. Since $\vol(U) = 2e(U) + e(U,\overline{U})$ (where~$\overline{U}$ denotes $V \backslash U$), we have
\begin{eqnarray*}
  p_G(U) &=& \notag
  2e(U) \nu - \vol(U)^2 = (\vol(U) - e(U,\overline{U})) \nu - \vol(U) (\nu - \vol(\overline{U}))\\
& = & - e(U,\overline{U}) \nu  + \vol(U) \vol(\overline{U})\,.
\end{eqnarray*}
Thus  
\begin{equation} \label{eqn.pG1}
    p_G(U) = \vol(U) \vol(\overline{U})  - e(U, \overline{U}) \nu\,.
\end{equation}
Similarly
\[  p_G(U) 
= 4 e(U) ( e(U)+e(U,\overline{U}) + e(\overline{U})) - (2e(U)+e(U,\overline{U}))^2\,, \]
and so
\begin{equation} \label{eqn.pG2}
p_G(U) = 4e(U)e(\overline{U}) - e(U,\overline{U})^2.
\end{equation}
These expressions for $p_G(U)$ are symmetric in $U$ and~$\overline{U}$, so $p_G(U) = p_G(\overline{U})$.
Now, if $\cA$ is a bipartition with parts $U$ and $\overline{U}$, then
\[ q_{\cA}(G) = \nu^{-2} (p_G(U) + p_G(\overline{U})) = 2 \nu^{-2} p_G(U).\]
Thus there is a bipartition $\cA$ with $q_{\cA}(G)>0$ if and only if  there is a set $U$ of vertices with $p_G(U)>0$, if and only if $\q(G)>0$. We have thus shown the following result, which we record as a lemma.  

\smallskip

\begin{lemma}\label{lem.pZero} 
Let $G$ be a graph with at least one edge. Then the following are equivalent : (i) $\q(G)=0$, (ii) $q_\cA(G) \leq 0$ for each bipartition $\cA$ of $V(G)$, and (iii) $p_G(U) \leq 0$ for each $U \subseteq V(G)$.
\end{lemma}

Recall the very nice result of Dinh and Thai~\cite{dinh2011finding}, that for any $k\geq2$ there is a partition $\cA$ with at most $k$ parts such that $q_\cA(G) \geq \q(G)\,(1-1/k)$. Note that this implies the equivalence of~(i) and (ii) in the lemma (that is, $\q(G)>0$ if and only if there is a bipartition $\cA$ with $q_\cA(G)>0$).

\smallskip

\noindent
\emph{Robustness}

The following `robustness' result will be used in the proof of Theorem~\ref{thm.Gnp}. It is essentially Lemma~5.1 in~\cite{ERmod}. Since we set $\q(G)=0$ if $G$ has no edges, we can rephrase the lemma there as the following.

\smallskip

\begin{lemma}\label{lem.edgeSim}
If $G=(V,E)$ is a graph, $E_0$ is a non-empty subset of $ E$, and 
$G'= (V,E \setminus E_0)$, then
\[ |\q(G)-\q(G')|< 2 \, |E_0|/|E|. \]  
\end{lemma}

\smallskip

\needspace{2\baselineskip}
\noindent
\subsection{Relevant previous work on modularity}\label{subsec.prev_work}
As discussed above the following graphs are known to have zero modularity: complete graphs~\cite{nphard} and complete multipartite graphs~\cite{majstorovic2014note, bolla2015spectral}. There are some general lower bounds which  can show strictly positive modularity : if we let $\bar{d}$ denote the average degree of a graph, then connected graphs with maximum degree $o(n)$ have modularity at least $2/\bar{d}-o(1)$~\cite{prokhorenkova2017modularity}, and both preferential attachment graphs~\cite{mod2023universal,prokhorenkova2017modularity} and deterministic graphs with mild assumptions on the degree sequence~\cite{mod2023universal} have modularity $\Omega(\bar{d}^{-1/2})$. 

At the other extreme, which graphs have modularity near 1? De Montgolfier, Soto and Viennot~\cite{modgraphclasses} define a sequence of graphs to be \emph{maximally modular} if their modularity values tend to 1 as their number $m$ of edges tends to infinity. The following are known to be maximally modular: trees with maximum degree $o(m)$, graphs of bounded genus with maximum degree $o(m)$~\cite{lasonsulkowska2023modularity}, graphs where treewidth times maximum degree is $o(m)$~\cite{treelike}, lattices~\cite{modgraphclasses,thesis}, whp random graphs $G_{n,p}$ and~$G_{n,m}$ with average degree at most $1+o(1)$~\cite{ERmod}, whp random hyperbolic graphs~\cite{chellig2022modularity}, and whp spatial preferential attachment graphs~\cite{prokhorenkova2017modularity}. 
There is also work on when modularity is able to pick up planted structure in random graphs; the stochastic block model (SBM) is considered  
in~\cite{bhamidi2026stochastic,bickel2009nonparametric}, the degree corrected SBM in~\cite{bickel2015correction,zhao2012consistency} and a related ABCD model in~\cite{kaminski2022modularity}. 

We consider alternative definitions for the degree tax in the modularity score in Section~\ref{sec.mod_var} of this paper. We note that several definitions for the modularity score exist for hypergraphs~\cite{kaminski2024modularity,poda2024comparison} -- see~\cite{poda2024comparison} for a comparison of the different definitions.

For a summary of known modularity values, see the table in~\cite{chapter}.
Notable recent results include improved bounds for random regular graphs, which were obtained in~\cite{lichev2022modularity} -- in particular the modularity of random cubic graphs was shown to lie in the interval $[0.667, 0.79]$ whp. Additionally, for the preferential attachment model $G_n^h$ (where $h\geq 2$ edges are added with each new vertex), the modularity satisfies 
$\: \Omega(1/\sqrt{h})=\q(G_n^h)=O(\sqrt{\log h}/\sqrt{h})$ whp. Here the lower bound is known to hold for generalised preferential attachment models~\cite{prokhorenkova2017modularity,mod2023universal}, and the upper bound follows from results showing the expansion of a set of vertices is linked to the `ages' of those vertices -- see~\cite{rybarczyksulkowska2025PA}.

Lastly, it is known that for any $n$-vertex base graph~$G$, letting $G_p$ be the `percolated' random graph obtained by keeping each edge with probability~$p$, then $\q(G_p)\rightarrow \q(G)$ in probability if $e(G)p/n\rightarrow \infty$ as $n \to \infty$~\cite{sampling}.

As mentioned earlier, Theorem~\ref{thm.detmod0} corrects Lemma~9 and the corresponding part of Theorem~2 in~\cite{trajanovski2012maximum}, which claimed that $\max_{|\cA|=2} q_\cA(G) < 0$ if $e(G)\in [\binom{n-1}{2}+2, \binom{n}{2}]$ (where the maximum is over all bipartitions $\cA$).  
This claim is also restated in Theorem~2 of~\cite{trajanovski2012maximum}, as the third range of the number $m$ of edges.
Recall, for example by Lemma~\ref{lem.pZero}, that $\q(G)>0$ if and only if there exists a bipartition $\cA$ 
with $q_\cA(G)>0$. Note also that $\binom{n-1}{2}+2 = \binom{n}{2}-n+3$. Thus, if Lemma~9 of~\cite{trajanovski2012maximum} were correct, it would imply that if $e(G) \geq \binom{n}{2}-n+3$ then 
$\q(G)=0$; that is, that $f^q(n)\geq n-2$, in contradiction to~Theorem~\ref{thm.detmod0}.

\needspace{8\baselineskip}
\section{Proofs for new deterministic results}
\label{sec.detmod0}
In this section we first prove Theorem~\ref{thm.detmod0} on $f^q(n) = \delta^-(K_n)$, and then quickly prove Proposition~\ref{prop.detmod0b} which characterises the corresponding extremal graphs.  Then, after a couple of preliminary lemmas, we prove Theorem~\ref{thm.bip_delta_all}, in which we determine $\delta^-(G)$, $\delta^+(G)$ and $\delta(G)$ for all complete bipartite graphs $G$.
Finally we give a couple of examples and briefly consider multipartite graphs. (Recall that $\delta^+, \delta^-, \delta$ are the minimum numbers of edges that may be added, removed or edited to obtain a graph with positive modularity - see the  paragraphs preceding Theorem~\ref{thm.bip_delta_all}
for definitions.)

\subsection{Nearly complete graphs: proof of Theorem~\ref{thm.detmod0}}
\label{subsec.21}

We first give a straightforward proof of the upper bound on $f^q(n) \,(=\delta^-(K_n) =\delta(K_n)$) for $n \geq 4$. We proceed via a slightly stronger result, since it will be helpful to refer back to it later. For disjoint vertex sets $A$ and $B$, write $E(A,B)$ for the set of edges between $A$ and $B$. 

\begin{clm}\label{clm.upperbound} Let $n\geq 4$, let $\cA=(A, B)$ be a bipartition of $[n]$ where $|A|= \lceil n/2 \rceil$ and $|B|= \lfloor n/2 \rfloor$, and let $G$ be obtained from $K_n$ by deleting any $\lfloor n/2\rfloor+1$ edges from $E(A,B)$. Then $\q(G)>0$.
\end{clm}

\begin{proof}[Proof of Claim~\ref{clm.upperbound}]
Let $F=E(A,B)$ be the set of edges between the parts, and let $k= \lfloor n/2 \rfloor +1$. Note that if $n$ is even then $|F| = n^2/4 \geq n \geq k$, and if $n$ is odd then $|F| = (n^2-1)/4 \geq n-1/4 \geq k$. 
Thus we may form $G$ by deleting any $k$ edges from $F$, so $G$ has $m=\binom{n}{2}-k$ edges. We consider separately the cases when $n$ is even, odd.

Suppose first that $n$ is even, so $k=n/2 +1$ and $m=\frac{1}{2}(n^2-2n-2)$. Since $|A|=|B|$ and all removed edges are between $A$ and $B$, we have $\vol(A)= \vol(B)$ and thus $q^D_{\cA}(G) =1/2$. Also
\[ q^E_{\cA}(G) = \frac{2\binom{n/2}{2}}{m} =\frac{n(n-2)}{2(n^2-2n-2)} = \frac12 + \frac{1}{n^2-2n-2}.\]
Hence $q_\cA(G) >0$, as required.

Now suppose that $n$ is odd, so $k=(n+1)/2$ and $m= \frac12 (n^2-2n-1) > 1$. Then
\begin{eqnarray*}
q^E_{\cA}(G) &=&
\frac{\binom{\lceil n/2 \rceil}{2}+\binom{\lfloor n/2 \rfloor}{2}}{m} \;\; = \;\; \frac{(n\!+\!1)(n\!-\!1)+ (n\!-\!1)(n\!-\!3)}{8m}\,,
\end{eqnarray*}
and we may check that the numerator is $4m+4$, so 
\begin{eqnarray*}
q^E_{\cA}(G)
\;\;\: = \;\; \frac12 + \frac{1}{2m}\,.
\end{eqnarray*}
Also $\vol(A)= \frac12 (n+1)(n-1) -k = \frac12(n^2-n-2)$ and $\vol(B)= \frac12 (n-1)^2 -k = \frac12(n^2-3n)$, 
so
\begin{eqnarray*}
q^D_{\cA}(G) &=&
\frac{1}{4}\, \frac{(n^2-n-2)^2+ (n^2-3n)^2}{4m^2}\\
&=&
\frac14 \, \frac{8m^2 + 4m+4}{4m^2}\;\; = \;\;
\frac12 + \frac{m+1}{4m^2}\,.
\end{eqnarray*}
Hence $\,q_\cA(G) = (m-1) / 4m^2 >0$, and this completes the proof of the claim. 
\end{proof}

\bigskip

\begin{proof}[Proof of Theorem~\ref{thm.detmod0}]
The upper bound in the theorem follows by Claim~\ref{clm.upperbound}, so
it remains to prove the lower bound. 
We must show that if $n \geq 4$ and we form $G$ by removing $x \leq n/2$ edges from $K_n$ then $\q(G)=0$.
Let us note first one useful inequality.
Let $G$ be a graph on vertex set $V$, let $U \subseteq V$, and form $G'$ by `moving' an edge currently between $U$ and $\overline{U}$ to the location of a missing edge within $U$ or $\overline{U}$.  Then by the expression~\eqref{eqn.pG2} for $p_G(U)$,
\begin{equation}\label{eqn.mono}
p_{G'}(U) > p_G(U).
\end{equation}
(This is similar to the `local rewiring' result of~\cite[Lemma~1]{trajanovski2012maximum}.)

Now suppose that the graph $H$ on $V=[n]$ and $\emptyset \neq U \subset V$ are such that~$p_H(U)$ maximises the value $p_G(W)$ over all graphs $G$ on $V$ with at most $n/2$ edges missing and all $\emptyset \neq W \subset V$. Since the number of edges missing is at most $n/2<n-1$ there are edges of $H$ in $(U,\overline{U})$. Hence by inequality\,(\ref{eqn.mono}), no edges within $U$ or within $\overline{U}$ are missing. 

Suppose that $|U|=j$ (where $1 \leq j \leq n-1$) and $x$ edges are missing from $H$. Then
\begin{eqnarray*}
p_H(U) &=&
\vol_H(U) \vol_H(\overline{U})  - e_H(U, \overline{U}) \vol(H) \\
&=&
(j(n-1)-x)((n-j)(n-1)-x) - (j(n-j)-x)(n(n-1)-2x)\\
&=&
-j(n-j)(n-1) + 2xj(n-j) - x^2\\
&=&
j(n-j)(2x-(n-1)) -x^2.
\end{eqnarray*}
Hence $p_H(U)<0$ if $x \leq (n-1)/2$.
Suppose that $x=n/2$.  Then $p_H(U)= j(n-j) -(n/2)^2$, so $p_H(U)<0$ if $j \neq n/2$, and $p_H(U)=0$ if $j = n/2$.
This completes the proof of Theorem~\ref{thm.detmod0}.
\end{proof}

\subsection{Extending Theorem~\ref{thm.detmod0}, proof of Proposition~\ref{prop.detmod0b}}

A proof of Proposition~\ref{prop.detmod0b} follows quickly from the following simple lemma and the proof of Theorem~\ref{thm.detmod0}.

\begin{lem}\label{lem.nearly_balanced}
Let $H$ be a bipartite graph on $[n]$ with $k \leq n/2 +1$ edges.
Then there is a bipartition $\cA=(A, B)$ of $[n]$ such that all edges in $H$ are between the parts, and where $|A|= \lceil n/2 \rceil$ and $|B|= \lfloor n/2 \rfloor$; except possibly in the case where $n$ is even and $k=n/2 +1$, when we may still insist that $n/2 -1 \leq |A|, |B| \leq n/2+1$.
\end{lem}

It is possible to be more precise in the case when $n$ is even and $k=n/2+1$, see Proposition~\ref{prop.unbalancedmeans1star}.

\begin{proof}[Proof of Lemma~\ref{lem.nearly_balanced}]
Let $\cA=(A,B)$ be a bipartition of $[n]$ with $E(H) \subseteq E(A,B)$, with $|A| \geq |B|$, and with $|A|$ as small as possible.
Thus $|A| \geq \lceil n/2 \rceil$.
If  $|A|= \lceil n/2 \rceil$ then of course $|B|= \lfloor n/2 \rfloor$.

 Now suppose that $|A| > \lceil n/2 \rceil$.  Then $A$ contains no isolated vertex $v$ (since we could move $v$ to $B$, and make $|A|$ one smaller), so $|A| \leq k$.  Now $n$ must be even and $k=n/2 +1$, and $n/2 -1 \leq |A|, |B| \leq n/2+1$, as required.
\end{proof}

\begin{proof}[Proof of Proposition~\ref{prop.detmod0b}] 
Let the graph $G$ on $[n]$ be such that $e(\overline{G})= \lfloor n/2 \rfloor +1 =f^q(n)$. In part (a) of the proof, we show that if $\overline{G}$ is bipartite,  
then $\q(G)>0$.  In part (b) of the proof, we show that if $\q(G)>0$, then $\overline{G}$ is bipartite.

(a)
Suppose that $\overline{G}$ is bipartite. We consider the two cases in the conclusion of Lemma~\ref{lem.nearly_balanced} separately.
For the case when  $|A|= \lceil n/2 \rceil$ and $|B|= \lfloor n/2 \rfloor$,
$\q(G)>0$ follows directly from Claim~\ref{clm.upperbound}. So it remains to consider the case when $n$ is even, $n \geq 6$, $k=\frac12 n +1$, $|A|=\frac12 n+1$, $|B|=\frac12 n -1$ and 
all edges in $H$ are between $A$ and $B$.

We must show that $q_\cA(G)>0$.
As before, $m= \binom{n}2-k= \binom{n}2 -\frac12n -1= \frac12 (n^2 -2n-2)$.
We may check that
\[ \binom{\frac12 n +1}2 + \binom{\frac12 n -1}2 = \tfrac14 (n^2 -2n +4),\]
and thus
\[ q_{\cA}^E(G) = \tfrac{\frac14(n^2-2n+4)}{\frac12 (n^2-2n-2)} = \tfrac12 + \frac{3}{n^2-2n-2}.\]
Also
\begin{eqnarray*}
 q_{\cA}^D(G) &=&
 \frac{\frac14 ( (n^2-4n)^2 + (n^2-4)^2)}{(n^2-2n-2)^2}\\
 &=&
 \tfrac12 \, \frac{n^4-4n^3+4n^2+8}{n^4-4n^3+4n+4}\\
 &=&
 \tfrac12 + \frac{2n^2-2n+2}{n^4 -4n^3+4n+4}. 
 \end{eqnarray*}
 Now
 \[ q_{\cA}^E(G) - \tfrac12 = \frac{3}{n^2-2n-2} =  \frac{3n^2-6n-6}{(n^2-2n-2)^2},\]
so
\[ q_{\cA}^E(G) - q_{\cA}^D(G) =
\frac{n^2-4n-8}{(n^2-2n-2)^2}
=
\frac{(n-2)^2-12}{(n^2-2n-2)^2}
\]
which is $>0$ for all $n \geq 6$.  Thus $q_\cA(G)>0$, as required.

\smallskip

(b) 
Now we show that if $\q(G)>0$, then $\overline{G}$ is bipartite.  Suppose that $\overline{G}$ is not bipartite: we shall show that $\q(G)=0$, and thus complete the proof.

Let $\cA = (A, B)$ be any bipartition of $V(G)$.  Then
$e^{\rm int}_\cA(G) \leq \binom{|A|}2+\binom{|B|}2 -1$.
Suppose first that $n$ is even.
Then
$e^{\rm int}_\cA(G) \leq 2 \binom{n/2}2 -1 = \frac14 (n^2-2n-4)$; and
$e(G) = \binom{n}2- \frac12 n -1 = \frac12 (n^2 - 2n -2)$.  Hence
\[q^{E}_\cA(G)  \leq \tfrac12 \frac{n^2-2n-4}{n^2-2n-2} < \tfrac12.\]

Now suppose that $n$ is odd.  Then
\begin{eqnarray*}
 e^{\rm int}_\cA(G) 
 &\leq & \textstyle \binom{(n-1)/2}2 + \binom{(n+1)/2}2 -1 \\
 &=& 
\tfrac18 ((n-1)(n-3) + (n^2-1))-1\\
&=& \tfrac14 (n^2-2n-3),
\end{eqnarray*}
and
$e(G) = \binom{n}2- \tfrac12 n -\tfrac12 = \tfrac12 (n^2 -2n -1)$.  Hence
\begin{eqnarray*}
q^{E}_\cA(G) 
&\leq & 
\tfrac12 \frac{n^2-2n-3}{n^2-2n-1} < \tfrac12.
\end{eqnarray*}
Hence, in both the even and the odd case, 
$q^{E}_\cA(G) < \frac12$, and so $q_\cA(G)<0$ (since $q^D_{\cB}(G)\geq 1/2$ for any bipartition $\cB$). It now follows by Lemma~\ref{lem.pZero} that $\q(G)=0$, as required.
\end{proof}

 Let us say that a bipartition $\{A,B\}$ is \emph{balanced} if $||A|-|B||\leq 1$. Lemma~\ref{lem.nearly_balanced} establishes that, if $k\leq \lfloor n/2\rfloor+1$, then any~$k$-edge bipartite graph $H$ on $[n]$ has a balanced bipartition with all edges between parts, except possibly for the case when $n$ is even and $k=n/2+1$. 
 Below, we further consider this last case, and show the following.

\begin{prop}\label{prop.unbalancedmeans1star}
Let $n \geq 4$ be even, let $k=n/2+1$, and let $H$ be a $k$-edge bipartite graph on~$[n]$.  Then there is no balanced bipartition of~$[n]$ such that the edges of $H$ lie between the parts if and only if $H$ less any isolated vertices and edges is a star with at least $3$ edges.    
\end{prop}

\begin{proof} Suppose first that $H$ less any isolated vertices and edges is a star with at least 3 edges.  Let $(A,B)$ be any bipartition of $[n]$ with $E(H) \subseteq E(A,B)$, and suppose wlog that the centre of the star is in $B$.  Then each vertex in $A$ has degree at most 1, so $|A| \geq k$, and the bipartition is not balanced.

Now suppose that there is no balanced bipartition as required.  Let $(A,B)$ be a bipartition of~$[n]$ such that the edges of $H$ lie between the parts, with $|A| >n/2$, and with $|A|$ as small as possible.  Then there are no isolated vertices $v$ in $A$ (since we could move $v$ to $A$).  Thus, since $|A|=k=|E(H)|$, each vertex in $A$ has degree~1.  
Hence, the graph $H$ is a disjoint union of proper-stars, edges and isolated vertices; where there is at least one proper-star, and the centres of the stars and the isolated vertices are in $B$. (Here, we say `proper-star' if the star has at least $2$ edges.) 

It remains to show that there is at most one proper-star, and then that this proper-star has $\geq 3$ edges.  
Suppose for a contradiction that there are at least 2 proper stars, pick a smallest one $S$, and swap sides for its vertices, to obtain the amended bipartition $(A',B')$.
Then $|A'| \leq |A|-1 = n/2$, 
and $|A'|$ is at least the number of non-isolated vertices in $B'$.  Hence, we can redistribute the isolated vertices of $H$ to obtain a balanced bipartition for $H$. This contradiction
shows there is at most one proper-star. %
Finally, simply note that if there is one proper-star and it has $2$ edges, then there is balanced bipartition, which completes the proof.
\end{proof}

\begin{rem}\label{rem.was_part_a} Proposition~\ref{prop.detmod0b} also implies the following. Let $0 \leq k \leq \lfloor n/2 \rfloor +1$, and let the  graph $G$ 
be obtained from $K_n$ by deleting the edges of a $k$-edge bipartite graph $H$ on $[n]$. Then $\delta^-(G)=\lfloor n/2\rfloor+1-k$. We may see this since if $k=\lfloor n/2\rfloor+1$ then $\q(G)>0$; and if we delete $k \leq \lfloor n/2\rfloor$ edges of a bipartite graph between $A$ and $B$ with $||A|-|B||\leq 1$, we may delete a further $\lfloor n/2\rfloor + 1 - k$ edges from between $A$ and $B$ and apply Claim~\ref{clm.upperbound}.
\end{rem}

\subsection{Local graph modifications}
We shall use the following lemma in the proof of part (a) of Theorem~\ref{thm.bip_delta_all}. Note that in the lemma below an edge between the parts is removed or an edge within a part added, whilst in the inequality~\eqref{eqn.mono} above an edge from between the parts is moved to within the parts.

\smallskip
\begin{lemma}\label{lem.twoparts_removeedge_or_addedge} 
Let $G$ be a graph and let $\cA=\{A, \overline{A}\}$ be a bipartition with $q_{\cA}(G) \geq 0$.
\begin{itemize}
    \item[(i)] 
    If $m \geq 2$ and $G^-$ is obtained by removing an edge between the parts, then $\, q_\cA(G^-)> q_\cA(G)$. 
    \item[(ii)] If $e_G(A)>0$ and $G^+$ is obtained by adding an edge within $\overline{A}$, then $q_\cA(G^+) > 0$. 
    \end{itemize}
\end{lemma}

\begin{proof}
The proof is almost immediate from~(\ref{eqn.pG2}), which shows that
\begin{equation}\label{eq:two_parts} q_\cA(G) = \frac{1}{2m^2}(4e_G(A)e_G(\overline{A})-e_G(A, \overline{A})^2).
\end{equation}
To prove part (i), note that in $G^-$ the quantities $e(A)$ and $e(\overline{A})$ are unchanged from $G$, but $e(A, \overline{A})$ has decreased by one. Thus by~\eqref{eq:two_parts} it follows that $q_\cA(G^-)>q_\cA(G) \: (\geq 0)$, as claimed. 

Part (ii) follows similarly. In this case, $e(A, \overline{A})$ and $e(A)>0$ are the same in $G$ and $G^+$, but $e(\overline{A})$ has increased by one, so by~\eqref{eq:two_parts} we may conclude $q_\cA(G^+)>0$.
\end{proof}

Given an $n$-vertex graph $G$, by adding $(\binom{n}{2} - \frac{n}{2} -e(G))^+$ edges we may form a graph $G_0$ with at most~$\frac{n}2$ edges missing, and then $\q(G_0)=0$ since $f^q(n) >n/2$ by Theorem~\ref{thm.detmod0}.  Hence, by Lemma~\ref{lem.edgeSim} (the robustness lemma), we have the following result, which extends the bound from Theorem~\ref{thm.detmod0} that $f^q(n) >n/2$. The notation $x^+$ means $\max \{x, 0\}$.

\smallskip

\begin{lemma}
\label{lem.qleq}
For each $n \geq 2$ and each nonempty $n$-vertex graph $G$
\[ \q(G) \leq \frac{2(\binom{n}{2} - \frac{n}{2} - e(G))^+}{e(G)}. \]
\end{lemma}

We will use Lemma~\ref{lem.qleq} to prove the upper bound in Theorem~\ref{thm.Gnp} (b).

\subsection{Nearly complete bipartite graphs: proof of Theorem~\ref{thm.bip_delta_all} }
\label{subsec.cmg}

We commented earlier that any complete multipartite graph $G$ satisfies $\q(G)=0$, but there is no robustness result for complete multipartite graphs corresponding to Theorem~\ref{thm.detmod0}.
In this section we prove the three parts of Theorem~\ref{thm.bip_delta_all} on $\delta^-(G)$, $\delta^+(G)$ and $\delta(G)$ for complete bipartite graphs~$G$, in  Lemmas~\ref{lem.bip_delta_minus},~\ref{lem.bip_delta_plus} and~\ref{lem.bip_delta}. In Section~\ref{subsec.multi} we briefly consider complete multipartite graphs.

We start with $\delta^-$, and see in Lemma~\ref{lem.bip_delta_minus} that $\delta^-(G)=1$ for all complete bipartite graphs $G$, with the exception of stars.

\smallskip

\begin{lemma}
[part (a) of Theorem~\ref{thm.bip_delta_all}]\label{lem.bip_delta_minus}
Let $G=K_{s,t}$ be a complete bipartite graph with $s \leq t$.  If $s=1$ then $\delta^-(G)=\infty$; and if $s \geq 2$ then $\delta^-(G)=1$.
\end{lemma}
\begin{proof}
Suppose first that $s=1$, so $G$ is the star $K_{1,t}$. As we delete edges we get stars $K_{1,t'}$ for $t'< t$ (together with isolated vertices) and eventually an empty graph. All these graphs have modularity zero and hence $\delta^-(K_{1,t})=\infty$.

From now on let $s \geq 2$. Note that $\q(K_{s,t})=0$ (shown in~\cite{bolla2015spectral, majstorovic2014note}, see also Section~\ref{sec.multi0}). Thus we must show that removing an edge from $K_{s,t}$ yields a graph with positive modularity value. (Recall by Lemma~\ref{lem.pZero} that $\q(G)>0$ if and only if there is a bipartition $\cA$ with $q_\cA(G)>0$, and hence $\q(K_{s,t})=0$ follows also from the claim on line~\eqref{eqn.complete_bipartite_mod} below.)

There are two cases. We first consider the case that $s$ and $t$ have a common factor and show that in this case there is a bipartition $\AA=\{A, \overline{A}\}$ with $q_\AA(K_{s,t})=0$ and $e(A, \overline{A})\geq 1$. Thus by Lemma~\ref{lem.twoparts_removeedge_or_addedge} part (i), for $K_{s,t}^-$ the graph obtained from $K_{s,t}$ by removing an edge between parts $A$ and $\overline{A}$ we have  $q_\AA(K_{s,t}^-)>0$. After that we consider the case when $s$ and $t$ are coprime which is only a little more complicated.

Before splitting into cases we give an expression for the modularity score of $K_{s,t}$ for bipartitions which will be useful for both cases. Write $S$, $T$ for the two parts of the bipartite graph of sizes $s$,~$t$ respectively, and let $\sigma=|A \cap S|/s$ and $\tau=|A \cap T|/t$ be the proportions of the vertex set $A$ in each of the bipartite parts. 
We \emph{claim} that for $\cA=\{ A, \overline{A} \}$
\begin{equation}\label{eqn.complete_bipartite_mod}
    q_{\cA}(K_{s,t})=-\tfrac{1}{2}(\sigma - \tau)^2.
\end{equation}

\emph{Proof of Claim.} The proof is by direct calculation. Note that 
\[e(A)=|A\cap S| \,|A\cap T|=\sigma \tau st, \] 
similarly, $e(\overline{A})=(1-\sigma)(1-\tau)st$, and thus
\begin{equation}\label{eqn.bipartite_edge_contribution}
    q^E_{\cA}(K_{s,t})=\sigma\tau+(1-\sigma)(1-\tau)=1-(\sigma+\tau)+2\sigma\tau.
\end{equation}
Also each vertex in part $S$ (resp. $T$) has degree $t$ (resp. $s$) and so 
\[\vol(A)=|A\cap S|t+|A\cap T|s=(\sigma+\tau)st\,, \] 
which together with the corresponding expression for $\vol(\overline{A})$ yields 
\[
q^D_{\cA}(K_{s,t})=(\tfrac{1}{2}(\sigma+\tau))^2+(1-\tfrac{1}{2}(\sigma+\tau))^2=1-(\sigma+\tau)+\tfrac{1}{2}(\sigma+\tau)^2.
\]
Thus we have
\[q_{\cA}(K_{s,t})=2\sigma\tau-\tfrac{1}{2}(\sigma+\tau)^2=-\tfrac{1}{2}(\sigma-\tau)^2\]
which completes the proof of the claim. 

\bigskip

\needspace{6\baselineskip}
\textbf{Case 1: $s$ and $t$ have a common factor.}\\
Recall that by Lemma~\ref{lem.twoparts_removeedge_or_addedge} part (i), it is sufficient to find a bipartition $\cA=\{A,\overline{A}\}$ such that $e(A, \overline{A}) \geq 1$ and $q_\cA(K_{s,t})=0$. Let $\ell>1$ be a common factor of $s$ and $t$. Then take $A$ to have $|S|/\ell$ vertices from $S$ and $|T|/\ell$ vertices from $T$. Clearly $e(A, \overline{A})\geq 1$; and since $\sigma=|A\cap S|/s=1/\ell = |A\cap T|/t = \tau$, by the claim on line~\eqref{eqn.complete_bipartite_mod} we have $q_\cA(K_{s,t})=0$ and we are done.

\bigskip

\textbf{Case 2: $s$ and $t$ are co-prime.}\\
The following relation between graphs $G$, $\Gminus$ and the modularities will be useful. Let $\AA=\{A, \overline{A}\}$ be a vertex bipartition of $G$ such that $e(A, \overline{A})\geq 1$ and let $\Gminus$ be any graph formed from $G$ by removing an edge between the parts $A$ and $\overline{A}$. Then
\begin{equation}\label{eqn.Gminus}
    e(\Gminus)^2q_\AA(\Gminus) = e(G)^2 q_\AA(G) - e(G) q^E_\AA(G)+e(G)-1/2.
\end{equation}
To see this note that because the removed edge is between parts in $\AA$, $e(\Gminus) q_\AA^E(\Gminus)=e(G)q_\AA^E(G)$. Similarly one may check that $e(\Gminus)^2q_\AA^D(\Gminus)=e(G)^2q_\AA^D(G)-e(G)+1/2$ and that these together give~\eqref{eqn.Gminus}.

Applied to our graph $G=K_{s,t}$ and $G^-= K_{s,t}^-$ obtained by removing an edge from between parts $A$ and~$\overline{A}$, we have
\begin{equation}\label{eqn.Kstminus}
    (st-1)^2q_\AA(K_{s,t}^-) = s^2t^2 \cdot q_\AA(K_{s,t}) - st\cdot (q^E_\AA(K_{s,t})-1)-1/2.
\end{equation}
Note from \eqref{eqn.bipartite_edge_contribution}
\begin{equation*}
    q^E_{\cA}(K_{s,t})-1
    =-(\sigma+\tau)+2\sigma\tau
    =-(\sigma-\tau)^2-\sigma(1-\sigma)-\tau(1-\tau).
\end{equation*}
Hence, substituting this and the expression for the modularity score in~\eqref{eqn.complete_bipartite_mod} we have
\begin{equation}\label{eqn.Kstminus2}
    (st-1)^2q_\AA(K_{s,t}^-) 
    = - st(\sigma-\tau)^2(\tfrac{1}{2}st - 1)
      + st\big(\sigma(1-\sigma)+\tau(1-\tau)\big)   
    -1/2.
\end{equation}

Observe that since $s$ and $t$ are co-prime the linear diophantine equation $at - bs=1$ has solutions $(a,b)$ with $a$ and $b$ integers. Note also that if $(a,b)$ is a solution then $(a+s, b+t)$ is also a solution. Hence there is a solution $(a',b')$ with $1 \leq a' \leq s-1$ (we cannot have $a'=0$). 
Thus $b's=a't-1$ satisfies $t-1 \leq b's \leq (s-1)t -1$, and so $1 \leq b' \leq t-1$. 

Our construction is to take $A$ with $a'$ vertices in~$S$ and $b'$ vertices in~$T$ then 
\begin{equation}\label{eqn.munuclose} |\sigma-\tau| = \left| \frac{|A \cap S|}{s} - \frac{|A \cap T|}{t} \right| = \bigg|\frac{a'}{s} - \frac{b'}{t}\bigg| = \frac{1}{st}.\end{equation}
Since $1\leq  a'\leq s-1$, $\sigma$ satisfies $1/s\leq \sigma \leq 1-1/s$ and hence $\sigma(1-\sigma)\geq 1/s^2$. Similarly $\tau(1-\tau)\geq 1/t^2$.

Hence by \eqref{eqn.Kstminus2}, substituting $|\sigma-\tau| = 1/(st)$, $\sigma(1-\sigma)\geq 1/s^2$ and $\tau(1-\tau)\geq 1/t^2$ we get 
\begin{equation}\label{eqn.Kstminus3}
    (st-1)^2q_\AA(K_{s,t}^-) 
    \geq \frac{1}{st}
      + \frac{t}{s} + \frac{s}{t} -1 >0
\end{equation}
and we are done.
\end{proof}

\smallskip

\needspace{4\baselineskip}
\begin{lemma}[part (b) of Theorem~\ref{thm.bip_delta_all}] \label{lem.bip_delta_plus}
Let $G$ be a complete bipartite graph $K_{s,t}$ with $s \leq t$. 
If $s=1$ and $t \geq 4$, or if $s \geq 2$ and $t \geq 3$, then $\delta^+(G)=1$.  In the other cases (namely when $(s,t) = (1,1), (1,2), (1,3)$ or $(2,2)$), we have $\delta^+(G)=\infty$.
\end{lemma}

\begin{proof}
Let $s=1$ and form $G^+$ by adding an edge to the star $K_{1,t}$. If $t$ is 1 or 2 then $\q(G^+)=0$ since $G^+$ has at most 3 vertices. If $t=3$ then $G^+$ is $K_4$ less 2 edges, so $\q(G^+)=0$ by Theorem~\ref{thm.detmod0}.
But for all $t \geq 4$
\begin{equation} \label{eqn.add_edge_to_star}
 \q(G^+) >0\,. 
\end{equation}
To see this, suppose that the added edge $e$ is $\{a,b\}$, and let $\cA$ be the partition of the vertex set $V$ into $\{a,b\}$ and $V \setminus \{a,b\}$.  Then 
\[q_\cA^E(G^+) = \frac{t-1}{t+1}\]
and
\[q_\cA^D(G^+)= \frac{4^2+(2t-2)^2}{4(t+1)^2} = \frac{4+(t-1)^2}{(t+1)^2}\,.\]
Thus
\[(t+1)^2 q_\cA(G^+) = (t^2 -1) -4-(t^2-2t+1) = 2t-6\,:\]
so $q_\cA(G^+)>0$ for $t \geq 4$. This completes the proof of~(\ref{eqn.add_edge_to_star}), and thus of the case $s=1$.

Suppose now that $s \geq 2$.  Note first that if $s=t=2$ (so $G$ is $K_{2,2}$) then at most 2 edges are missing from $K_4$, and thus $\delta^+(G)=\infty$ by Theorem~\ref{thm.detmod0}.   Now assume also that $t \geq 3$. It remains to show that $\delta^+(G)=1$.

The following relation between graphs $G$, $\Gplus$ and the modularity scores will be useful. Let $\AA=\{A, \overline{A}\}$ be a vertex bipartition of $G$ such that $e(A)<|A|(|A|-1)/2$ and let $\Gplus$ be any graph formed from $G$ by adding an edge within the part $A$. 
Since the extra edge is added inside a part in $\AA$, $e(\Gplus) q_\AA^E(\Gplus)=e(G)q_\AA^E(G)+1$. For the degree tax we may calculate
\begin{equation*}
 e(\Gplus)^2q_\AA^D(\Gplus)
 =\tfrac{1}{4}(\vol_{G}(A)+2)^2+\tfrac{1}{4}\vol_{G}(\overline{A})^2
 =e(G)^2q_\AA^D(G) + \vol_{G}(A)+1\,.
 \end{equation*}
Thus we get an expression similar to \eqref{eqn.Gminus}, 
\begin{eqnarray}
\notag    e(\Gplus)^2q_\AA(\Gplus) 
    & = & e(\Gplus)^2q^E_\AA(\Gplus) - e(\Gplus)^2q^D_\AA(\Gplus) \\
\notag    & = & (e(G)+1)(e(G)q^E_\AA(G)+1) - e(G)^2q_\AA^D(G) - \vol_{G}(A)-1 \\
\label{eqn.Gplus}    & = & e(G)^2q_\AA(G) + e(G) + e(G)q^E_\AA(G) - \vol_{G}(A).
\end{eqnarray}

Note that making the following substitutions for total edges : $e(G)=e_G(A)+e_G(A, \overline{A})+e_G(\overline{A})$, for internal edges : $\;e(G)q^E_\cA(G)=e_G(A)+e_G(\overline{A})$ and for volume of $A$ : $\vol_G(A)=2e_G(A)+e_G(A, \overline{A})$ we get
\begin{equation}\label{eqn.Gplus_fromzero} 
e(G^+)^2q_{\cA}(G^+) = e(G)^2 q_{\cA}(G)+ 2e_G(\overline{A}).
\end{equation}

\smallskip
\needspace{1\baselineskip}
\textbf{Case 1: $s$ and $t$ have a common factor.}\\
Let $\ell>1$ be a common factor of $s$ and $t$. Take $\overline{A}$ to be the part with $|S|/\ell$ vertices from $S$ and $|T|/\ell$ vertices from $T$ (we want to add an edge within the vertex set $A$, so we take $|A|\geq |\overline{A}|$). Then $|A \cap T|\geq 2$ since $t \geq 3$, and thus we may add an edge within part $A$ - call this new graph $K_{s,t}^+$. 

Similarly to in the proof of Lemma~\ref{lem.bip_delta_minus}, $\sigma=|A\cap S|/s=1-1/\ell = |A\cap T|/t = \tau$ and we have $q_{\cA}(K_{s,t})=0$ for the bipartition $\cA=\{A, \overline{A}\}$. Thus by \eqref{eqn.Gplus_fromzero}
\[ 
e(K_{s,t}^+)^2q_{\cA}(K_{s,t}^+) = 2e_{K_{s,t}}(\overline{A}).
\]
But now, since $\overline{A}$ has non-empty intersection with both $S$ and with $T$ we have $e_{K_{s,t}}(\overline{A})\geq 1$ and thus $q_{\cA}(K_{s,t}^+)>0$ as required.

\bigskip

\needspace{6\baselineskip}
\textbf{Case 2: $s$ and $t$ are co-prime.}\\
As in the proof of Lemma~\ref{lem.bip_delta_minus} we may assume that $a't-b's=1$ has a solution with $1\leq a'\leq s-1$ and $1\leq b'\leq t-1$ positive integers. Now, if either $a'>1$ or $b'>1$ then take $A$ with $a'$ vertices in~$S$ and $b'$ vertices in~$T$, otherwise take $A$ to be the complement of that, i.e. with $s-a'$ vertices in~$S$ and $t-b'$ vertices in~$T$. Let $\sigma=|A\cap S|/s$ and $\tau=|A \cap T|/t$. In either case, as in \eqref{eqn.munuclose}, \begin{equation}\label{eqn.munusmall_again} |\sigma - \tau|=1/(st).\end{equation} Notice also by our choice of $A$ that the set $A$ intersects either $S$ or $T$ in at least two vertices, and hence we have a place to add an edge within the set $A$, again we call the resulting graph $K_{s,t}^+$.

Now by~\eqref{eqn.Gplus_fromzero} for bipartition $\cA=\{A, \overline{A}\}$
\[
 e(K_{s,t}^+)^2q_{\cA}(K_{s,t}^+)
 =  e(K_{s,t})^2q_{\cA}(K_{s,t}) + e_{K_{s,t}}(\overline{A}) = -\tfrac{1}{2} +  e_{K_{s,t}}(\overline{A}) 
 \]
 where the second equality follows because by the claim in~\eqref{eqn.complete_bipartite_mod} $q_{\cA}(K_{s,t})=-\frac{1}{2}(\sigma-\tau)^2$, by~\eqref{eqn.munusmall_again} and because the total number of edges is $e(K_{s,t})=st$.
Again, by construction $\overline{A}$ has non-empty intersection both with $S$ and with $T$ and hence $e_{K_{s,t}}(\overline{A})\geq 1$, and $q_{\cA}(K_{s,t}^+) >0$ as required. 
\end{proof}
\smallskip

The above results show that for a complete bipartite graph $G$,
always $\delta^-(G)$ and $\delta^+(G)$ are 1 or~$\infty$: that is not quite the case for $\delta(G)$.

\smallskip

\begin{lemma}[part (c) of Theorem~\ref{thm.bip_delta_all}] \label{lem.bip_delta}
Let $G$ be a complete bipartite graph $K_{s,t}$ with $s \leq t$. 
If $s=1$ and $t \geq 4$, or if $s \geq 2$, then $\delta(G)=1$. 
 If $(s,t) = (1,1)$ or $ (1,2)$ then $\delta(G)=\infty$, and in the remaining case $(s,t)=(1,3)$ we have $\delta(G)=2$.
\end{lemma}
\begin{proof}
 If $s=1$ and $t \geq 4$, or if $s \geq 2$, then $1 \leq \delta(G) \leq \delta^+(G)=1$ by Lemma~\ref{lem.bip_delta_plus}; and if $s+t \leq 3$ then $\delta(G)=\infty$.  This leaves only the case $(s,t)=(1,3)$. 
When $G=K_{1,3}$ we saw that $\delta^-(G)=\delta^+(G)= \infty$, and so $\delta(G) \geq 2$.  But if $H$ is the 4-vertex path then $\q(H)= \frac16 >0$, and so $\delta(G)=2$. 
\end{proof}

\subsection{Nearly complete multipartite graphs - a partial story}
\label{subsec.multi}

We start this section by showing that complete multipartite graphs $G$ whose part sizes have a non-trivial common divisor have $\delta^-(G) = \delta^+(G)=\delta(G) =1$\,; though
some other complete multipartite graphs $G$ have~$\delta^+(G)=\infty$, see Example~\ref{ex.K2222}. Finally we give an example of a complete tripartite graph~$G$ with $\delta^+(G)=2$, see Example~\ref{ex.K133}. We do not aim here to give a complete story but rather give a few partial results and examples which indicate that for complete multipartite graphs the results will be a little more involved than for complete bipartite graphs.

\begin{prop} \label{prop.multi}
Let $G=K_{s_1, \ldots, s_k}$ be a complete multipartite graph where $k\geq 2$; and let $\ell | s_i$ for each $i$ for some integer $\ell \geq 2$. Then $\delta^-(G)= \delta^+(G)= \delta(G) =1$, except that $\delta^+(G)=\infty$ if each~$s_i=2$.  
\end{prop}
\begin{proof}
Denote the parts of the complete multipartite graph by $S_1, \ldots, S_k$. Construct the vertex set $A$ by including $|S_i|/\ell$ vertices from $S_i$ for each $i$.
We prove first that for the bipartition $\cA=\{A, \overline{A}\}$ we have $q_\cA(G) = 0$. %
Note that \[ e_G(A)=\sum_{i < j} \frac{|S_i|}{\ell}\frac{|S_j|}{\ell} = \frac{1}{\ell^2}e(G);\]
that is, part $A$ contains a $1/\ell^2$ proportion of the edges of $G$.  Similarly, part $\overline{A}$ contains a $(\ell-1)^2/\ell^2$ proportion of the edges; and thus $q_\cA^E(G)=(1+(\ell-1)^2)/\ell^2$. 
Also, part $A$ has a $1/\ell$ proportion of the volume of the graph, and part $\overline{A}$ has a $(\ell-1)/\ell$ proportion; and so $q^D_\cA(G)=q^E_\cA(G)$. Hence we have $q_\cA(G)=0$, as required.

Recall that $\q(G)=0$, by~\cite{bolla2015spectral,trajanovski2012maximum}.  Since $e_G(A, \overline{A})\geq 1$,
we can form $G^-$ by removing an edge between $A$ and $\overline{A}$; and by Lemma~\ref{lem.twoparts_removeedge_or_addedge} part (i), $\q(G^-)>0$ and so $\delta^-(G)=1$. It follows that $\delta(G)=1$.

Now consider $\delta^+(G)$.  
Observe that $e_G(A) \geq 1$.  Suppose first that some $s_i \neq 2$: then $|\overline{A} \cap S_i| \geq s_i/2>1$, so we can form $G^+$ by adding an edge within $\overline{A}$.
Now by Lemma~\ref{lem.twoparts_removeedge_or_addedge} part (ii), $\q(G^+)>0$ and so $\delta^+(G)=1$. 
On the other hand, if each $s_i=2$ and $n = \sum_i s_i = 2k$ then $G$ is $K_n$ less $n/2$ edges, so $\delta^+(G)= \infty$ by Theorem~\ref{thm.detmod0}. 
\end{proof}

\smallskip

\begin{example}\label{ex.K2222}
Let $j, k \geq 0$ with $j+k \geq 1$, and let $G=G_{j,k}$ be the complete multipartite graph  $K_{2,\ldots, 2, 1, \ldots, 1}$ with $j$ partite sets of size 2 and $k$ of size 1 (and so with $n=2j+k$ vertices), and note that $\q(G)=0$. 
We shall show that $\delta^+(G)=\infty$,
and $\delta(G)=\delta^-(G)= \lfloor k/2 \rfloor +1$ if $n \geq 4$. 

Consider $\delta^+(G)$. If $j=0$ it is not possible to add an edge. If $j\geq 1$ we may add $\ell=j$ edges to yield the complete graph $K_{2j+k}$. If $j\geq 2$ we may also add $\ell < j$ edges which yields the complete multipartite graph $G_{j-l,k+2\ell}$. 
However, all such graphs have modularity zero, and so $\delta^+(G)=\infty$ as claimed. 

Now let $n =2j+k \geq 4$ and consider $\delta^-(G)$. Note that $G$ is the complete graph $K_{n}$ less the bipartite graph formed by $j$ disjoint edges, and $j \leq n/2$.  Hence by Proposition~\ref{prop.detmod0b} we have 
$\delta^-(G) = \lfloor n/2 \rfloor +1-j = \lfloor k/2 \rfloor +1$ if $n \geq 5$, and the case $n=4$ is easily checked.
\end{example}

\smallskip

\begin{figure}[t]
    \centering
    \includegraphics[scale=0.8]{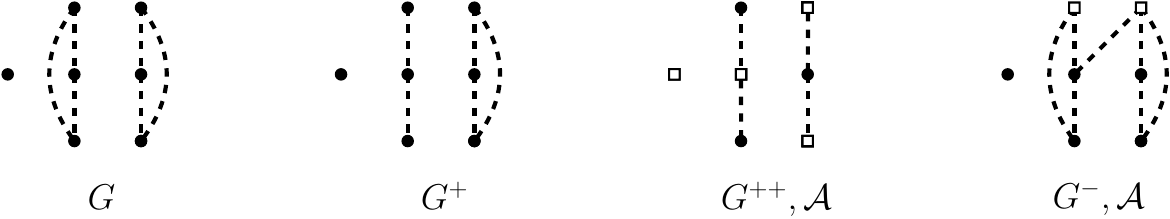}
    \caption{An illustration of the graphs considered in Example~\ref{ex.K133}, the dotted lines indicate the non-edges in the graph. For $G^{++}$ and $G^{-}$ the vertices are depicted with dots and squares to illustrate a bipartition which achieve positive modularity. }\label{fig.K133}
\end{figure}

\begin{example}\label{ex.K133}
Let $G$ be the complete tripartite graph $K_{1,3,3}$. We shall see that $\delta^+(G)=2$ and $\delta(G)=\delta^-(G)=1$. 

Note that $\q(G)=0$ and thus $\delta^+(G), \delta^-(G) \geq 1$. 
    
We first show that $\delta^+(G) \geq 2$ and then that $\delta^+(G)=2$. All graphs obtained by adding a single edge to $G$ are isomorphic: denote this 16-edge graph by $G^+$. Observe that $G^+$ is the complete graph $K_7$ less a disjoint triangle and two-edge path: a total of five missing edges. By Lemma~\ref{lem.pZero}, to show that $\q(G^+)=0$ it is enough to consider an arbitrary bipartition $\cA = \{A, \overline{A}\}$ for $G^+$, and show that $q_\cA(G^+)\leq 0$. Without loss of generality, $|A|\leq |\overline{A}|$.
Note that for every graph $H$ and every bipartition $\cB$ with a part of size 1, we have $q_{\cB}(H) \leq 0$.  Thus we may assume that $|A| \geq 2$, so $|A|$ is 2 or 3.

Of the five edges missing in $G^+$ at most four can be between the parts $A$ and $\overline{A}$, and thus $e_{G^+}(A, \overline{A})\geq |A||\overline{A}|-4$. Firstly, if $|A|=3$ and $|\overline{A}|=4$ then $e_{G^+}(A, \overline{A})\geq |A||\overline{A}|-4=8$. Hence $q_\cA^E(G^+)\leq 8/16=1/2$, and since the degree tax of any bipartition is always at least~$1/2$, in this case $q_\cA(G^+)\leq 0$.  

To show that $\delta^+(G) \geq 2$, it remains to check the case $|A|=2$ and $|\overline{A}|=5$, when $e_{G^+}(A, \overline{A})\geq 6$. If $e_{G^+}(A, \overline{A}) \geq 8$ then as in the case $|A|=3$ above we have $q_\cA(G^+) \leq 0$. If $e_{G^+}(A, \overline{A}) =7$ then $\vol_{G^+}(A)=2e_{G^+}(A)+e_{G^+}(A, \overline{A})\leq 2+7=9$. The degree tax is minimised by having volumes as equal as possible, so $q^D_\cA(G^+)\geq (9^2+23^2)/32^2$. But $q^E_\cA(G^+) = 1-7/16=9/16$ so $q_\cA(G^+)\leq 9/16-305/512<0$.  Finally, suppose that $e_{G^+}(A, \overline{A}) =6$.  Then arguing as before we see that $\vol_{G^+}(A) \leq 2+6=8$, so $q^D_\cA(G^+)\geq (8^2+24^2)/32^2     =5/8$.  But  $q^E_\cA(G^+) = 1- 6/16 = 5/8$, and so $q_\cA(G^+)=0$.

We have now seen that $\delta^+ (G) \geq 2$. To show that $\delta^+ (G)=2$, let $G^{++}$ be the 17 edge graph obtained from $K_{1,3,3}$ by adding two disjoint edges. Note that $G^{++}$  is the complete graph less two disjoint two-edge paths.
Let $\cA$ be a bipartition with part sizes $3$ and $4$ such that all four missing edges are between the parts, so $e_{G^{++}}(A, \overline{A}) = 8$. Then $q^E_\cA(G^{++}) =9/17$ and the degree tax is $q^D_\cA(G^{++})=(14^2+20^2)/34^2$ which gives $q_\cA(G^{++})>0$ and thus $\delta^+(G)=2$.

Finally, we show that $\delta^-(G)=1$ and thus $\delta(G)=1$.
Form $G^-$ from $G$ by deleting an edge, $ij$ say, between the partite sets of size~3. Then let $A=\{i,h\}$ where $h\neq j$ is in the same partite set as $j$. Note that the 7 `missing' edges in $G^-$ form two $K_3$'s joined by an edge, and 5 of these `missing' edges are between $A$ and $\overline{A}$ so there are 5 edges in $G^-$ between $A$ and $\overline{A}$. Thus for the bipartition $\cA=\{A, \overline{A}\}$ we have $q^E_\cA(G^-)=1-5/14=9/14$. 
For the degree tax of $\cA$, note  that $d_i=3$ and $d_h=4$ so $\vol(A)=7$, $\vol(\overline{A})=21$ and thus $q^D_{\cA}(G^-)=(7^2+21^2)/28^2=(1^2+3^2)/4^2=5/8$. Hence $q_\cA(G^-)=9/14-5/8=1/56 >0$, and so we have $\delta^-(G)=1$ as required. 
\end{example}


\needspace{8\baselineskip}
\section{Short proof that complete multipartite graphs have modularity zero}
\label{sec.multi0}

We noted that complete multipartite graphs have modularity zero~(\cite{majstorovic2014note,bolla2015spectral}).
The proofs given for this result are based on properties of the eigenvalues of related matrices. Here we give a shorter simple proof based on graph expansion.

\begin{thm}
  If $G$ is a complete multipartite graph then $\q(G)=0$.
\end{thm}
\begin{proof}
We use the following recent expansion result~\cite{modexp}: for every graph $G$
\[\q(G)=0 \;\; \mbox{ if and only if } \; \; e(A, \overline{A}) \geq 2\sqrt{e(A) e(\overline{A})} \;\;\;\; \forall A \subseteq V(G)\,.\]
Here $\overline{A} = V(G) \setminus A$, and recall that $\q(G)=0$ if $G$ has no edges.  

Fix a complete multipartite graph $G$ on parts $V_1, \ldots, V_k$ for some $k \geq 2$. Let $\emptyset \neq A \subsetneq V(G)$, let $x_i = |A \cap V_i|$ and $y_i=|\overline{A} \cap V_i|$; and let $x=\sum_i x_i=|A|$ and $y=\sum_i y_i=|\overline{A}|$. %
Observe that
$2e(A) = \sum_i x_i (x - x_i) = x^2 - \sum_i x_i^2$, similarly
$2e(\overline{A}) = y^2 - \sum_i y_i^2$, and
$e(A, \overline{A}) = \sum_i x_i (y-y_i) = xy - \sum_i x_i y_i$.
Thus $\, e(A, \overline{A}) \geq 2\sqrt{e(A) e(\overline{A})}\,$ if and only if
\begin{equation} \label{eqn.1}
  (xy - \sum_i x_i y_i)^2 \geq (x^2 - \sum_i x_i^2) (y^2 - \sum_i y_i^2).
\end{equation}
We shall prove the theorem by showing that~(\ref{eqn.1}) holds.
 Let $u_i = x_i /x$ so $\sum_i u_i = 1$, and let $v_i = y_i /y$ so $\sum_i v_i =1$.  Dividing both sides of~(\ref{eqn.1}) by $x^2 y^2$ shows that it is equivalent to
\begin{equation} \label{eqn.2}
 (1-\sum_i u_i v_i)^2 \geq(1-\sum_i u_i^2) (1-\sum_i v_i^2).
\end{equation}
Let $s= (\sum_i u_i^2)^{\frac12}$ and $t= (\sum_i v_i^2)^{\frac12}$, so $0<s,t \leq 1$.  
By the Cauchy-Schwarz inequality we have $\sum_i u_iv_i \leq st$.
Let $\alpha = (\sum_i u_iv_i)/st$, so $0< \alpha \leq 1$.
We may now write the inequality~(\ref{eqn.2}) as
\[ (1-\alpha st)^2 \geq (1-s^2)(1-t^2), \]
that is
\[ s^2 +t^2 - 2 \alpha st \geq (1-\alpha^2) s^2 t^2.\]
The LHS here equals
\[ s^2-2st +t^2 +2(1-\alpha)st = (s-t)^2 + 2(1-\alpha)st  \geq 2(1-\alpha)st. \]
Thus to prove~(\ref{eqn.1}) it suffices to show that
\begin{equation} \label{eqn.3}
2(1-\alpha)st \geq (1-\alpha^2) s^2 t^2.
\end{equation}
But both sides in~(\ref{eqn.3}) are non-negative (since $0<\alpha \leq 1$ and $st >0$), $st \geq s^2t^2$ (since $0 < s,t \leq 1$), and $2(1-\alpha) \geq (1-\alpha^2)$ (since $2(1-\alpha) - (1-\alpha^2) = (1-\alpha)^2 \geq 0$).
Hence~(\ref{eqn.3}) holds, and the proof is complete.
\end{proof}


\needspace{8\baselineskip}
\section{Proofs for results on random graphs}\label{sec.vdenserg}
Most of this section is taken up with proving Theorem~\ref{thm.Gnp}. After that we quickly prove Theorem~\ref{thm.from-modER}. 
\subsection{Proof of Theorem~\ref{thm.Gnp}}
Part (a). This part follows easily from Theorem~\ref{thm.detmod0}. 
Let $p=p(n)$ satisfy $p \geq 1- 1/n + \omega(n)\, n^{-3/2}$.
The number of edges missing in $\Gnp$ is at most $X$, where $X \sim \Bin(\binom{n}{2}, 1/n - \omega \, n^{-3/2})$.  Now $\E[X] \leq n/2 - \omega \sqrt{n}/2$, and $\var(X) \leq n/2$. 
Thus, by Chebyshev's inequality,
\[ \pr(X \geq n/2) \leq \pr(|X- \E[X]| \geq \omega \sqrt{n}/2) \leq \frac{n/2}{\omega^2 n/4}  = o(1). \] 
Hence $\q(\Gnp)=0$ whp by Theorem~\ref{thm.detmod0}. This completes the proof of Theorem~\ref{thm.Gnp}\,(a).
\smallskip

Part\,(b).  Let $p=p(n) = 1-1/n$.
For subpart (i), note that the number $X$ of edges missing in $\Gnp$ satisfies $X \sim \Bin(\binom{n}{2}, 1/n)$, and $X$ is asymptotically normal with mean and variance $n/2$ (see for example~\cite{bollobas2001random}).
Thus $\pr(X \leq n/2) = \frac12 +o(1)$, 
and so by Theorem~\ref{thm.detmod0} as above,  
$\pr(\q(\Gnp)=0) \geq \frac12 +o(1)$. 
It is now enough to appeal to subpart (ii) to conclude that we must have $\pr(\q(\Gnp)=0) = \frac12 +o(1)$. Hence we have finished the proof of (b) subpart (i), conditional on the proof of subpart (ii).

Now consider subpart (ii).  For $X$ as above, there is a constant $c>0$ such that $\pr(X > \lfloor n/2 \rfloor + c\sqrt{n}) < \eps/2$.
Let $\hat{G}$ be any $n$-vertex graph with $x$ edges missing, where $x \leq \lfloor n/2 \rfloor + c\sqrt{n}$;
and let $H$ be a graph obtained by deleting $\min\{x, \lfloor n/2 \rfloor\}$ of the missing edges of $\hat{G}$ from $K_n$. Then $\q(H)=0$ by Theorem~\ref{thm.detmod0}; and by Lemma~\ref{lem.edgeSim} (the robustness lemma),
\[\q(\hat{G}) = |\q(H)- \q(\hat{G})| \leq 2 (X- \lfloor n/2 \rfloor)^+ /e(H) \leq (1+o(1)) 2c\sqrt{n}/(n^2/2) < 5cn^{-3/2}
\]
for $n$ sufficiently large.  Now let $\beta =5c$, and observe that
\[\pr(\q(G_{n,p}) \geq \beta n^{-3/2}) \leq \pr(X > \lfloor n/2 \rfloor + c\sqrt{n}) < \eps/2.
\]
It remains to show that there is a constant $\alpha>0$ such that $\pr(\q(G_{n,p}) \leq \alpha n^{-3/2}) < \eps/2$. There is a constant $c_0 >0$ such that $\pr(X < n/2 + c_0 \sqrt{n}) < \frac12 + \eps/2$.  Note that $\pr(n/2+c_0 \sqrt{n} \leq X \leq n/2 + c \sqrt{n}) > \frac12- \eps$.
Let $m=m(n)$ satisfy $n/2+c_0 \sqrt{n} \leq m \leq n/2 + c \sqrt{n}$.  Consider the random graph $G_{n,m}$  and its complement $\overline{G}_{n,m}$.
It suffices to show that $\q(\overline{G}_{n,m}) \geq (c_0+o(1)) n^{-3/2}$ whp.

Let us recall some definitions.
If the graph $G$ has $n$ vertices, $m$ edges and $k$ components, then the \emph{cyclomatic number} (or cycle rank) $\cyclo(G)$ is $m-n+k$, which is the least number of edges that can be deleted to yield a tree or forest.  Thus $\cyclo(G)$ is $0$ if $G$ is a tree, and is 1 if $G$ is a unicyclic graph.  Also, let $\isol(G)$ be the number of isolated vertices; and let $L_1(G)$ be the maximum order of a component.  We shall use one simple deterministic lemma concerning these quantities.
\begin{lemma} \label{lem.claim}
Suppose that the graph $G$ satisfies
\begin{equation} \label{eqn.cycLisol}
   L_1(G) \leq \isol(G) +1. 
\end{equation} 

Then $G$ has a balanced bipartition $A,B$ such that 
\begin{equation} \label{eqn.claimA}
    e_G(A)+e_G(B) \leq \cyclo(G).
\end{equation}
\end{lemma}
\begin{proof}
To prove the inequality~(\ref{eqn.claimA}), initially set $A=B=\emptyset$.  Consider in turn the components $C$ in $G$ other than the isolated vertices. Given $C$, delete a set $D$ of at most $\cyclo(C)$ edges to obtain a connected bipartite subgraph, with (vertex) parts $C_1,C_2$.  Add $C_1$ to $A$ and $C_2$ to $B$, or $C_1$ to $B$ and $C_2$ to $A$, so as to keep the maximum of $|A|$ and $|B|$ as small as possible. Observe that the edges in $D$ and no other edges are added to $e(A) \cup e(B)$.
At the end of this part of the process, $||A|-|B|| \leq L_1$ and $e_G(A)+e_G(B) \leq \cyclo(G)$.
Now consider the (many) isolated vertices, and add them in turn to a smaller of $A$ and $B$.
Finally we have $||A|-|B|| \leq 1$, which completes the proof of~(\ref{eqn.claimA}).
\end{proof}

Now let us return to the random graph  $G=G_{n,m}$.
Whp it satisfies (1) $\cyclo(G) = O(\log n)$, (2) $L_1(G) \leq n^{2/3 +o(1)}$,
and (3) $\isol(G) = (1/e +o(1))\, n$.  In more detail we have the following.  Suppose that $0 \leq s \leq n^{2/3}$ and  $m=n/2 +s$.
Then $\cyclo(G_{n,m}) = (\frac23 +o(1)) \log n$ whp, by~\cite{luczakcomponent}
Theorem~10\!~(i); and $L_1(G_{n,m}) \leq \omega \, n^{2/3}$ whp, see for example~\cite{JLRbook} Theorem~5.23. 
Finally, $\isol(G_{n,m})$ is asymptotically Po$(n/e)$, see for example~\cite{bollobas2001random}. 

Hence whp (\ref{eqn.cycLisol}) holds for $G$, and so there is a balanced bipartition $\cA=(A,B)$ such that the inequality~\eqref{eqn.claimA} holds.  We use this partition for the complementary graph $\overline{G}=\overline{G}_{n,m}$.
Now
\begin{eqnarray*}
q_\cA^E (\overline{G}) 
& \geq &
\left(\binom{|A|}{2} + \binom{|B|}{2} - \cyclo(G) \right)/ m\\
& \geq &
(\tfrac14 n(n-2) - O(\log n))/(\tfrac12 n(n-2) -c_0 \sqrt{n})\\
&=&
\tfrac12 + (\tfrac12 c_0 \sqrt{n}- O(\log n))/(\tfrac12 n(n-2) -c_0 \sqrt{n})\\
&=&
\tfrac12 + c_0n^{-3/2} - O(n^{-2} \log n).
\end{eqnarray*}
For the degree tax, note that if $n$ is odd then
$2\big(\binom{(n+1)/2}{2} - \binom{(n-1)/2}{2}\big) = n-1$, and so by~(\ref{eqn.claimA})
\[ |\vol_{\overline{G}}(A)-\vol_{\overline{G}}(B)| \leq n-1 + 2 \cyclo(G) \leq n+2m \leq 3n. \]
Letting $\nu = \vol(\overline{G})$, we have
\[ q_{\cA}^D(\overline{G}) \leq \nu^{-2} ( ( \frac{\nu}{2} + \tfrac32 n)^2  + (\frac{\nu}{2} - \tfrac32 n)^2 ) =  \nu^{-2} ( \frac{\nu^2}{2} + \tfrac92 n^2) = \tfrac12 + O(n^{-2}).
\]
Hence $q_{\cA}(\overline{G}) \geq c_0n^{-3/2} -O(n^{-2} \log n)$, which completes the proof of part (b).

\bigskip \bigskip

Part\,(c). In order to prove part\,(c) of Theorem~\ref{thm.Gnp} we start with three deterministic lemmas, Lemmas~\ref{lem.qD} and~\ref{lem.excess} which we will apply to the complementary graph, together with Lemma~\ref{lem.smallcomps} which will be used to prove Lemma~\ref{lem.excess}.  We shall need here only the case $k=2$ of Lemma~\ref{lem.qD}.

\smallskip

\begin{lemma} \label{lem.qD}
Let $k \geq 2$ be fixed, and let the $n$-vertex graph $G$ satisfy
$e(G) = o(n^{3/2})$.
Let $\cA=(A_1,\ldots,A_k)$ be a partition of $V(G)$ such that $\max_i |A_i|  \leq n/k \, +o(\sqrt{n})$. Then for the complementary graph $\overline{G}$ we have  $q_{\cA}^D(\overline{G})= \frac1{k} + o(1/n)$.
\end{lemma}
\begin{proof}
For each $i$
\[ \vol_{\overline{G}}(A_i) \leq (\frac{n}{k}+ o(\sqrt{n}))(n-1) = \frac2{k} \binom{n}{2} + o(n^{3/2}) =  \frac2{k} e(\overline{G}) +x \]
where 
$x =  \frac2{k}e(G) + o(n^{3/2}) = o(n^{3/2})$.  Hence, by the convexity of $f(x)=x^2$,
\begin{eqnarray*}
\sum_{i=1}^{k}\vol_{\overline{G}}(A_i)^2 & \leq &
(k-1) (\frac2{k} e(\overline{G}) +x)^2 + (\frac2{k} e(\overline{G}) -(k-1)x)^2\\
& = &
\frac4{k} e(\overline{G})^2 + k(k-1) x^2 \; = \; \frac4{k} e(\overline{G})^2 + o(n^{3}),
\end{eqnarray*}
and so $q_{\cA}^D(\overline{G}) \leq \frac1{k} + o(n^{-1})$.  Also
$q_{\cA}^D(\overline{G}) \geq \frac1{k}$ since $\cA$ has $k$ parts, which completes the proof. 
\end{proof}

\smallskip

\needspace{6\baselineskip}
\begin{lemma} \label{lem.smallcomps}
Let $F$ be an $n$-vertex forest with maximum degree at most $d$, and let $\rho \geq 1/n$.
Then we can find a set of less than $d/\rho$ edges such that deleting these edges yields a forest in which each component tree has at most $\rho n$ vertices. 
\end{lemma}
\begin{proof} 
Let $x$ be the total number of vertices in components with more than $\rho n$ vertices.  Thus $x \leq n$, and we may assume that $x>0$.
Let the tree $T$ be a component with more than $\rho n$ vertices.
Pick a leaf $r$ in $T$, and  orient the edges away from $r$.  Call an edge $uv$ \emph{open} if when we delete it the component containing the terminal vertex $v$ has order more than $\rho n$.  Start at $r$ and follow a path of open edges until we reach a vertex $z$ such that no edge out of $z$ is open.   
Remove from $T$ all the edges incident with $z$. This removes at most $d$ edges, and decreases $x$ by more than $\rho n$.  We need to repeat this less than $1/\rho$ times, and so we delete in total less than $d/\rho$ edges.
\end{proof}
Let $G$ be a graph with $n$ vertices and $m$ edges. The \emph{excess} of $G$ is $m-n$. The \emph{core} of $G$ is the graph obtained by repeatedly deleting leaves, and (if it is non-empty) it is the unique maximal subgraph of~$G$ with minimum degree at least~2.  We call a bipartition $\cA=(A,B)$ \emph{balanced} when~$||A|-|B|| \leq 1$.

\smallskip

\begin{lemma} \label{lem.excess}
Let $G$ be an $n$-vertex graph with maximum degree at most $d$, let $\rho \geq 1/n$, and let~$G$ have at least $\rho n - 1$ isolated vertices; and suppose that the core of $G$ has excess at most~$s$, and has at most $j$ components.  Then there is a balanced bipartition $(A,B)$ of $V(G)$ such that $e(G) - e(A,B) < s+j+d/\rho$. 
\end{lemma}
\begin{proof}
By removing at most $s+j$ edges from $G$ we may reduce it to a forest $F$.  By Lemma~\ref{lem.smallcomps}, by removing less than $d/\rho$ further edges we may form a forest $F'$ where each (tree) component has at most $\rho n$ vertices.  We may find a proper 2-colouring $(A',B')$ of $F'$ less its isolated vertices such that $||A'|-|B'|| \leq \rho n -1$.  Now by adding the isolated vertices we can form a proper 2-colouring $(A,B)$ of $F'$ such that $||A|-|B|| \leq 1$. To complete the proof note that
$e(A,B) \geq e(F') > e(G) - (s+j+d/\rho)$.
\end{proof}

\smallskip

Now consider the random graph $G \sim G_{n,c/n}$.  (We shall be interested in the complementary graph~$\overline{G}$.)
For each $c \in [1,\infty)$ there is a unique $x=x(c) \in (0,1]$ such that $xe^{-x}=ce^{-c}$.  Here we are following Frieze and Karonski~\cite{frieze2015book}.
The function $x(c)$ is continuous and strictly decreasing on~$[1,\infty)$.

\smallskip

\begin{lemma}[Lemma~2.16 of~\cite{frieze2015book}]\label{lem.FK}
Let $c>1$ and let $x=x(c)$.  Then whp the core of $G_{n,c/n}$ has $(1-x)(1-x/c)n +o(n)$ vertices and $(1-x/c)^2 cn/2 +o(n)$ edges.
\end{lemma}

We now use the above lemmas to prove the lower bound in Theorem~\ref{thm.Gnp}\,(c) in the special case when the complementary graph is $G_{n,c/n}$ with $c$ small (precisely, when $1<c \leq 2$). 

\smallskip

\begin{lemma}\label{lem.bipdelta}
Let $1<c \leq 2$, and let $G \sim G_{n,c/n}$. Then there exists $\delta>0$ such that whp there is a balanced bipartition $\cA=(A,B)$ of $V(G)$ with
$q_{\cA}(\overline{G}) \geq \delta/n + o(1/n)$.
\end{lemma}
\begin{proof}
Let $x=x(c)$ as above. We shall see that we may take
\begin{equation} \label{eqn.delta}
\delta = \tfrac{1}{2c}((c-1)(2-c)+2(x-1)^2).
\end{equation}
It is easy to handle the degree tax for any balanced bipartition $\cA$.
Since $e(G) = cn/2 +o(n) = o(n^{3/2})$ whp, by Lemma~\ref{lem.qD} we have $q_{\cA}^D(\overline{G})= \tfrac12+o(1/n)$ whp.  We need to work to find $\cA$ with sufficiently high edge-contribution $q_{\cA}^E(\overline{G})$.

By Lemma~\ref{lem.FK}, the core of $G$ whp has $(1-x) (1-x/c)n +o(n)$ vertices and $(1-x/c)^2 cn/2 +o(n)$ edges. So whp the excess of the core is
\begin{eqnarray*}
(1-\tfrac{x}{c})^2 cn/2 - (1-x) (1-\tfrac{x}{c})n +o(n)
& = &
(1-\tfrac{x}{c})(\tfrac{c}{2} - \tfrac{x}{2} -(1-x))n +o(n)\\
& = &
\tfrac12 (1-\tfrac{x}{c})(c + x -2)n +o(n).
\end{eqnarray*}
(It is easy to check that $c+x>2$.) Also, whp the maximum degree in $G$ is $O(\log n / \log\log n)$, there is exactly one complex component (with more than one cycle), there are $o(n)$ (actually many fewer) unicyclic components, and there are $\Omega(n)$ isolated vertices~\cite{frieze2015book}. Thus, by Lemma~\ref{lem.excess}, whp there is a balanced bipartition $(A,B)$ of $V(G)$ such that 
\[e(G) -  e_G(A,B) \leq \tfrac12 (1-\tfrac{x}{c})(c + x -2)n +o(n).\]
But $e(G) = \tfrac12 cn +o(n)$ whp, so whp
\begin{eqnarray*}
e_G(A,B)  & \geq &
\tfrac12 cn -  \tfrac12 (1-\tfrac{x}{c})(c + x -2)n +o(n)\\
& = &
\tfrac{n}{2} ( c - (c+x-2 -x-\tfrac{x^2}{c} + \tfrac{2x}{c})) +o(n)\\
& = &
( 1- \tfrac{x(2-x)}{2c})\, n +o(n).
\end{eqnarray*}
Since $|A| |B| \leq n^2/4$, we may see that
\begin{eqnarray*}
q_{\cA}^E(\overline{G}) & \geq &
1 - \frac{\frac{n^2}{4} - e_G(A,B)}{e(\overline{G})}\\
&=&
1- \frac{\frac12 e(\overline{G}) +  (\frac{n^2}{4} - \frac12 e(\overline{G})   - e_G(A,B))}{e(\overline{G})}\\
&=&
\tfrac1{2} - \frac{ \frac{n^2}{4} - \frac12 \binom{n}{2} - (e_G(A,B) - \frac1{2} e(G))}{e(\overline{G})}\\
&=&
\tfrac12 +  \frac{e_G(A,B) - \frac12 e(G)- \frac14n}{e(\overline{G})}.
\end{eqnarray*}
But by the above, whp
\[ e_G(A,B) - \tfrac12 e(G)- \tfrac14 n \geq \delta n/2 + o(n) \]
where
\[ \delta =  2- \tfrac{x(2-x)}{c}  - \tfrac12 (1+c) = \tfrac32 - \tfrac{x(2-x)}{c} - \tfrac12 c.\]
Thus $q_{\cA}^E(\overline{G}) \geq  \frac12 + \tfrac{(\delta/2) n}{e(\overline{G})} + o(1/n)$.
Hence, recalling that $q_{\cA}^D(\overline{G})= \frac12+o(1/n)$, we have
\[ q_{\cA}(\overline{G}) \geq \tfrac{(\delta/2) n}{e(\overline{G})} + o(1/n) = \tfrac{\delta}{n} + o(1/n) \mbox{ whp}.\]
We may rewrite $\delta$ as
\[ \delta = \tfrac1{2c} (3c-2x(2-x)-c^2) = \tfrac{1}{2c}((c-1)(2-c)+2(x-1)^2)\]
as in~(\ref{eqn.delta}),
so clearly $\delta>0$ since $1<c \leq 2$; and this completes the proof.
\end{proof}

Observe from~(\ref{eqn.delta}) that if $c_1=1+\eps$ with $\eps$ small, then $\delta = \eps/2 + O(\eps^2)$ as $\eps \to 0$.

The next lemma uses Lemma~\ref{lem.bipdelta} and completes the proof of the lower bound in part (c) of Theorem~\ref{thm.Gnp}. Given a  graph $G$ and $0 \leq p \leq 1$, let the \emph{percolated} random graph $G_p$ be the subgraph of $G$ obtained by considering each edge independently and keeping it with probability~$p$ (and otherwise deleting it).  Thus the binomial random graph $\Gnp$ could be written as $(K_n)_p$, where $K_n$ is the $n$-vertex complete graph. 

\smallskip

\begin{lemma} \label{lem.thin}
Let $1< c \leq 2$ and let $\delta>0$ be as in~\eqref{eqn.delta}. Let the probability $p=p(n)$ satisfy $p \to 1$ as $n \to \infty$ and $p \leq 1-c/n$ for $n$ sufficiently large. Then $\q(\Gnp) \geq \delta/n +o(1/n)$ whp.
\end{lemma}
\begin{proof}
Consider large $n$.  Let $\rho=\rho(n)=p/(1-c/n)$.  We may form $\Gnp$ from $K_n$ in two steps.  First delete edges independently with probability $c/n$ to give $G_{n,1-c/n}$, and then keep the remaining edges independently with probability $\rho$. 
By Lemma~\ref{lem.bipdelta} there is a balanced bipartition $\cA$ of $[n]$ such that whp $q_{\cA}(G_{n,1-c/n}) \geq \delta/n +o(1/n)$.  Note that also $e(G_{n,1-c/n}) \geq \binom{n}{2} - cn$ whp.

Let $H$ be a graph on $[n]$ with $e(H) \geq \binom{n}{2} - cn$, and suppose that $\cA$ is a balanced bipartition of $[n]$ with $q_{\cA}(H) \geq \delta/n + o(1/n)$. It suffices to show that the random subgraph $H_{\rho}$ of $H$ must satisfy $q_{\cA}(H_{\rho}) \geq \delta/n +o(1/n)$ whp.
Note that $\rho \to 1$ as $n \to \infty$: let $\omega = \omega(n) \to \infty$ as $n \to \infty$ sufficiently slowly that $\omega \, (1- \rho)^{1/2} \to 0$.

Consider $e(H_{\rho})$, which has the binomial distribution $\Bin(e(H),\rho)$. By Chebyshev's inequality
\[ \pr\left(|e(H_{\rho})-e(H)\rho| \geq x (e(H) \rho(1-\rho))^{1/2}\right) \leq x^{-2}.\]
Thus whp
\begin{eqnarray*}
  e(H_{\rho}) 
&= &
  e(H) \rho + O(\omega \, (e(H) \rho (1-\rho))^{1/2})\\
&=&
  e(H) \rho \, (1+ O((\omega\, ((1-\rho)/e(H))^{1/2})\\
&=&
  e(H) \rho \, (1+ O(\omega\, (1-\rho)^{1/2}/n)),
\end{eqnarray*}
and so
\begin{equation} \label{eqn.eHp}
    e(H_{\rho})  = e(H) \rho \, (1+ o(1/n)) \;\; \mbox{ whp}.
\end{equation}
For a graph $H'$ on $[n]$ let $e^{\rm int}(H')=e^{\rm int}_\cA(H')$ denote the number of internal edges within the parts of $\cA$.
Observe that $e^{\rm int}(H)= (1+o(1)) n^2/4$, since $\cA$ is a balanced bipartition and at most $cn=o(n^2)$ edges are missing from $H$.  
Thus, much as above, we have
\[ e^{\rm int}(H_{\rho}) =  e^{\rm int}(H) \rho  \, (1+ o(1/n)) \;\; \mbox{ whp}.\]
Hence
\begin{equation} \label{eqn.qEHp}
 q_{\cA}^E (H_{\rho}) = 
\frac{e^{\rm int}(H_{\rho})}{e(H_{\rho})} =  q_{\cA}^E (H) (1+ o(1/n)) = q_{\cA}^E (H) + o(1/n) \; \mbox{ whp}.
\end{equation}

The degree tax is only a little more complicated.  Let the parts of $\cA$ be $A$ and $B$, and let $X= \vol_{H_{\rho}}(A)$.  Suppose that $H$ has $i$ edges within $A$ and $j$ edges between $A$ and $B$, so $\vol_H(A)=2i+j$. Then $X$ has variance
\[ \var(X) = (4i+j) \rho(1-\rho) \leq 2 \vol_H(A)\rho (1-\rho).\]
Hence, by Chebyshev's inequality as before, whp
\begin{eqnarray*}
  X & = & 
  \vol_H(A) \rho + O(\omega (\vol_H(A)\rho(1-\rho))^{1/2})\\
&=&
  \vol_H(A) \rho \, (1+ O((\omega ((1-\rho)/\vol_H(A))^{1/2})\\
&=&
  \vol_H(A) \rho \, (1+ O(\omega (1-\rho)^{1/2}/n)),
\end{eqnarray*}
since $\vol_H(A) = \Theta(n^2)$.
Thus
$\vol_{H_{\rho}}(A) = \vol_H(A) \rho \, (1+ o(1/n))$ whp; and similarly
$\vol_{H_{\rho}}(B) = \vol_H(B) \rho \, (1+ o(1/n))$ whp.  Hence, using also~(\ref{eqn.eHp}), whp
\begin{eqnarray*}
q_{\cA}^D(H_{\rho}) &=&
(\vol_{H_{\rho}}(A)^2 + \vol_{H_{\rho}}(B)^2)/\vol(H_{\rho})^2\\
&=&
q_{\cA}^D(H) (1+o(1/n)) \; = \; q_{\cA}^D(H) +o(1/n).
\end{eqnarray*} 
But now, using also~(\ref{eqn.qEHp}), whp
\begin{eqnarray*}
q_{\cA}(H_{\rho}) &=& q_{\cA}^E(H_{\rho}) - q_{\cA}^D(H_{\rho}) \; =\; q_{\cA}^E(H) - q_{\cA}^D(H) + o(1/n)\\
&=&
q_\cA(H) + o(1/n) \; \geq \; \delta/n + o(1/n)\,,
\end{eqnarray*}
which completes the proof of Lemma~\ref{lem.thin}.
\end{proof}

It remains in this subsection to prove the upper bound in Theorem~\ref{thm.Gnp} (c).
Let $c>0$ and let $p \geq 1-c/n$. Then whp $e(G_{n,p}) \geq \binom{n}{2} - (c+o(1)) n/2$, and so
$\binom{n}{2} - n/2-e(G_{n,p}) \leq (c -1 +o(1))\, n/2$.  Hence by Lemma~\ref{lem.qleq}
\[ \q(G_{n,p}) \leq \frac{(c-1+o(1))\, n}{e(G_{n,p})}  = 2(c-1+o(1))/n
\;\;\; \mbox{whp}.
\]
Thus we may set $\beta= 2 c_2$ to complete the proof.  (Indeed, given any $\eps>0$, we may set $\beta= 2(c_2-1) + \eps$, which shows that we may take $\beta$ small if $c_2$ is not much bigger than 1.)


\subsection{Proof of Theorem~\ref{thm.from-modER}}

 By Theorem~4.1 of~\cite{ERmod} there exist $\alpha_1>0$ and $c_0$ such that, if $p=p(n)$ satisfies $1/n \leq p \leq 1-c_0/n$ for $n$ sufficiently large, then $\q(G_{n,p}) \geq \alpha_1 \sqrt{\frac{1-p}{np}}$ whp.  We shall choose $0<\alpha \leq \alpha_1$, so we need not consider this range of~$p$ further.  We may assume that $c_0 \geq 2$, say.

Now let $\eps>0$, and let $p$ satisfy $1- c_0/n \leq p \leq 1-1/n - \eps/n$.  Then
\[\sqrt{\tfrac{1-p}{np}} \leq \sqrt{\tfrac{c_0/n}{n(1-c_0/n)}} = (\sqrt{c_0}+o(1/n)) /n.\]
By Theorem~\ref{thm.Gnp} there exists $\alpha_2>0$ such that $\q(G_{n,p}) \geq \alpha_2/n$ whp. Let $\alpha>0$ satisfy $\alpha \leq \alpha_1$ and $\alpha < \alpha_2 /\sqrt{c_0}$.  Then $\alpha_2/n \geq \alpha \sqrt{\frac{1-p}{np}}$ for $n$ sufficiently large, so $\q(G_{n,p}) \geq \alpha \sqrt{\frac{1-p}{np}}$ whp, as required.



\needspace{6\baselineskip}
\section{Variant forms of degree tax and when modularity is zero}
\label{sec.mod_var}

Let $G$ be a graph with $m \geq 1$ edges, and let $\cA$ be a partition of $V(G)$. In the definition of the modularity score $q_\cA(G)$, the degree tax is based on the model of a random pseudograph $R$ which is approximately uniformly distributed over the pseudographs on $V(G)$ with the same degrees as $G$, where $R$ is generated by the configuration model. Note that $R$ may have loops and multiple edges.
As we noted earlier, the expected number of edges in $R$ between distinct vertices $i$ and $j$ is $d_i d_j/(2m-1)$. 
If we do not worry too much about loops, this yields 
\begin{equation}
\label{eqn.Eint}
\E[e^{\rm int}_{\cA} (R)]/m  \approx \sum_A \sum_{i,j \in A} d_i d_j /(4m^2) =  (1/4m^2) \sum_A \vol(A)^2 =  q_{\cA}^D(G),
\end{equation} 
and so
\begin{equation}
\label{eqn.Eint2} q_\cA(G) = q_{\cA}^E(G) - q_{\cA}^D(G)  \approx (1/m) \left(e^{\rm int}_{\cA} (G) - \E[e^{\rm int}_{\cA} (R)] \right). \end{equation}
This approximation motivated the original definition of modularity~\cite{NewmanGirvan}.

A referee of an earlier version of this paper commented that allowing loops in the random pseudograph $R$ leads to bias in particular against partitions with small parts, and questioned how changing the definition of $R$ to exclude loops would change our results on very dense graphs.
In this (new) section we consider five natural variants of the Newman--Girvan definition of modularity, and we investigate which graphs have corresponding modularity value~0. 
Recall that `graph' here always means `simple graph'.

Given a graph~$G$ and a vertex-partition $\cA$, each modularity variant uses the original edge contribution $q_{\cA}^E(G)$, and has its own penalty (corresponding to the degree-tax); and the modularity of~$G$ is the maximum of the scores for each vertex-partition.
In each case, the modularity value lies in~$[0,1)$;
there are graphs with modularity 0 (for example, complete graphs); there are graphs with modularity arbitrarily close to 1 (for example, disjoint unions of many $k$-cliques for any $k \geq 2$); and isolated vertices are irrelevant, except for modified modularity and Erd\H{o}s-R\'enyi modularity.

In Subsection~\ref{subsec.variants} we define (precise) `configuration modularity' $cq$ (where $R$ is distributed exactly according to the configuration model),
`simple modularity' (or `exact modularity') $sq$ (where $R$ is simple),
`modified modularity' $\tilde{q}$ (a `local' version of simple modularity), 
and `loop-free modularity'~$l\!fq$  (where $R$ is loop-free), and finally we recall the well-established Erd\H{o}s-R\'enyi modularity $ERq$. We note that simple/exact modularity is considered independently in the recent paper~\cite{chang2025modularity}, and modified modularity is introduced in that paper. 
Corresponding to $f^{\rm q}(n)$, we let $f^{\rm cq}(n)$ be the least number of edges missing from the complete graph $K_n$ in a graph with strictly positive configuration modularity; and similarly for the other variants.  We characterise those graphs~$G$ with simple modularity value $sq^*(G)=0$, those with modified modularity value  $\tq(G)=0$,
and those with loop-free modularity value $l\!fq^*(G)=0$; and these results yield easily the values $f^{\rm sq}(n)$, $f^{\rm \tilde{q}}(n)$, and~$f^{\rm lfq}(n)$.
In Subsection~\ref{subsec.edit},
we complete the story concerning these values by determining $f^{\rm cq}(n)$ and $f^{\rm ERq}(n)$. See also the table in Figure~\ref{fig.mod_variants} on page~\pageref{fig.mod_variants}.

\needspace{5\baselineskip}
\subsection{Definitions of modularity and natural variants}
\label{subsec.variants}

Calculating the left hand side of~(\ref{eqn.Eint}) precisely, we have 
\begin{eqnarray*}
 \E[e^{\rm int}_{\cA} (R)] /m
 &=&
 \frac1{m} \sum_A \bigg(\sum_{i \neq j \in A} \frac{d_i d_j}{2(2m-1)} + \sum_{i \in A}  \frac{d_i(d_i-1)}{2(2m-1)} \bigg)\\
 &=&
   \frac1{m} \sum_A \bigg(\sum_{i, j \in A} \frac{d_i d_j}{2(2m-1)} - \sum_{i \in A}  \frac{d_i}{2(2m-1)} \bigg)\\
  &=&
  \frac1{2m (2m-1)} \sum_A ( \vol^2(A) - \vol(A)).
  \end{eqnarray*}
From the last equation,
\begin{equation} \label{eq.cf_deg_tax}
\E[e^{\rm int}_{\cA} (R)] /m =
    \sum_{A\in A} {\textstyle \binom{\vol(A)}{2}/\binom{2m}{2} } :
\end{equation}
we shall use this observation in equation~\eqref{defn.cq} below.
Continuing from the previous equations, since $\sum_A \vol(A)=2m$ we have
\[ \E[e^{\rm int}_{\cA} (R)] /m =
 \frac{4m}{2(2m-1)} \bigg( \frac1{4m^2} \sum_A \vol^2(A) - \frac1{2m}  \bigg) =
\frac{2m}{2m-1} \,(q_{\cA}^D(G) - \frac1{2m}). \]
Thus 
\begin{equation} \label{eqn.R0}
\E[e^{\rm int}_{\cA} (R)] /m    = q_{\cA}^D(G) -    \frac1{2m-1} (1-q_{\cA}^D(G)).
\end{equation}
But $0<q_{\cA}^D(G)\leq 1$ for each vertex-partition $\cA$, and so 
\begin{equation} \label{eqn.R}
q_{\cA}^D(G) - \frac1{2m-1} <  \E[e^{\rm int}_{\cA} (R)] /m \leq q_{\cA}^D(G),
\end{equation}
with equality in the second inequality if and only if $\cA$ has a single part (ignoring isolated vertices), in which case both sides equal 1.
The inequalities~(\ref{eqn.R}) yield a more precise version of the approximation~(\ref{eqn.Eint2}), see also 
(\ref{defn.cq}) below.
\medskip

\begin{figure}[t] 
\centering
\footnotesize
$
\renewcommand{\arraystretch}{1.2}
\begin{array}{|l l | ll | lll | }
\cline{1-7}\mbox{\;\;\textbf{variant}} && \;\; \mbox{\textbf{degree tax}} && \;\; &\;\; f(n) & \;\;\;\;\;\;\;\;\;\;\;\;\;\;\;\;\;\;  \\
\cline{1-7}  &&&&&&\\
q_\cA & \mbox{(standard)} &{\displaystyle \sum_A} \vol(A)^2/(2m)^2& \mbox{Def}~\eqref{def.mod} &&\!\!\lfloor n/2 \rfloor+1\; & \mbox{Thm}~\ref{thm.detmod0} \\
&&&&&&\\
cq_\cA & \mbox{configuration} & 
 {\displaystyle \sum_A} \binom{\vol(A)}{2}/\binom{2m}{2}
& \mbox{eqn}~\eqref{defn.cq}
&&\!\! \lceil n/2 \rceil &\mbox{Prop}~\ref{prop.det_config_mod0}\\ 
&&&&&&\\
sq_\cA & \mbox{simple} / \mbox{exact} & \E[e^{\rm int}_\cA(S)] \; \mbox{where $S$ is simple}
&\mbox{eqn}~\eqref{defn.sq} && 2 & \mbox{eqn}~\eqref{eq:f_simp_mod}\\
&&&&&&\\
\tilde{q}_\cA & \mbox{modified} / \mbox{local} & 
\mbox{(see section on~page~\pageref{subsec.mod_mod})} 
&\mbox{eqn}~\eqref{defn.mq} && 2 & \mbox{eqn}~\eqref{eq:f_mod_mod}\\
&&&&&&\\
\lfqA & \mbox{loop-free} & 
\E[e^{\rm int}_\cA(R_0)] \; \mbox{where $R_0$ is loop-free} 
& \mbox{eqn}~\eqref{defn.lfq} && 1 & \mbox{eqn}~\eqref{eq:f_loop_free_mod}\\
&&&&&&\\
ERq_\cA & \mbox{Erd\H{o}s-R\'enyi} & \displaystyle \sum_{A}|A|^2/n^2 & \mbox{eqn}~\eqref{defn.ERq} && 2 & \mbox{Prop}~\ref{prop.det_ER_mod0}\\
&&&&&&\\ \cline{1-7} 
\end{array}$\caption{\small  Table showing how variations in the modularity definition would change the results in this paper. For each modularity variant, we write $f(n)$ for the least number of edges missing from any $n$-vertex graph~$G$ with positive modularity value, and list these values (note that we may for example need $n \geq 4$ or similar) along with a reference to where to find the result in the paper.}\label{fig.mod_variants}
\end{figure}

\needspace{6\baselineskip}
{\em 
Configuration modularity $cq^*$}

Suppose that we define the (precise) \emph{configuration modularity score} $\cqA(G)$
to be the right hand side of~(\ref{eqn.Eint2}) (without using the approximation~(\ref{eqn.Eint}), though still using the configuration model to generate $R$). Then, by~\eqref{eq.cf_deg_tax}, 
\begin{equation}\label{defn.cq} \cqA(G) = (1/m)  (e^{\rm int}_{\cA}(G) - \E[e^{\rm int}_{\cA} (R)]) = e^{\rm int}_{\cA}(G)/m - \sum_{A\in \cA} {\textstyle \binom{\vol(A)}{2}/\binom{2m}{2} },
\end{equation}
and thus
\begin{equation}\label{eq.cq_score_higher} \cqA(G) = q_\cA(G) + (1-q^D_\cA(G))/(2m-1) \end{equation}
by~(\ref{eqn.R0}).
{Observe that isolated vertices are irrelevant.}
Also let the (precise) \emph{configuration modularity} $\cq(G)$ of~$G$ be the maximum value of $\cqA(G)$ over all vertex-partitions~$\cA$. By~(\ref{eq.cq_score_higher}), we have
\begin{equation*} \label{eqn.mq1}
 q_{\cA}(G) \leq \cqA(G) < q_{\cA}(G) + \frac1{2m-1}\,, 
\end{equation*}
and so
\begin{equation} \label{eqn.mq2new}
q^*(G) \leq \cq(G) < q^*(G) \, + \frac1{2m-1}\,.
\end{equation} 
Also, if $G_m$ denotes the graph consisting of $m$ isolated edges, then $\q(G) \leq \q(G_m) = 1-1/m$ (see~\cite[Cor~1.4.2]{thesis}), and so $0 \leq \cq(G) <1$ for all graphs~$G$. (If $m=1$ then $\cq(G)=0$.)
Further, if $\q(G)>0$, or $q_{\cA}(G)=0$ for some partition~$\cA$ 
with at least two parts of positive volume,
then there is a $\q$-optimal partition $\cA$ with $q_\cA^{D}(G)<1$, and so $\cq(G) > \q(G)$ by equation~\eqref{eq.cq_score_higher}.
For example, consider the complete graph $K_n$ where $n$ is even, and let~$G$ be this graph less a perfect matching (so~$G$ is the complement of the $n$-vertex 1-regular graph).  Then $\q(G)=0$ by Theorem~\ref{thm.detmod0}, and we saw in the proof of the theorem that there is a bipartition $\cA$ with $q_{\cA}(G)=0$; and hence $\cq(G) \geq \cqA(G) > 0$. See Proposition~\ref{prop.det_config_mod0} for the edit-distance $f^{cq}(n)$ from $K_n$ to a graph with positive configuration modularity. 

\medskip

\needspace{6\baselineskip}
{\em Simple modularity $sq^*$}

As we saw above, in the standard definition of the modularity score $q_\cA(G)$, the degree tax `penalty' is based on the expected number of internal edges for a random pseudograph $R$ which is nearly uniformly distributed over the graphs on $V(G)$ with the same degrees as~$G$, as generated by the configuration model. 
What if we insist on the random pseudograph $R$ not having loops or multiple edges (that is, on $R$ being a simple graph), and on $R$ being \emph{exactly} uniformly distributed on the graphs with the same degrees as~$G$? 
For a vertex-partition $\cA$ of~$G$, define the corresponding
`simple modularity’ (or `exact modularity') score $sq_{\cA}(G)$  by
\begin{equation}\label{defn.sq} \sqA(G) = (1/m) \left(e^{\rm int}_{\cA} (G) - \E[e^{\rm int}_{\cA} (R)] \right)= \frac1m \sum_{i<j} \left( {\mathbf 1}_{ij\in E(G)} - \pr(ij \in E(R)) \right) \delta_{ij}(\cA) ,  \end{equation}
where the simple graph $R$ is uniformly distributed on the graphs on $V(G)$ with the same degrees as~$G$.
(Recall that $\delta_{ij}(\cA)$ is the indicator that vertices $i$ and $j$ are in the same part in $\cA$.)

Let $\sq(G)$ be the maximum value of $sq_\cA(G)$ over all vertex-partitions $\cA$, {and}
observe that $0 \leq \sq(G)<1$. 
This seems like a natural variant of the definition of modularity, though now there are no simple general formulae like the degree-tax to help us to calculate or approximate a modularity score $sq_\cA(G)$.

Simple\,/\,exact modularity $sq$ is introduced also in Chang and Van Mieghem~\cite{chang2025modularity}, where it is called exact modularity. That paper notes (as above) the lack of a general formula to calculate the modularity score $sq_\cA(G)$. (The same paper also introduces a simplified `local' version of $\sq$ called `modified modularity' $\tq$ -- which we discuss next, see page~\pageref{subsec.mod_mod}). 
To get a feel for simple/exact modularity $sq$, consider first the easy case when~$G$ is $r$-regular, so  $r n = 2m$.
By symmetry, each edge appears in $R$ with probability $p= m/\binom{n}{2} = r/(n-1)$, so 
\[ \E[e^{\rm int}_{\cA} (R)] = \sum_A \binom{|A|}{2} p  = \sum_A \binom{|A|}{2} r/(n-1).\] 
Hence
\[ \E[e^{\rm int}_{\cA} (R)]/m =  \,\sum_A \binom{|A|}{2} / \binom{n}{2}, \]
the proportion of all $\binom{n}{2}$ possible edges which are internal in $\cA$. 
Also
\[
\sum_A \binom{|A|}{2} / \binom{n}{2}= \sum_{A\in \cA} |A|^2/n^2 - \frac{1}{n-1} \Big(1-\sum_{A\in \cA} |A|^2/n^2\Big) = q_\cA^D(G) -  \frac{1}{n-1} (1- q_\cA^D(G)), \]
where $q_\cA^D(G) = \sum_A |A|^2/n^2$ is the usual degree tax (since~$G$ is regular).  Thus, for a regular graph~$G$ with a non-trivial vertex partition $\cA$,
\[
q_\cA^D(G) -\frac{1}{n-1} <  \E[e^{\rm int}_{\cA} (R)]/m < q_\cA^D(G).
\]
(For the `Erd\H{o}s-R\'enyi modularity', we always use $\sum_A |A|^2 / n^2$ as the degree-tax, see later in this section.)
For example, if~$G$ is the 4-cycle $C_4$  ($\fourcycle$ which is also $K_{2,2}$), then the set of labelled simple graphs with the same degree sequence is the set of graphs isomorphic to~$G$ (i.e. $\{\fourcycle$, $\fourcyclea$, $\fourcycleb \}$). Now if $\cA$ is a bipartition with $e^{\rm int}_{\cA}(G) =2$, then
\begin{equation}\label{eq.cycle_q_sq} q_\cA(G) = \frac24 - \frac{4^2+4^2}{8^2} =0 \;\; \mbox{ and } \;\; \sqA(G) = \frac{2}{4} - \frac26= \frac16 , \end{equation}
and both these values are optimal (that is, $\q(G)=0$ and $sq^*(G)= \frac16$).

For a non-regular example, let~$G$ be a 3-edge path ($\pfourb$).  If $\cA$ is the bipartition with $e^{\rm int}_{\cA}(G) =2$, then
\[ q_\cA(G) = \tfrac23 - \tfrac{3^2 + 3^2}{6^2} = \tfrac16 ;\]
and noting that there are exactly two (simple) labelled graphs with the same degrees as~$G$ (i.e.~$\pfourb$,~$\pfourc$), we may see that
\[\sqA(G) = \tfrac23 - \tfrac{\frac12(2+0)}{3}= \tfrac13.\]
Both these values are optimal (that is, $\q(G)=\frac16$ and $\sq(G)= \tq(G) = \frac13$).  

We have been interested in the graphs~$G$ for which $\q(G)=0$: but when is $\sq(G)=0$?
We say that a graph~$G$ is \emph{determined by its degrees} if~$G$ is the unique graph on $V(G)$ with the same vertex degrees as~$G$.
\begin{prop}\label{prop.sq}
For every graph~$G$, 
$\sq(G) =0$ if and only if~$G$ is determined by its degrees.
\end{prop}
\begin{proof}
If there is a unique graph on $V(G)$ with the degrees of~$G$, then $R$ must be $G$; hence $\E[e^{\rm int}_\cA(R)] = e^{\rm int}_\cA(G)$ for each vertex-partition $\cA$, and so $\sq(G)=0$.  
If $G$ is not determined by its degrees, then there must be an edge $vw$ of $G$ such that $vw$ is an edge in $R$ with probability~$<1$.
Let~$\cA$ be the partition of $V(G)$ with one part $\{v,w\}$ and each other part a singleton. Then $e^{\rm int}_\cA(G)=1$ and $\E[e^{\rm int}_\cA(R)]<1$, so $\sq(G) \geq \sqA(G)>0$.
\end{proof}
We can tell easily if a graph $G$ on $[n]$ is determined by its degrees.  This is always the case for $n \leq 3$. 
Let $n \geq 4$ and let $G$ be a graph on $[n]$.
Let $a,b,c,d$ 
be four distinct vertices of $G$; and suppose that $ab$ and $cd$ 
are edges, and $bc$ and $ad$ are non-edges. (Thus $a,b,c,d$
form an `alternating 4-cycle' -- edge, non-edge, edge, non-edge.) The corresponding 2-\emph{switch} is the replacement of edges $ab$ and $cd$ by new edges $bc$ and $ad$. 
By a theorem of Berge (1973) (see for example Theorem~1.3.33 in~\cite{west2001introduction})
the degrees of $G$ determine $G$ if and only if $G$ has no 2-switches.  Observe also that $G$ has a 2-switch if and only if the complementary graph $\overline{G}$ has a 2-switch.
We can say more.
\needspace{4\baselineskip}
\begin{thm} \label{thm.detbydeg}
For every graph $G$ with at least one edge, the following are equivalent.
\begin{description}
\item{(a)}  $G$ is determined by its degrees.
\item{(b)} $G$ has no 2-switch.
\item{(c)} Repeatedly deleting an isolated or universal vertex destroys $G$ (leaving nothing). 
\end{description}
\end{thm}
It is straightforward to prove this result directly, without using Berge's result; and we do so now.
\begin{proof}
(a)$\implies$(b). Performing a 2-switch on $G$ yields a distinct graph $G'$ with the same degrees as $G$, so this part of Berge's result is immediate.

(b)$\implies$(c).
Suppose that $\delta(G), \delta(\overline{G}) \geq 1$.  It suffices to show that $G$ has a 2-switch.
Suppose wlog that $\delta:= \delta(G) \leq \delta(\overline{G})$.  Let vertex $a$ have degree $\delta$ in $G$.  Fix a neighbour $b$ of $a$ in $G$, and let $Z=N_G(a) \setminus \{b\}$. 
Then $b$ has at least $\delta = |Z|+1$ neighbours in $\overline{G}$, so in $\overline{G}$ vertex $b$ has a neighbour $c \not\in Z$.
Note that vertices $a,b,c$ are distinct, $ab \in E(G)$, and $ac,bc \in E(\overline{G})$.
Vertex $c$ has degree in~$G$ at least $\delta = |Z|+1$ so $c$ has a neighbour $d \not\in Z$ in $G$.  Now $d \not\in \{a,b,c\}$ and $ad \in E(\overline{G})$, so we have a 2-switch, as required.

(c)$\implies$(a).  This is trivial if $v(G) \leq 3$.  Let $k \geq 3$ and suppose it holds for all graphs with $k$ vertices.
Let graph $G$ have $v(G) = k+1$, and let graph $H$ on $V(G)$ have the same degrees as $G$.  We must show that $H=G$.

Let vertex $v$ in $G$ be isolated or universal, let $G'=G -v$ and $H'=H-v$.
Consider a vertex $w$ in $V(G')=V(H')$: if $v$ was isolated in $G$ then in $G'$ and $H'$, each vertex $w$ has degree $d_G(w)$; and if $v$ was universal in $G$ then in $G'$ and $H'$, each vertex $w$ has degree $d_G(w)-1$.  Thus $H'$ has the same degrees as $G'$, so $H'=G'$ by the induction hypothesis, and so $H=G$ as required.
\end{proof}
We may read off various results from Proposition~\ref{prop.sq} and Theorem~\ref{thm.detbydeg}.
Recall that $f^{\rm sq}(n)$ is the minimum number of edges that may be deleted from $K_n$ to form a graph $G$ with $s\q(G)>0$.
If~$G$ is the complete graph $K_n$, or~$G$ is $K_n$ less one edge, then~$G$ is determined by its degrees, so $\sq(G)=0$.
If~$G$ is $K_n$ less two disjoint edges, then~$G$ has a 2-switch and so is not determined by its degrees, and $\sq(G)>0$ (for $n\geq 4$). 
Hence
\begin{equation}\label{eq:f_simp_mod}
f^{\rm sq}(n)=2 \;\;\; \mbox{for each }n\geq 4\,. 
\end{equation}
If~$G$ is a complete multipartite  graph with at least $2$ parts of size at least $2$, then again~$G$ has a $2$-switch, so $G$ is not determined by its degrees, and thus $\sq(G) >0$  (whereas $\q(G)=0$ in contrast). If $G$ is $K_n$ less all the edges in a set $A \subset [n]$, or if $G$ is $K_n$ less the edges of a star, then $G$ is determined by its degrees (for example, part (c) of Theorem~\ref{thm.detbydeg} clearly holds) so $\sq(G)=0$.

Finally here, let us observe some features of $\sq$ not shared by $\q$, concerning a graph $G$ and its complement $\overline{G}$.
We saw above that if $G$ is $C_4$ then $\q(G)=0$; and the complementary graph $\overline{G}$ consists of two disjoint edges, so $\q(\overline{G})=\tfrac12$. But clearly
a graph $G$ is determined by its degrees if and only if  this holds for the complementary graph $\overline{G}$. 
Hence, by Proposition~\ref{prop.sq} 
\begin{equation} \label{eqn.sq0}
\sq(G) =0 \;\; \mbox{ if and only if } \;\; \sq(\overline{G})=0.
\end{equation}

For clarity, for $ij \in E(G)$ we shall write $ij \in G$, etc.
We may take $R(\overline{G})$ to be the complement $\overline{R(G)}$ of $R(G)$. 
Thus 
\[ \pr(ij \in R(G)) + \pr(ij \in R(\overline{G})) =1. \]
Hence
\begin{eqnarray*}
    sq_\cA(G) + sq_\cA(\overline{G}) &=&
    \frac1{m} \sum_{i<j} \left( {\mathbf 1}_{ij\in G} - \pr(ij \in R(G)) + {\mathbf 1}_{ij\in \overline{G}} - \pr(ij \in R(\overline{G})) \right) \delta_{ij}(\cA)\\
    &=&
    \frac1{m} \sum_{i<j} \left( 1-1 \right) \delta_{ij}(\cA) \;\; = \;\; 0\,.
\end{eqnarray*}
Thus, for every graph $G$ and vertex-partition $\cA$,
\begin{equation} \label{eqn.sqGbarG}
sq_\cA(G) + sq_\cA(\overline{G})=0.
\end{equation}

Now we may amplify~(\ref{eqn.sq0}) as follows.
\begin{prop}
For every graph $G$, the following four statements are equivalent:
(a) $\sq(G)=0$, (b) $\sq(\overline{G})=0$,
(c) $sq_\cA(G)=0$ for every vertex-partition $\cA$, (d) $sq_\cA(\overline{G})=0$ for every vertex-partition $\cA$.
\end{prop}

\medskip
\needspace{6\baselineskip}

\noindent
{\em Modified modularity $\tq$}\label{subsec.mod_mod}

Modified modularity may be considered to be a `local' version of simple modularity.
It was introduced in~\cite{chang2025modularity}, together with simple modularity.

For modified modularity, we consider the pairs of vertices $i<j$ in $V$ separately, and let
$\tilde{R}_{ij}$ be a random graph on $V$ picked uniformly from the set of all graphs with the same number of edges as~$G$, and with degrees $d_i$ and $d_j$ the same as in $G$. 
Then we set
\begin{equation}\label{defn.mq} \widetilde{q}_\cA(G) = \frac1m \sum_{i<j} \left( {\mathbf 1}_{ij\in E(G)} - \pr\big(ij \in E(\widetilde{R}_{ij})\big) \right) \delta_{ij}(\cA)
\end{equation}
(recall $\delta_{ij}(\cA)$ is the indicator that vertices $i$ and $j$ are in the same part in the partition $\cA$). Also, we let $\tq(G)$ be the maximum value of $\widetilde{q}_\cA(G)$ over all vertex partitions $\cA$, and note that $0 \leq \tq(G) <1$.
Observe that, to determine the `penalty term' for vertices $i,j$ in~(\ref{defn.mq}), we consider a graph chosen uniformly at random from those which \emph{locally} match the degrees of $G$, as well as matching the total number of edges. For a graph $G$ with vertex set $V$ and pair of distinct vertices $i,j$, we denote by $\cG(G,ij)$ the set of labelled simple graphs $G'$ with $V'=V$, $d_i'=d_i$, $d_j'=d_j$ and $m'=m$ (where $V'$, $d_i'$ and~$m'$ denote the vertex-set, degree of vertex $i$, and number of edges in $G'$ respectively), so that~$\tilde{R}_{ij}$ is chosen uniformly from $\cG(G,ij)$.

We again begin with some examples to get a feel for the definition.
First, suppose that $G$ is the 4-cycle $C_4$  ($\fourcycle$ with vertex labels $a,b,c,d$ clockwise from bottom--left).
The set of labelled simple graphs with $d_a=d_b=2$ and $m=4$ edges is $\cG(\fourcycle, ab)= \{\fourcycle$, $\fourcyclea$, $\fourcycleb$, $\pawd$, $\pawe \}$. Thus $\tilde{R}_{ab}$ is chosen uniformly from these five graphs, and the probability that $ab \in E(\widetilde{R}_{ab})$ is $\frac45$; and this is the same for all pairs of vertices by symmetry. Now let 
$\cA$ be the bipartition $\{\{a,b\},\{c,d\}\}$, with edges $ab$ and $cd$ as internal edges. Thus 
\[
\tilde{q}_\cA(G)  = \tfrac{1}{4}\Big( \big(1-\tfrac45\big) + \big(1-\tfrac45\big)\Big) = \tfrac1{10} >0,
\]
and this is optimal, so $\tq(G)=\frac{1}{10}$. (To check optimality, note that including any non--edge within a part would lead to a contribution of $-\frac15$.) Note that the modified modularity score $\widetilde{q}_\cA(G)$ for this partition is different from the scores for both usual modularity $q_\cA(G)=0$, and simple modularity $sq_\cA(G)=1/6$ both given in \eqref{eq.cycle_q_sq}.

We now consider a non-regular example -- see also Fig.~\ref{fig.pfour_and_pawb} (left) on~page~\pageref{fig.pfour_and_pawb}. Let $G$ be the 3-edge path on vertices $V\!=\!\{a,b,c,d\}$ with edges $ab, bc, cd$ (i.e.~$\pfourb$ labelled clockwise from bottom--left). Then~$\widetilde{R}_{ab}$ is chosen uniformly from the set of graphs on $V$ with $d_a=1$, $d_b=2$ and $m=3$. There are exactly four such graphs, i.e.  $\cG(\pfourb, ab)= \{ \pfourb$, $\pfourc$, $\pfourd$, $\pfoure \}$), and so the probability that $ab$ is an edge in $\widetilde{R}_{ab}$ is $\frac12$. Note that by symmetry, this is the same for $\widetilde{R}_{ac}$ (that is, the probability that $ac$ is an edge in $\widetilde{R}_{ac}$ is $\frac12$), for~$\widetilde{R}_{bd}$, and for $\widetilde{R}_{cd}$. 
Similarly, $\widetilde{R}_{bc}$ is chosen uniformly from the set of graphs on $V$ with $d_b=2$, $d_c=2$ and $m=3$ (i.e. $\{\pfourb$, $\pfourc$, $\cofourstar$, $\cofourstarb\}$), and each of these graphs have the edge~$bc$. Lastly $\widetilde{R}_{ad}$ is chosen uniformly from $\{ \pfourb$, $\pfourc \}$, and neither of these graphs have the edge $ad$.

We may now see that $ab$, $cd$ are the only pairs that can make a positive contribution in~(\ref{defn.mq}), and $ac$, $bd$ make a negative contribution; -- see also last column of Fig.~\ref{fig.pfour_and_pawb}. Hence, the bipartition $\cA = \{\{a,b\},\{c,d\}\}$ is the unique optimal partition for $\widetilde{q}$, and 
\begin{equation}\label{eq.pfour_mq}
\tq(G)=\tilde{q}_\cA(G) = \tfrac13\Big( (1 - \tfrac{1}{2}) + (1 - \tfrac{1}{2})\Big) = \tfrac13.\end{equation}
Recall that $q_\cA(G)=\frac16$ and $sq_\cA(G)= \frac13$; and that $\cA$ is an optimal partitions for both standard and simple modularities.

\begin{figure}
\centering
$
\renewcommand{\arraystretch}{1.2}
\begin{array}{||l l| c | c |  c || }
\multicolumn{5}{}{} \mbox{$G=\pfourb$}\\
\multicolumn{5}{}{} \\
\cline{1-5}
i,j && \cG(\pfourb, ij) 
& p_{ij} & w_{ij}\\ 
\cline{1-5} 
&&&&\\
a, b & \pfourLab \; &
\; \pfourb,\; \pfourc, \; \pfourd, \; \pfoure \;&
\tfrac{1}{2} & \;\tfrac{1}{2}\\ 
{\color{gray} c,d} & {\color{gray} \pfourLcdgray} &&
{\color{gray} \tfrac{1}{2}} & {\color{gray} \;\tfrac{1}{2}}\\
&&&&\\
b,c & \pfourLbc &
\pfourb, \;\pfourc, \;\cofourstar, \;\cofourstarb&
1 & \; 0\\
&&&&\\
a,c & \pfourLac & \pfourb, \; \pfourf, \; \pfourc, \; \pfourg&
\tfrac{1}{2}
& -\tfrac{1}{2} \\
{\color{gray} b,d} & \pfourLbdgray &&
{\color{gray} \tfrac{1}{2}} & {\color{gray} -\tfrac{1}{2}}\\
&&&&\\
a,d& \pfourLad & 
\pfourb, \pfourc &  
0 & \;0  \\
&&&&\\
\cline{1-5}				
\end{array}
$
\centering
\hspace{14mm}
$
\renewcommand{\arraystretch}{1.2}
\begin{array}{||l l| c | c |  c || }
\multicolumn{5}{}{} 
\mbox{$G=\pawb$}\\
\multicolumn{5}{}{} \\
\cline{1-5}
i,j && \cG(\pawb, ij) 
&p_{ij} &w_{ij}\\ 
\cline{1-5} 
&&&&\\
a, b & \pawbLab \; &
\pawb &
1 & 0  \\ 
&&&&\\
a,c & \pawbLac & 
\pawb,\; \pawh &
0 & 0  \\
{\color{gray} a,d} & \pawbLadgray &&
{\color{gray} 0} & {\color{gray}0}\\
&&&&\\
b,c & \pawbLbc &
\pawb,\; \paw&
1 & 0 \\
{\color{gray} b,d} & {\color{gray} \pawbLbdgray} &&
{\color{gray} 1} & {\color{gray}0}\\
&&&&\\
c,d& \pawbLcd & 
\; \pawb, \; \pawc, \; \fourcycle, \; \fourcyclea, \; \fourcycleb \;& 
\tfrac{4}{5} &\tfrac{1}{5}  \\
&&&&\\
\cline{1-5}				
\end{array}
$
\caption{Figure showing calculation details for the modified modularity score of the $3$-edge path (left), and the graph consisting of a triangle with a pendant edge (right), where for both graphs the vertex sets are labelled clockwise from bottom--left. 
Here, $p_{ij}=\pr(ij \in E(\widetilde{R}_{ij}))$ and $w_{ij}= {\mathbf 1}_{ij\in E(G)} - \pr(ij \in E(\widetilde{R}_{ij}))$. 
See also the discussions preceding~\eqref{eq.pfour_mq} and~\eqref{eq.pawb_mq}. }
\label{fig.pfour_and_pawb}
\end{figure}
\needspace{6\baselineskip}
When do we have $\tq(G)=0$? We say that the graph $G$ is \emph{locally edge-determined by its degrees} if, for every edge $ij$ in $G$, every graph $G'$ in $\cG(G,ij)$ contains the edge $ij$. {(The set $\cG(G,ij)$ is defined just below~\eqref{defn.mq}.)
The following result shows that we can tell easily if $\tq(G) =0$.  
\needspace{6\baselineskip}
\begin{thm} \label{thm.locdetbydeg}
For every graph $G$ with at least one edge, the following are equivalent.\begin{description} \vspace{-3.5mm}
\item{(a)} $\tq(G) =0$.
\item{(b)} $G$ is locally edge-determined by its degrees.
\item{(c)} Either every edge of $G$ is incident with a universal vertex, or $G$ less any isolated vertices is a triangle or a star.
\end{description}
\end{thm}

From Theorem~\ref{thm.locdetbydeg} we may quickly determine $f^{\tilde{q}}$, and we do so now before proceeding with the proof of the theorem. Let $n \geq 4$. If $G$ is a complete graph $K_n$ or $K_n$ less one edge, then every edge of $G$ is incident to a universal vertex, so $\tq(G)=0$. 
If $G$ is $K_n$ less any two edges, then some edge is not incident to any universal vertex, $G$ has no isolated vertices, and $G$ is not a triangle or a star, and so $\tq(G)>0$ by Theorem~\ref{thm.locdetbydeg}. Hence \begin{equation}\label{eq:f_mod_mod} 
f^{\tilde{q}}(n)=2 \;\;\; \mbox{ for each } n \geq 4.
\end{equation}

\begin{proof}[Proof of Theorem~\ref{thm.locdetbydeg}.]
Let $G$ have $m$ edges.
Note first that for all distinct vertices $i, j$ in $G$ we have $d_i + d_j \leq m+1$ (since otherwise $G$ would have more than $m$ edges).  Also, if $d_i + d_j = m+1$ then~$G$ and every other graph in $\cG(G,ij)$ must contain the edge $ij$.  We shall show that $(b) \implies (a)$, then $(c) \implies (b)$ and finally $(a) \implies (c)$.

$(b) \implies (a)$. 
Let $G$ be locally edge-determined by its degrees. First note that for any graph, the singletons partition which places each vertex in a separate part has modified modularity score $0$,
so it suffices to prove that $\tq(G)\leq 0$.
For each edge $ij$ of $G$, the random graph $\widetilde{R}_{ij}$ must contain the edge $ij$, that is $\pr(ij \in E(\widetilde{R}_{ij}))=1$.  Thus, whatever the partition $\cA$, if there is a corresponding summand in~(\ref{defn.mq}) it is $0$. Also, any summand corresponding to a non-edge is $\leq 0$ and so $\tilde{q}_\cA(G) \leq 0$.  Thus $\tq(G)=0$, as required.

$(c) \implies (b)$.  If an edge $ij$ of $G$ is incident with a universal vertex, then it must be in every graph in $\cG(G,ij)$.  Thus if each edge of $G$ is incident with a universal vertex, then $G$ is locally determined by its degrees.

Now let $G^-$ be $G$ less any isolated vertices, and suppose that $G^-$ is a triangle or a star.
It remains to show that $G$ is locally determined by its degrees.  If $ij$ is an edge in $G$, then $d_i+d_j = m+1$, so each graph in $\cG(G,ij)$ contains the edge $ij$.  Hence, again $G$ is locally determined by its degrees, as required.

It remains to show that $(a) \implies (c)$.
Assume that (a) holds, let $G^-$ be $G$ less any isolated vertices, and assume for a contradiction that $G^-$ is not a triangle or a star, and $ij$ is an edge in $G$ with neither $i$ nor $j$ universal.  Suppose that $d_i + d_j \leq m$. Then there is a graph $G'$ in $\cG(G,ij)$ with no edge $ij$, so the partition $\cA$ with one part $\{i,j\}$ and each other part a singleton has $\tilde{q}_\cA(G)>0$, and thus (a) fails. Thus we must have $d_i + d_j = m+1$. 
  
We may assume wlog that $d_i \geq d_j$.  If $d_j=1$ then $d_i=m$, so $G^-$ is a star, contradicting our assumptions.  Hence $d_j \geq 2$, and there is an edge $jh$ where $h \neq i$. Now $h$ can be adjacent only to $i$ and $j$, so $d_h \leq 2$.  Suppose that $m \geq 4$.  Then $d_j+d_h \leq (m+1)/2 + 2 = m + (5-m)/2 \leq m + 1/2$, and so $d_j+d_h \leq m$.  This gives a contradiction as above, and thus completes the proof of the case $m \geq 4$.
Finally suppose that $m \leq 3$.  If $G^-$ is not a triangle or a star, then either $G^-$  is the 3-edge path (i.e. $\threepathb$) and $\tq(G) = 1/3$ (as we saw above),
or $G$ has edges $uv$ and $u'v'$ in distinct components
(so there is an alternating cycle $C_4$).
In the latter case, let $\cA$ be the vertex-partition with parts $\{u,v\}$ and $\{u',v'\}$ and all other parts singletons.
By considering the graph with edges $uv$ and $u'v'$ replaced by $uu'$ and $vv'$ we see that $\tilde{q}_\cA(G)>0$.
This contradiction completes the proof that $(a) \implies (c)$, and thus completes the entire proof.
\end{proof}

Observe that if a graph $G$ is locally edge-determined by its degrees, then every graph with all the same degrees as $G$ contains all the edges of $G$, and so it is $G$: thus $G$ is determined by its degrees. Hence, by Theorem~\ref{thm.locdetbydeg} and Proposition~\ref{prop.sq}, $\tq(G)=0$ implies $\sq(G)=0$.  
However, there are graphs $G$ with $\sq(G)=0$ but such that $\tq(G)>0$. For example, let the $4$-edge graph $G$ on $\{a,b,c,d\}$ have leaf vertex $a$ adjacent to vertex $b$ in the triangle $b,c,d$ ($\pawb$ labelled clockwise from bottom--left) -- see also Fig.~\ref{fig.pfour_and_pawb} (right).
Observe that $G$ is the unique graph with vertex degrees $1,3,2,2$ (and so $\sq(G)=0$).  But consider vertices $c,d$ and $\widetilde{R}_{cd}$.  Given that there are $m=4$ edges and $d_c=2, d_d=2$, there are exactly five possible graphs, namely the set $\cG(\pawb, cd)=\{\pawb, \; \pawc, \; \fourcycle, \; \fourcyclea, \; \fourcycleb \}$. 
Thus $cd$ is an edge in $\widetilde{R}_{cd}$ with probability $\frac45$. Hence, letting $\cA$ be the partition with parts $\{a\}, \{b\}$, and $\{c,d\}$, we have \begin{equation}\label{eq.pawb_mq}
\tq(G)=\tilde{q}_\cA(G) = (1-\frac45)/4 = \frac{1}{20} >0
\end{equation}
where the optimality follows since $cd$ is the only pair $ij$ for which we have a positive contribution to the score $w_{ij}= {\mathbf 1}_{ij\in E(G)} - \pr(ij \in E(\widetilde{R}_{ij}))$, see Fig~\ref{fig.pfour_and_pawb} (right).

Two undesirable properties of modified modularity $\tilde{q}$ are that the trivial one-part partition can have modularity score $>0$, and that isolated vertices can make a difference.
Both these properties are illustrated by the graph $K_4 +v$ $= \kfourisolb$, which is $K_4$ plus an isolated vertex $v$. For $\tq(K_4)=0$ (for example by Theorem~\ref{thm.locdetbydeg}), but the trivial one-part partition ${\cT}$ of $K_4+v$ has modified modularity score $\tilde{q}_{\cT}(K_4+v) > 0$,
as we see in the next paragraph.

Indeed, let $n \geq 4$, start with the complete graph $K_n$, and remove $t$ edges incident with one vertex~$v$, where $t=2$ if $n=4$ and $2 \leq t \leq n-1$ if $n \geq 5$; and call the resulting graph $G$. (When $n=4$, $G$ is the graph $\pawb$, which we discussed above; and when $n=5$ and $t=4$, $G$ is $K_4 +v$.) Then  the trivial one-part partition $\cT$ gives $\tilde{q}_\cT(G) >0$.  To see this, consider a `missing edge' $vw$ (not in $G$), and note that no graph $\tilde{R}_{vw}$ can contain this edge, since then it would have too few edges. Thus $\tilde{q}_\cT(G) \geq 0$. Further, let the vertex $x$ be such that $vx$ is  also a missing edge (with $x\neq w)$: then some graph $\tilde{R}_{wx}$ in $\cG(G, wx)$ does not contain the edge $wx$, since we can obtain such a graph by deleting the edge $wx$ from $G$, adding edges $vw$ and $vx$, and finally deleting an edge not incident with $w$ or $x$.  This yields a strictly positive contribution to the score, and so $q_\cT(G) > 0$.

To conclude here, we note that (unlike $\sq$) we can have $\tq(G) > 0$ and $\tq(\overline{G}) = 0$.
Again, take $G=\pawb$, and we have $\q(G)=1/20>0$.
Note that $\overline{G}=\copawc$ is a star on $a,c,d$ with centre $a$, and an isolated vertex $b$. Thus by Theorem~\ref{thm.locdetbydeg} we have $\q(\overline{G})=0$.

Observe that $\copawc$ is locally edge-determined, i.e.\ for both edges $ac$ and $ad$ in $\copawc$, the sets 
$\cG(\copawc, ac) =\{ \copawc, \; \copawf \} $ 
and $\cG(\copawc, ad) =\{ \copawc, \; \copawe \}$ contain only graphs which contain the edges $ac$ and $ad$ respectively. 
However, for the non-edge $cd$ in $\copawc$, we have $\cG(\copawc, cd)= \{\copawc, \; \copawd, ,\; \cofourcyclec, \; \cofourcycle, \; \cofourcycleb \}$ (i.e. the complements of the graphs in~$\cG(\pawb, cd)$), which includes
a graph with $cd$ an edge, and a graph with $cd$ a non-edge.
In general, the graphs in $\cG(\overline{G},ij)$ are exactly the complements of the graphs in $\cG(G,ij)$.  Thus for $\widetilde{R}_{ij}(\overline{G}))$ we may take the complement of $\widetilde{R}_{ij}(G)$; and so
\[ \pr\left(ij \in E(\widetilde{R}_{ij}(G))\right) +  \pr\left(ij \in E(\widetilde{R}_{ij}(\overline{G}))\right) =1.\]
Hence, as for simple modularity $sq$ in~(\ref{eqn.sqGbarG}), for any vertex-partition $\cA$ we have
\begin{equation} \label{eqn.modsum}
 \tilde{q}_\cA(G) + \tilde{q}_\cA(\overline{G}) = 0.
\end{equation}

\bigskip
\smallskip
\needspace{6\baselineskip}
{\em Loop-free modularity}

Above, we introduced simple modularity $\sq(G)$ and modified modularity $\tq$, where we `penalised' by considering a random \emph{simple} graph $R$ which is uniformly distributed on the graphs with the same degrees  as $G$, or `locally' the same degrees as $G$; 
and explored when we have $\sq(G)=0$ and when we have $\tq(G)=0$.
What if we insist that the random graph has no loops but it is allowed multiple edges?

Let $R_0$ be a multigraph, with no loops, uniformly distributed on such multigraphs on $V(G)$ with the same degrees as $G$.
For a vertex-partition $\cA$ of~$G$, define the corresponding
`loop-free modularity'~$\lfq_{\cA}(G)$ by
\begin{equation}\label{defn.lfq} \lfqA(G) = (1/m) \left(e^{\rm int}_{\cA} (G) - \E[e^{\rm int}_{\cA} (R_0)] \right)\,.
\end{equation}
Let $\lfq(G)$ be the maximum value of $\lfqA(G)$ over all vertex-partitions $\cA$,
and observe that $0 \leq \lfq(G)<1$.
Let $S$ be chosen uniformly at random over all simple graphs with the same degrees as~$G$.} %
When $G$ is a regular graph with $m$ edges,  the expected number of edges between any distinct pair of vertices is $m/\binom{n}{2}$ for both the simple graph $S$ and the loop-free graph $R_0$. Hence for regular graphs $G$, we have $\E[e_{\cA}^{\rm int}(R_0)] = \E[e_{\cA}^{\rm int}(S)]$,
and so $\lfq(G)=sq^*(G)$.
We saw in Proposition~\ref{prop.sq} exactly which graphs $G$ have $\sq(G)=0$: a small subset of these graphs have~$\lfq(G)=0$.
\begin{prop} \label{prop.lfq}
For any graph $G$, $\lfq(G)=0$ if and only if $G$ less any isolated vertices is a star or a complete graph.
\end{prop}
\begin{proof}
Sufficiency is easy. Let $G^-$ denote $G$ less any isolated vertices. If $G^-$ is a star then $R_0$ must be~$G$, so  $\lfqA(G)=0$ for each vertex-partition $\cA$, and $\lfq(G)=0$.  If $G^-$ is complete and $e \in E(G^-)$, then by symmetry the expected number of times $e$ appears in $R_0$ is 1.  Thus $\lfqA(G)=0$ for each vertex-partition $\cA$, and $\lfq(G)=0$.

To prove necessity we first introduce a weakened version of a 2-switch.  
If $a,b,c,d$ are distinct vertices in a graph $G$ and $ab, cd \in E(G)$ and $bc \in E(\overline{G})$ we call this a $2^{-}$-switch (or an `alternating~$P_3$').  (There is no restriction on $ad$.)  Observe that a 2-switch is a $2^{-}$-switch but not necessarily conversely.
Now we may argue as in the proof of Proposition~\ref{prop.sq}.
Suppose that $G$ has a $2^{-}$-switch on $a,b,c,d$.  Then we may form a multigraph $G'$ with the same degrees as $G$ by deleting the edges $ab, cd$ and adding~1 to the number of copies of $bc$ and $ad$. Thus the probability that $bc$ is an edge in~$R_0$ is $>0$, so there must be an edge $vw$ of $G$ such that $\E[e_{R_0}(vw)]<1$. Let $\cA$ be the partition of $V(G)$ with one part $\{v,w\}$ and each other part a singleton. Then $e^{\rm int}_\cA(G)=1$ and $\E[e^{\rm int}_\cA(R_0)]<1$, so $\lfq(G) \geq \lfqA(G)>0$.

Now suppose that $\lfq(G)=0$, and so $G$ has no $2^{-}$-switch.   Again, let $G^-$ denote $G$ less any isolated vertices. We may assume that $G^-$ has at least 4 vertices, and is not a star. Thus there are two disjoint edges; and since there is no $2^{-}$-switch, any two disjoint edges are part of a complete graph $K_4$. In particular $G^-$ is connected.
Let $W$ with $|W| \geq 3$ be a maximal set of vertices in $G^-$ which induce a complete graph.  If some vertex is not in $W$ then there is an edge $ab$ with $a \not\in W$ and $b \in W$.  Let $c \in W$, $c \neq b$.  Then $ac \in E( G^-)$ since there exists $d \in W \setminus \{b,c\}$ and $a,b,c,d$ do not form a $2^{-}$-switch.  Thus vertex $a$ is adjacent to each vertex in $W$, contradicting the maximality of $W$.  Hence $G^-$ is complete, and we are done.
\end{proof}
Recall that $f^{\rm lfq}(n)$ is the minimum number of edges that may be deleted from $K_n$ to form a graph $G$ with $\lfq(G)>0$.
An easy corollary of Proposition~\ref{prop.lfq} is that
\begin{equation}\label{eq:f_loop_free_mod} 
f^{\rm lfq}(n)=1 \;\;\;
\mbox{for each } n\geq 4\,.
\end{equation}
\smallskip

Consider an $m$-edge graph $G$, with a given vertex-partition. In each of the above four variant forms of modularity -- namely (precise) configuration modularity $cq$, simple modularity $sq$, modified modularity $\widetilde{q}$, and loop-free modularity $l\!f q$ --  the score is penalised by using the expected behaviour of a suitable random graph or pseudograph. 
We specify a set $\cH$ of pseudographs on $V$, each with the same number $m$ of edges as $G$, and also with the same vertex-degrees as $G$ (except that, for modified modularity, we have only `locally' the same degrees). Then we let the random pseudograph~$R$ be uniformly  distributed on $\cH$ (for $sq$, $\widetilde{q}$ and $l\!f q$), or at least nearly uniformly distributed on $\cH$ (for $cq$, assuming suitable vertex degrees).
For any partition $\cA$ of $V$, the corresponding modularity score is defined to be
\begin{equation}
\label{defn.R_modularity}
q^{R}_{\cA}(G) = \tfrac1{m}( e^{\rm int}_{\cA}(G) - \E[e^{\rm int}_\cA(R)])\,. 
\end{equation}
The fifth and final modularity variant we consider here is quite different, with a simple form of degree tax.
\smallskip

\needspace{6\baselineskip}
{\em  Erd\H{o}s-R\'enyi modularity}

This form of modularity has been studied and used for several years, see for example~\cite{gacksgens2023hyperspherical, zhao2012consistency}.
For a graph $G$ and vertex-partition $\cA$, the Erd\H{o}s-R\'enyi modularity score is defined to be
\begin{equation}\label{defn.ERq}  ERq_{\cA}(G) =  e^{\rm int}_{\cA}(G)/m - \sum_A |A|^2/n^2 ;\end{equation}
and the Erd\H{o}s-R\'enyi modularity $ERq^*(G)$ is the maximum value of this score over all vertex-partitions $\cA$. Observe that $0 \leq ERq^*(G) <1$.
If $G$ is regular then clearly $ERq_{\cA}(G) = q_{\cA}(G)$ for each vertex-partition $\cA$, and $ERq^*(G) = \q(G)$.   In particular we see that $ ERq^*(K_n) =0$, and if $n$ is even and $G$ is obtained from $K_n$ by deleting a perfect matching then $ERq^*(G) =0$.

However, we will see later, in Proposition~\ref{prop.det_ER_mod0}, that it is enough to delete just two edges from $K_n$ to obtain a graph with positive Erd\H{o}s-R\'enyi modularity. For an easy example, let $n=3$ and let $G$ be obtained by deleting two edges from $K_3$, so $G$ consists of an edge and isolated vertex, $G = \, \cocherry$; and let $\cA$ be the bipartition which takes the adjacent vertices as one part and the isolated vertex as the other.
Then $ERq_\cA (G) = 1 - (1/3)^2 - (2/3)^2 = 4/9$. Note that this last example shows that adding an isolated vertex can change the Erd\H{o}s-R\'enyi modularity, since for $G'$ a single edge $\edge$ we have $ERq^*(G')=0$.

\bigskip

The standard modularity score $q_\cA(G)$ and the Erd\H{o}s-R\'enyi modularity score $ERq_\cA(G)$ are both examples of a `modularity score with weighted degree tax', or a `modularity score penalised by weighted part sizes'. Let the graph $G$ have vertex set $V=[n]$, and let ${\bf c} = (c_1,\ldots,c_n)$ be given vertex weights $c_j>0$.  For $A \subseteq V$ let $c(A)= \sum_{j \in A} c_j$. Given a partition $\cA$ of $V$, the corresponding `weighted tax on parts' or `${\bf c}$-penalty' is $\sum_{A \in \cA} c(A)^2 / c(V)^2$; and the corresponding modularity score is \[ q^{\bf c}_{\cA}(G) = \tfrac1{m} e^{\rm int}_\cA(G) - \sum_{A \in \cA} c(A)^2 / c(V)^2.\]
Observe that, if each $c_j$ is the degree of vertex $j$ then $q^{\bf c}_{\cA}(G)$ is the standard modularity score $q_\cA(G)$; and if each $c_j=1$ then $q^{\bf c}_{\cA}(G)$ is the Erd\H{o}s-R\'enyi modularity score $ERq_\cA(G)$.  Thus the versions of modularity we have considered above (including the standard modularity) are particular cases of either an $R$-penalised modularity or a ${\bf c}$-penalised modularity.
\smallskip

\subsection{Edit distance from complete graphs to a positive modularity graph}
\label{subsec.edit}
Recall that $f^{\rm cq}(n)$ is the least number of edges missing from any $n$-vertex graph $G$ with $cq^*(G)>0$ (that is, with positive precise configuration modularity); and similarly for $f^{\mathrm{ERq}}(n)$.
In this section we determine $f^{\rm cq}$ and $f^{\rm ERq}$.
Note that analogous results for simple modularity $sq$, modified modularity $\tilde{q}$, and loop-free modularity $l\!fq$ have already been established, since they followed directly from other results in the previous subsection, see~\eqref{eq:f_simp_mod}, \eqref{eq:f_mod_mod} and~\eqref{eq:f_loop_free_mod}. 

\medskip
\needspace{6\baselineskip}
{\em  Configuration modularity $cq$}

\begin{prop}\label{prop.det_config_mod0} 
For each $n \geq 4$ we have $f^{\rm cq}(n)=\lceil n/2 \rceil$.
\end{prop}

We first prove a lemma analogous to Lemma~\ref{lem.pZero}.  Let $G=(V,E)$ be a graph. Similarly to $p_G(U)$, for $U \subseteq V$, let 
\[\tp_G(U) = -e(U, \overline{U})(\vol(G)-1) + \vol(U)\vol(\overline{U}).\] 
Observe that $\tp_G(\emptyset) = \tp_G(V) =0$ (as for $p_G$).
\begin{lemma}\label{lem.pcZero} 
Let $G$ be a graph with vertex set $V$ and with $m\geq 1$ edges. Then the following three statements are equivalent: 
(i) $\cq(G)=0$,
(ii) $cq_\cA(G) \leq 0$ for each bipartition $\cA$ of $V$, and 
(iii) $\tp_G(U) \leq 0$ for each $U \subseteq  V$.
\end{lemma}

\begin{proof}
Observe that for any $A \subseteq V$, $\vol(A)-1 = 2m-1-\vol(\overline{A})$. Thus for any partition $\cA$ of $V$,
\begin{eqnarray*}
    cq_\cA(G) & = & 1 - \frac{1}{2m}\sum_A e(A, \overline{A}) - \frac{1}{2m(2m\! -\! 1)}\sum_A \vol(A)(2m-1-\vol(\overline{A})) \\
    & = & - \frac{1}{2m}\sum_A e(A, \overline{A})\; + \; \frac{1}{2m(2m\! -\! 1)}\sum_A \vol(A)\vol(\overline{A})\,; \\
    \end{eqnarray*} 
and so, 
writing $\nu=\vol(G) \:(=2m)$, 
\[cq_{\cA}(G) = \nu^{-1}(\nu-1)^{-1}  \sum_{A \in \cA} \tp_G(A).\] 
But $\tp_G(U)$ is symmetric in $U$ and~$\overline{U}$, so $\tp_G(U) = \tp_G(\overline{U})$, and the lemma follows.
\end{proof}

\begin{proof}[Proof of Proposition~\ref{prop.det_config_mod0}]
To prove the lower bound we must show that if $n \geq 4$ and we form $G$ by removing $x < \lceil n/2 \rceil$ edges from $K_n$ then $c\q(G)=0$. 
Note that $x \leq (n-1)/2$, and so $G$ has minimum degree $\delta_G \geq  (n-1)/2$, and in particular $\delta_G\geq 2$ since $n\geq 4$.

We first note an inequality (corresponding roughly to~\eqref{eqn.mono}).
Let $\emptyset \neq U \subset V$, and note that there is an edge {$e$} in $G$ between $U$ and $\overline{U}$ (since $n-1-x \geq (n-1)/2>0$).
Suppose that there is a missing edge within $U$ or $\overline{U}$, and form $G'$ by `moving' the edge $e$ to the location of such a missing edge.
Then
\begin{equation} \label{eqn.config_mono}
\tp_{G'}(U) > \tp_G(U).
\end{equation}
To prove this, we may assume that the edge $e$ is added to $U$.
Since $e_{G'}(U,\overline{U})=e_{G}(U, \overline{U})-1$, $\vol_{G'}(U)=\vol_G(U)\!+\!1$ and $\vol_{G'}(\overline{U})=\vol_G(\overline{U})\!-\!1$, we have
\begin{eqnarray*}
    \tp_{G'}(U)- \tp_{G}(U)
    & = & (2m-1)(e_G(U, \overline{U}) - e_{G'}(U, \overline{U})) \\
    && + (\vol_G(U)\!+\!1)(\vol_G(\overline{U})\!-\!1) - \vol_G(U)\vol_G(\overline{U})\\
    & = & 2m-1 + \vol_G(\overline{U}) - \vol_G(U) \; -1 .
\end{eqnarray*}
But $\vol_G(\overline{U}) \geq \delta_G \geq 2$, and so $\vol_G(U)\leq 2m-2$
Hence
\[  \tp_{G'}(U)- \tp_{G}(U) \geq 2
>0, \]
and~(\ref{eqn.config_mono}) follows.

As in the proof of Theorem~\ref{thm.detmod0}, suppose that the graph $H$ on $V=[n]$ and $\emptyset \neq U \subset V$ are such that $\tp_H(U)$ maximises the value $\tp_G(W)$ over all graphs $G$ on $V$ with at most $(n-1)/2$ edges missing and all $\emptyset \neq W \subset V$. There must be edges between $U$ and $\overline{U}$; so by inequality\,(\ref{eqn.config_mono}), all missing edges are between $U$ and $\overline{U}$.

Suppose that $|U|=j$ (where $1 \leq j \leq n-1$), each missing edge in $H$ is between $U$ and $\overline{U}$, and $x$ edges are missing from $H$, where $0 \leq x < \lceil n/2 \rceil$. 
Then
\begin{eqnarray*}
\tp_H(U) &=&
\vol_H(U) \vol_H(\overline{U})  - e_H(U, \overline{U}) (\vol(H)-1) \\
&=&
(j(n-1)-x)((n-j)(n-1)-x) - (j(n-j)-x)(n(n-1)-2x-1)\\
&=&
j(n-j)(2x-(n-2)) -x^2-x.
\end{eqnarray*}
Thus $\tp_H(U)<0$ if $x \leq (n-2)/2$. Suppose that $n$ is odd, and consider the case that $x=(n-1)/2$. Then $\tp_H(U)=  - (n-2j+1)(n-2j-1)/4$; and so for all integer $j$, $\tp_H(U)\leq 0$ with equality when $j=(n-1)/2$ or $j=(n+1)/2$. 
Hence for $x < \lceil n/2 \rceil$, we have $\tp_H(U)\leq 0$, which by Lemma~\ref{lem.pcZero} proves the lower bound.

We now prove the upper bound. For odd $n$, we may take $x=(n+1)/2$ and $j=(n-1)/2$, and then $\, \tp_H(U)=(n+1)(n-3)/2$. Thus for odd $n\geq 4$, there is a graph $G$ with $(n+1)/2$ edges missing and $c\q(G)>0$.
Similarly, for even $n\geq 4$, we may take $x=j=n/2$, and then $\tp_H(U) = n(n-2)/4$.  Thus, for even $n \geq 4$, there is a graph $G$ with $n/2$ edges missing and $c\q(G)>0$.
\end{proof}

\medskip

\needspace{6\baselineskip}
{\em  Erd\H{o}s-R\'enyi modularity $ ERq^*$}
\smallskip

\begin{prop}\label{prop.det_ER_mod0} 
For each $n \geq 3$ we have $f^{\rm ERq}(n)=2$.
\end{prop}

As for configuration modularity, we first prove a lemma analogous to Lemma~\ref{lem.pZero}. For a graph $G = (V,E)$ with $n$ vertices and $m$ edges, for $U \subseteq V$, let 
\[r_G(U) = -e(U, \overline{U})\,n^2 + |U||\overline{U}|\,2m.\] 
Observe that $r_G(\emptyset) = r_G(V) =0$ (as for $p_{G}$ and $\tp_{G}$).

\begin{lemma}\label{lem.pER_Zero} 
Let $G$ be a graph with $m\geq 1$ edges. Then the following three statements are equivalent: (i) $ER\q(G)=0$, (ii) $ERq_\cA(G) \leq 0$ for each bipartition $\cA$ of $V(G)$, and (iii) $r_G(U) \leq 0$ for each $U \subseteq V(G)$.
\end{lemma}

\begin{proof}
Begin by noting that for any vertex-partition $\cA$, 
\begin{eqnarray*}
    ERq_\cA(G) & = & 1 - \frac{1}{2m}\sum_{A \in \cA} e(A, \overline{A}) - \frac{1}{n^2}\sum_{A \in \cA} |A|(n-|\overline{A}|) \\
    & = & - \frac{1}{2m}\sum_{A \in \cA} e(A, \overline{A}) \, +\, \frac{1}{n^2}\sum_{A \in \cA} |A||\overline{A}|\\
    & = & \frac{1}{2mn^2} \sum_{A \in \cA} r_G(A).
    \end{eqnarray*} 
Since $r_G(U)$ is symmetric in $U$ and~$\overline{U}$, we have $r_G(U) = r_G(\overline{U})$,
and the lemma follows.
\end{proof}

\begin{proof}[Proof of Proposition~\ref{prop.det_ER_mod0}]
To prove the lower bound we must show that, if $n \geq 3$ and we form $G$ by removing one or zero edges from $K_n$, then $ERq^*(G)=0$. 

Again, we note an inequality (corresponding to~\eqref{eqn.mono}).
Let $G$ be a graph on vertex set $V$, let $\emptyset \neq U \subset V$, and form $G'$ by `moving' an edge currently between $U$ and $\overline{U}$ to the location of a missing edge within $U$ or $\overline{U}$.  Then (noting that the sizes of $U$ and $\overline{U}$ do not change),
\begin{eqnarray*}
    r_{G'}(U)
    & = & - e_{G'}(U, \overline{U}) n^2 + |U| |\overline{U}| \, 2e(G') \\
    & = & - (e_{G}(U, \overline{U})-1) n^2 + |U| |\overline{U}| \, 2e(G) ,
\end{eqnarray*}
and so
\begin{equation} \label{eqn.ER_mono} 
     r_{G'}(U) >  r_{G}(U).
\end{equation}
Again, as in the proof of Theorem~\ref{thm.detmod0}, suppose that the graph $H$ on $V=[n]$ and $\emptyset \neq U \subset V$ are such that $r_H(U)$ maximises the value $r_G(W)$ over all graphs $G$ on $V$ with at most one edge missing.
Note that if there is a missing edge then by~\eqref{eqn.ER_mono} it is between $U$ and $\overline{U}$.

Suppose that $|U|=j$ (where $1 \leq j \leq n-1$) and $x$ edges are missing from $H$, where $x\in\{0,1\}$. Then
\begin{eqnarray*}
r_H(U) &=& - e_H(U, \overline{U})\, n^2
+ |U| |\overline{U}|\, 2m   \\
&=&
-\big(j(n-j) -x\big)\, n^2 + j(n-j)\big( n(n-1) -2x \big)
\\
&=&
-j(n-j)(n +2x) + n^2x. 
\end{eqnarray*} 
Clearly $r_H(U)<0$ for $x=0$\,; and for $x=1$, noting that $j(n-j) \geq n-1$, 
we have 
\[ r_H(U) = - j(n-j)(n+2) + n^2 \leq -(n-1)(n+2) +n^2 = -(n-2) <0 \]
since $n \geq 3$. This proves the lower bound, that $f^{\rm ERq}(n) \geq 2$.

To see the upper bound, let $n \geq 3$ and consider the graph $G$ formed from $K_n$ by picking a vertex $v$ and deleting two of the incident edges.  Let $\cA$ be the bipartition with one part $\{v\}$, and note that, since $x=2$ and $j=|\{v\}|=1$, we have 
\[ r_G(\{v\}) = -(n-1)(n+4) +2n^2 = n^2-3n+4 = (n-1)(n-2)+2 >0;\]
and hence by Lemma~\ref{lem.pER_Zero}, the bipartition $\cA$ has positive Erd\H{o}s-R\'enyi modularity. 
\end{proof}


\section{Concluding remarks}\label{sec.concl}
We considered graphs with modularity zero, with starting point the complete graphs and complete bipartite graphs. For each such complete graph, we found the least number of edges which we may add to or delete from or edit to obtain a graph with non-zero modularity.

Complete multipartite graphs also have modularity zero but some corresponding questions for them are still open - see Section~\ref{subsec.multi}. 
It is also open, see Question~\ref{q.far}, whether complete graphs are the furthest (in edit distance) from graphs with positive modularity.

We also investigated the modularity of very dense random graphs, and in particular we found that there is a transition to modularity zero when the average degree of the complementary graph drops below 1.

Finally, in Section~\ref{sec.mod_var}, we considered some natural variant definitions of modularity, found the distance from a complete graph to one with positive modularity value for each variant, and for simple modularity~$\sq(G)$, modified modularity~$\tq(G)$, and loop-free modularity $\lfq(G)$, we found precisely which graphs have corresponding modularity zero. See also~\cite{reimaginingdegreetax}, where we investigate further properties of these variants, for example when we limit the number of parts.

\medskip

{\bf Acknowledgements.}
We would like to thank the referees for helpful and interesting comments.  One such comment was that the standard definition of modularity penalised partitions with small parts, since the degree tax was based on random graphs which may have loops. This led to the study in Section~\ref{sec.mod_var} of some variant definitions, and when the corresponding modularity value is $0$. %

FS was partially supported by the Wallenberg AI, Autonomous Systems and Software Program WASP. This material is partially based upon work supported by the National Science Foundation under Grant No.\ DMS-1928930, while FS was in residence at the Simons Laufer Mathematical Sciences Institute (MSRI) in Berkeley, California, during the Spring 2025 semester.

\bibliographystyle{plain}
\bibliography{articles}

\begin{thebibliography}{10}

\bibitem{mod2023universal}
V.\ Agdur, N.\ Kam{\v{c}}ev, and F.\ Skerman.
\newblock Universal lower bound for community structure of sparse graphs.
\newblock {\em arXiv preprint arXiv:2307.07271}, 2023.

\bibitem{bhamidi2026stochastic}
S.\ Bhamidi, D.\ Gamarnik, R.\ van~der Hofstad, N.\ Litvak, P.\ Pra{\l}at, F.\
  Skerman, and Y.\ Tousinejad.
\newblock The stochastic block model has the overlap gap property for
  modularity.
\newblock In {\em 53rd International Colloquium on Automata, Languages, and
  Programming (ICALP 2026)}. Schloss Dagstuhl--Leibniz-Zentrum f{\"u}r
  Informatik, 2026.

\bibitem{bickel2009nonparametric}
P.~J.\ Bickel and A.\ Chen.
\newblock A nonparametric view of network models and {N}ewman--{G}irvan and
  other modularities.
\newblock {\em Proceedings of the National Academy of Sciences}, 106(50), 2009.

\bibitem{bickel2015correction}
P.~J.\ Bickel, A.\ Chen, Y.\ Zhao, E.\ Levina, and J.\ Zhu.
\newblock Correction to the proof of consistency of community detection.
\newblock {\em The Annals of Statistics}, 2015.

\bibitem{louvain}
V.~Blondel, J.~Guillaume, R.\ Lambiotte, and E.\ Lefebvre.
\newblock Fast unfolding of communities in large networks.
\newblock {\em Journal of Statistical Mechanics: Theory and Experiment},
  2008(10):P10008, 2008.

\bibitem{bolla2015spectral}
M.\ Bolla, B.\ Bullins, S.\ Chaturapruek, S.\ Chen, and K.\ Friedl.
\newblock Spectral properties of modularity matrices.
\newblock {\em Linear Algebra and Its Applications}, 473:359--376, 2015.

\bibitem{bollobas2001random}
B.~Bollob{\'a}s.
\newblock {\em Random graphs}.
\newblock Cambridge Studies in Advanced Mathematics (Book 73), 2nd edition.
  Cambridge University Press, Cambridge, 2001.

\bibitem{nphard}
U.\ Brandes, D.\ Delling, M.\ Gaertler, R.\ Gorke, M.\ Hoefer, Z.\ Nikoloski,
  and D.\ Wagner.
\newblock On modularity clustering.
\newblock {\em Knowledge and Data Engineering, IEEE Transactions on},
  20(2):172--188, 2008.

\bibitem{chang2025modularity}
B.~L.\ Chang and P.\ Van~Mieghem.
\newblock Modularity with a more accurate baseline model.
\newblock {\em Physical Review E}, 111(4):044317, 2025.

\bibitem{chellig2022modularity}
J.\ Chellig, N.\ Fountoulakis, and F.\ Skerman.
\newblock The modularity of random graphs on the hyperbolic plane.
\newblock {\em Journal of Complex Networks}, 10(1):cnab051, 2022.

\bibitem{dinh2011finding}
T.~N.\ Dinh and M.~T.\ Thai.
\newblock Finding community structure with performance guarantees in scale-free
  networks.
\newblock In {\em Privacy, Security, Risk and Trust (PASSAT) and 2011 IEEE
  Third Inernational Conference on Social Computing (SocialCom), 2011 IEEE
  Third International Conference on}, pages 888--891. IEEE, 2011.

\bibitem{frieze2015book}
A.\ Frieze and M.\ Karo{\'n}ski.
\newblock {\em Introduction to random graphs}.
\newblock Cambridge University Press, 2015.

\bibitem{gacksgens2023hyperspherical}
M.\ Gösgens, R.\ van~der Hofstad, and N.\ Litvak.
\newblock The hyperspherical geometry of community detection: modularity as a
  distance.
\newblock {\em Journal of Machine Learning Research}, 24(112):1--36, 2023.

\bibitem{JLRbook}
S.\ Janson, T.\ {\L}uczak, and A.\ Ruci{\'n}ski.
\newblock {\em Random {G}raphs}, volume~45.
\newblock John Wiley \& Sons, 2011.

\bibitem{kaminski2024modularity}
B.\ Kami{\'n}ski, P.\ Misiorek, P.\ Pra{\l}at, and F.\ Th{\'e}berge.
\newblock Modularity based community detection in hypergraphs.
\newblock {\em Journal of Complex Networks}, 12(5):cnae041, 2024.

\bibitem{kaminski2022modularity}
B.\ Kami{\'n}ski, B.\ Pankratz, P.\ Pra{\l}at, and F.\ Th{\'e}berge.
\newblock Modularity of the abcd random graph model with community structure.
\newblock {\em Journal of Complex Networks}, 10(6):cnac050, 2022.

\bibitem{popular}
A.\ Lancichinetti and S.\ Fortunato.
\newblock Limits of modularity maximization in community detection.
\newblock {\em Physical Review E}, 84(6):066122, 2011.

\bibitem{lasonsulkowska2023modularity}
M.\ Laso{\'n} and M.\ Sulkowska.
\newblock Modularity of minor-free graphs.
\newblock {\em Journal of Graph Theory}, 102(4):728--736, 2023.

\bibitem{lichev2022modularity}
L.\ Lichev and D.\ Mitsche.
\newblock On the modularity of 3-regular random graphs and random graphs with
  given degree sequences.
\newblock {\em Random Structures \& Algorithms}, 61(4), 2022.

\bibitem{modexp}
B.\ Louf, C.\ McDiarmid, and F.\ Skerman.
\newblock Modularity and graph expansion.
\newblock In {\em 15th Innovations in Theoretical Computer Science Conference
  (ITCS 2024)}. Schloss-Dagstuhl-Leibniz Zentrum f{\"u}r Informatik, 2024.

\bibitem{luczakcomponent}
T.\ {\L}uczak.
\newblock Component behavior near the critical point of the random graph
  process.
\newblock {\em Random Structures \& Algorithms}, 1(3):287--310, 1990.

\bibitem{luczak1990equivalence}
T.\ {\L}uczak.
\newblock On the equivalence of two basic models of random graph.
\newblock In {\em Proceedings of Random graphs}, volume~87, pages 151--159,
  1990.

\bibitem{majstorovic2014note}
S.\ Majstorovic and D.\ Stevanovic.
\newblock A note on graphs whose largest eigenvalues of the modularity matrix
  equals zero.
\newblock {\em Electronic Journal of Linear Algebra}, 27(1):256, 2014.

\bibitem{chapter}
C.\ Mc{D}iarmid and F.\ Skerman.
\newblock Modularity and random graphs.
\newblock In {\em In Topics in Probabilistic Graph Theory, edited by R. J.
  Wilson, L. W. Beineke, and C. McDiarmid. Cambridge University Press.}

\bibitem{treelike}
C.\ Mc{D}iarmid and F.\ Skerman.
\newblock Modularity of regular and treelike graphs.
\newblock {\em Journal of Complex Networks}, 6(4), 2018.

\bibitem{ERmod}
C.\ Mc{D}iarmid and F.\ Skerman.
\newblock Modularity of {E}rd{\H{o}}s-{R}{\'e}nyi random graphs.
\newblock {\em Random Structures \& Algorithms}, 57(1):211--243, 2020.

\bibitem{sampling}
C.\ Mc{D}iarmid and F.\ Skerman.
\newblock Modularity and partially observed graphs.
\newblock {\em arXiv preprint arXiv:2112.13190}, 2021.

\bibitem{reimaginingdegreetax}
C.\ Mc{D}iarmid and F.\ Skerman.
\newblock Redefining the degree tax in modularity.
\newblock {\em In preparation}, 2026+.

\bibitem{modgraphclasses}
F.~de Montgolfier, M.\ Soto, and L.\ Viennot.
\newblock Asymptotic modularity of some graph classes.
\newblock In {\em Algorithms and Computation}, pages 435--444. Springer, 2011.

\bibitem{NewmanBook}
M.~E.~J.\ Newman.
\newblock {\em Networks: {A}n {I}ntroduction}.
\newblock Oxford University Press, 2010.

\bibitem{NewmanGirvan}
M.~E.~J.\ {N}ewman and M.\ Girvan.
\newblock Finding and evaluating community structure in networks.
\newblock {\em Physical Review E}, 69(2):026113, 2004.

\bibitem{poda2024comparison}
V.\ Poda and C.\ Matias.
\newblock Comparison of modularity-based approaches for nodes clustering in
  hypergraphs.
\newblock {\em Peer Community Journal}, 4, 2024.

\bibitem{prokhorenkova2017modularity}
L.~O.\ Prokhorenkova, A.\ Raigorodskii, and P.\ Pra{\l}at.
\newblock Modularity of complex networks models.
\newblock {\em Internet Mathematics}, 2017.

\bibitem{rybarczyksulkowska2025PA}
K.\ Rybarczyk and M.\ Sulkowska.
\newblock Modularity of preferential attachment graphs.
\newblock In {\em Theoretical Aspects of Computer Science (STACS), 2026 43rd
  International Symposium on}, 2026.

\bibitem{thesis}
F.\ Skerman.
\newblock {\em Modularity of Networks}.
\newblock PhD thesis, University of Oxford, 2016.

\bibitem{traag2019louvain}
V.\ Traag, L.\ Waltman, and N.\ Van~Eck.
\newblock From {L}ouvain to {L}eiden: guaranteeing well-connected communities.
\newblock {\em Scientific reports}, 9(1):1--12, 2019.

\bibitem{trajanovski2012maximum}
S.\ Trajanovski, H.\ Wang, and P.\ Van~Mieghem.
\newblock Maximum modular graphs.
\newblock {\em The European Physical Journal B-Condensed Matter and Complex
  Systems}, 85(7):1--14, 2012.

\bibitem{west2001introduction}
D.\ West.
\newblock {\em Introduction to graph theory}, volume~2.
\newblock Prentice Hall Upper Saddle River, 2001.

\bibitem{zhao2012consistency}
Y.\ Zhao, E.\ Levina, and J.\ Zhu.
\newblock Consistency of community detection in networks under degree-corrected
  stochastic block models.
\newblock {\em The Annals of Statistics}, 40(4):2266, 2012.

\end{thebibliography}
\end{document}